\pdfoutput=1
\documentclass[a4paper,10pt]{amsart}
\usepackage{latexsym}
\usepackage[thinlines]{easytable}
\usepackage{amsmath,mathtools}
\usepackage{amsfonts}
\usepackage{amssymb}
\usepackage{mathrsfs}
\usepackage{hyperref}
\usepackage[retainorgcmds]{IEEEtrantools}
\usepackage{enumerate}
\usepackage{enumitem}
\usepackage{amsthm}
\usepackage{pgf}
\usepackage{fullpage}
\usepackage{graphicx}
\usepackage{caption}
\usepackage{color}
\usepackage{float}
\usepackage{qtree}
\usepackage{subfigure}
\usepackage[final]{microtype}
\usepackage{graphicx}
\usepackage{verbatim}
\usepackage{tikz}
\usepackage{tikz-qtree}
\usetikzlibrary{arrows,automata,positioning}

\theoremstyle{definition} \newtheorem{Definition}{Definition}[section]
\theoremstyle{plain} \newtheorem{Theorem}[Definition]{Theorem}
\theoremstyle{plain} \newtheorem{corollary}[Definition]{Corollary}
\theoremstyle{plain} \newtheorem{lemma}[Definition]{Lemma}
\theoremstyle{definition} \newtheorem{Remark}[Definition]{Remark}
\theoremstyle{plain}
\newtheorem{proposition}[Definition]{Propostion}
\theoremstyle{plain} \newtheorem{claim}[Definition]{Claim}
\theoremstyle{plain} 
\theoremstyle{definition}
\newtheorem{example}[Definition]{Example}

\newcommand{\im}{\mbox{ im}}
\newcommand{\Skip}[3]{#1#2 \ldots#2 #3}

\newcommand{\calscr}[1]{\mathcal{#1}}

\newcommand{\gen}[1]{\langle #1 \rangle}

\newcommand{\T}[1]{\mathcal{#1}}

\newcommand{\id}{\mbox{id}}

\newcommand{\pn}[1]{\mathcal{P}_{#1}}
\newcommand{\spn}[1]{\widetilde{\T{P}}_{#1}}
\newcommand{\hn}[1]{\mathcal{H}_{#1}}
\newcommand{\shn}[1]{\widetilde{\T{H}}_{#1}}
\newcommand{\wequal}{=_{\omega}}
\newcommand{\nwequal}{\ne_{\omega}}

\newcommand{\core}{\mathrm{Core}}

\newcommand{\Z}{\mathbb{Z}}
\newcommand{\N}{\mathbb{N}}
\newcommand{\sym}{\mathmbox{Sym}}

\newcommand{\dual}[1]{#1^{\vee}}

\DeclarePairedDelimiter{\floor}{\lfloor}{\rfloor}

\makeatletter
\renewcommand*{\eqref}[1]{%
  \hyperref[{#1}]{\textup{\tagform@{\ref*{#1}}}}%
}
\makeatother

\begin{document}
\author{
  Olukoya, Feyishayo\\
  School Of Mathematics and Statistics,\\
  University of St. Andrews, St Andrews, Fife,\\
    \texttt{fo55@st-andrews.ac.uk}
}
\title{The growth rates of automaton groups generated by reset automata}
\begin{abstract}
We give sufficient conditions for when groups generated by automata in  a class $\mathcal{C}$ of transducers, which contains the class of reset automata transducers, have infinite order. As a consequence we also demonstrate that if a group generated by an automaton in $\T{C}$ is infinite, then it contains a free semigroup of rank at least 2. This gives a new proof, in the context of groups generated by automaton in $\mathcal{C}$, of a result of Chou  showing that  finitely generated elementary amenable groups either have polynomial growth or contain a free semigroup of rank 2. 

\end{abstract}
\maketitle

\section{Introduction}
 This article studies the finiteness problem in automata groups for a specialised class of automata which includes all \emph{reset automata}. In particular we give sufficient conditions for when these groups are infinite. Our approach also provides a new proof, in the context of groups generated by such automata, of a result of Chou \cite{ChouChing} that either the groups generated by these automata are finite or contain a free semigroup of rank at least 2.  
 
 The class of automata we are interested in are called \emph{synchronizing automata} in the sense introduced in the paper  \cite{BlkYMaisANav}. In this paper, the authors  characterise the automorphism groups of $G_{n,r}$, as subgroups of the rational group $\mathcal{R}_{n}$ of Grigorchuk, Nekreshevich Suschanski\u{\i}. 
 In describing the outer-automorphisms $\mathcal{O}_{n,r}$,  $1 \le r < n$,  of $G_{n,r}$ they introduce a sequence of subgroups as follows: $\hn{n} \lneq \mathcal{ L}_{n,r} \lneq \mathcal{O}_{n,r}$. These groups turn out to have connections to the automorphism groups of shift spaces in a manner made precise in the forthcoming paper \cite{BleakCameronOlukoya}. In particular the authors of \cite{BleakCameronOlukoya} identify $\mathcal{H}_{n}$ as a group of  topical interest in dynamics, being isomorphic to the automorphisms of the one-sided shift on $\{0,1 \ldots, n-1\}^{\mathbb{N}}$. 
 
 In order to state our main results, we informally define the terms appearing above, formal definitions will be given in Section~\ref{preliminaries}.
 
 In this article an \emph{automaton} (we will also interchangeably use the terms \emph{transducer} and \emph{Mealy-automaton}) is a machine with finitely many states, which on reading an input symbol from an alphabet of size $n$, writes a string from the same alphabet, possibly the empty string, and changes state awaiting the next input.  In the case that each state of the transducer changes an input symbol to another symbol (and not a string) we say that the transducer is \emph{synchronous}. In the case that this transformation is invertible we say the transducer is \emph{invertible}. It  is a result in \cite{GriNekSus} that the inverse is also representable by a transducer. 
 
 A transducer is said to be \emph{synchronizing at level $k$} (see \cite{BlkYMaisANav}) if there is a natural number $k \in \mathbb{N}$ such that for all strings of length $k$ in the input alphabet, the state of the transducer reached after processing such a string is independent of the starting state and depends only on the string processed. We should point out that there is a weaker notion of synchronization for automaton which occurs, for instance, in the \v{C}ern\'{y} Conjecture \cite{Volkov2008} and in the road colouring conjecture proved by Trahtman \cite{Trahtman}. We shall only be concerned with the definition of synchronization given above. We call a transducer \emph{synchronizing} if it is synchronizing at some level. If the transducer is invertible, with synchronizing inverse, then we say that the transducer is \emph{bi-synchronizing}. For a synchronizing transducer $A$ which is synchronizing at level $k$ we shall denote by $\core(A)$ the sub-transducer consisting of those states reached after reading words of length $k$ from any state of $A$.   

 We may now more properly define the groups of automata we are interested in. The  monoid $\pn{n}$ consists of core, synchronous, bi-synchronizing transducers. We note that the forthcoming paper \cite{BleakCameronOlukoya} demonstrates that elements of $\pn{n}$ induce homeomorphisms of the full two-sided shift, therefore they are invertible and although the inverse may not be contained in $\pn{n}$ it can be represented by a synchronizing transducer. Hence, the term bi-synchronizing makes sense in $\pn{n}$.   The monoid $\pn{n}$ is itself contained in the monoid $\spn{n}$ of synchronous. 
 
 Given an element $A \in \spn{n}$ we shall say that $A$ has a \emph{homeomorphism state} if there is some state $q$ of $A$ such that when $A$ is considered as the initial transducer $A_{q}$, then this induces a homeomorphism of the Cantor space $X^{\mathbb{N}}$ where  $X$ is the alphabet set. Thus, let  $\shn{n}$ be the monoid consisting of elements of $\spn{n}$ which have a homeomorphism state (as it turns out if one of the states of such an $A \in \spn{n}$ is a homeomorphism state then they all are). This article shall mainly be concerned with the group or semigroup generated by automata in the monoid $\shn{n}$, although some of our results shall also apply to $\pn{n}$ and $\spn{n}$. The subgroup of $\shn{n}$ consisting of bi-synchronizing elements we shall denote by $\hn{n}$. As observed above all elements of $\spn{n}$ induce shift commuting continuous functions on $X_n^{\Z}$; elements of $\hn{n}$ induce shift commuting homeomorphisms of $X_n^{\Z}$ \cite{BleakCameronOlukoya}. We shall not differentiate between a automata in $\spn{n}$ and the continuous function it induces on $X_n^{\Z}$.

We associate to each automata $A \in \shn{n}$ and for each natural number $r$ bigger than the synchronizing level of $A$ a finite graph $G_{r}(A)$ the \emph{graph of bad pairs}. Our main result may now be stated as follows:

\begin{Theorem}\label{Thm:badpairshascircuitimpliesifiniteorder}
Let $A \in \widetilde{\T{H}}_{n}$ be synchronizing at level $k$, and let $r \ge k$. If the graph $G_{r}(A)$ of bad pairs has a circuit, then there is rational word in $X_n^{\Z} $ on an infinite orbit under the action of $A$.
\end{Theorem}

Notice that if $A \in \hn{n}$ is such that $G_{r}(A)$ has a circuit, then $A$, in its action on $X_n^{\Z}$ has infinite order. In this case we say that $A$ is an element of infinite order. A corollary of the above result is the following:

\begin{corollary}
Let $A \in \widetilde{\T{H}}_{n}$ be synchronizing at level $k$, and let $r \ge k$. If the graph $G_{r}(A)$ of bad pairs has a circuit, then the automaton group generated by $A$ is infinite.
\end{corollary}

It turns out that  if the graph $G_{r}(A)$ has a circuit for some $r$, then $A$ contains a free semigroup. By studying the graphs $G_{r}(A)$ for sufficiently large $r \in \N$, we deduce the following result: 

\begin{Theorem}\label{Thm:ifinitethenfreesemigroup}
Let $A \in \widetilde{\T{H}}_{n}$ generate an infinite group. Then either there is a $j \in \mathbb{N}$ such that the  graph of bad pairs $G_{j}(A)$ has a loop otherwise the automaton semigroup generated by $A$ contains a free semigroup of rank at least 2.
\end{Theorem}

From this it follows that either the group or semigroup generated by $A$ is finite   or it has exponential growth rate. This result is in the spirit of  \cite{IKlimann} which studies \emph{bireversible automata}. One should note, however, that the class of bireversible automata is disjoint from the class of synchronizing automata. 

Silva and Steinberg in \cite{SilvaSteinberg} study reset automata which fit in this context as level 1 synchronizing transducers.
They show that the automaton group generated by a level 1 synchronizing transducer with a homeomorphism state is infinite if and only if it contains an element of infinite order, and in this case the group is locally-finite-by-cyclic and so amenable.  Thus if the order problem is solvable in groups generated by reset automata, then the finiteness problem is also solvable in groups generated by reset automata (in fact this is and if and only if by the results of \cite{SilvaSteinberg}). While an early draft of this paper has been under review, the two papers, \cite{LBartholdiIMitrofanov} and \cite{PierreGillibertOrder} demonstrating that the order problem in groups generated by automata is undecidable in general were uploaded to the arxiv. However the question remains open for reset automata. The results above  address this question by providing sufficient conditions for when the group generated by a reset automata is finite.
 
In \cite{SilvaSteinberg}, by making further assumptions the authors of that paper are also able to show that in the case where this group is infinite then it has exponential growth. {Using Chou's  classification of elementary amenable groups \cite{ChouChing}, a result of Rosset \cite{RossetShmuel} and the result of Silva and Steinberg that groups generated by reset automata are locally-finite-by-cyclic, it turns out that all groups generated reset automata are actually elementary amenable. Moreover, it is a result of Chou \cite{ChouChing} that all finitely generated elementary amenable groups are either virtually nilpotent (and so of polynomial growth \cite{GromovMikhael}) or contain a free semigroup on two generators. From this one may deduce that all finitely generated locally-finite-by-infinite cyclic groups contain a free semigroup on two generators}. 
Thus we reprove the result of Chou in the context of groups generated by  reset automata.
In particular, the class of automata groups generated by reset automata do not furnish examples of infinite Burnside groups or groups of intermediate growth:  these groups are finite or they have exponential growth. 

In the case that $A \in \hn{n}$ generates a finite group, the graph $G_{r}(A)$ is eventually empty. In this case, it is more fruitful to study powers of the dual automaton. We have the following result:

\begin{proposition}
Let $A$ be an element of $\T{P}_{n}$ and suppose $A$ is synchronizing at level $k$. Then the semigroup $\gen{A}$ generated by  $A$  is finite if and only if there is some $m \in \mathbb{N}$ such that the following holds:
\begin{enumerate}[label = (\roman*)]
\item $\dual{A}_{m}$ is a zero of the semigroup $\langle \dual{A} \rangle$  and,
\item $\dual{A}_{m}$ is $\omega$-equivalent to a transducer with $r$ components such that:
\begin{enumerate}
\item For each component $D_i$ $1\le i \le r$, there is a fixed pair of words $w_{i,1}, w_{i,2}$ (in the states of $A$) of  associated to $D_i$, and
\item Whenever we read any input from a state in the $D_i$, the output is of the form $w_{i,1}w_{i,2}^{l}v$ for $l \in \mathbb{N}$ and $v$ a prefix of $w_{i,2}$ or has the from $u$ for some prefix $U$ of $w_{i,1}$. Moreover the output depends only on which state in the component $D_i$ we begin processing inputs.
\end{enumerate}
 \end{enumerate}
\end{proposition}

We shall formally define the dual automaton in Section \ref{tools}, where we  also prove the above proposition.

In \cite{GriZuk}, Grigorchuk, Nekreshevich  and Suschanski\u{\i}, show that the lamplighter group can viewed as the automata group generated by a 2-state automaton. Silva and Steinberg by studying Cayley machines of  finite groups (these are finite automata whose inverse are reset automata), show that the automata group generated by the Cayley machine of a finite abelian group $G$ is isomorphic to the restricted wreath product $G \wr \mathbb{Z}$. In particular all lamplighter groups are in this class. Furthermore, by showing that these automata satisfy certain conditions, they are also able to prove that the automaton group generated by the Cayley machine of an abelian group has exponential growth.

Using tools introduced later in the paper, we demonstrate that the number of states of the $n$'th power of the Cayley machine of a group $G$ is precisely $|G|^n$. This is in keeping with the Cayley machine of the cyclic group of order 2 as was demonstrated in \cite{GriNekSus}.

In \cite{GriNekSus}, the authors also demonstrate that for the  Cayley machine of the cyclic  group of order 2, all powers of this automaton are connected (i.e every state is accessible from any other.) We show that this is also satisfied by the Cayley machine of any finite group.

\subsection{Outline of Paper}
In Section \ref{preliminaries} we give a formal definition of the semigroup $\spn{n}$ and present some of the basic properties of this group that we will need later on. In Section \ref{propertiesofpn} we introduce some various useful properties of $\spn{n}$ which, when restricted to $\shn{n}$ begin to shed light on when two elements commute or are conjugate to each other. In Section \ref{tools} we develop the techniques required to prove Theorem \ref{Thm:badpairshascircuitimpliesifiniteorder}. In Section \ref{embeddings} we introduce techniques for combining elements of $\shn{n}$ and show that direct-sums of copies of $\shn{n}$ can be embedded in $\widetilde{\T{H}}_{m}$ for $m$ large enough. In Section \ref{proofofmainresult} we prove Theorem~\ref{Thm:ifinitethenfreesemigroup}. 

\subsection{Acknowledgements}
The author would like to thank his supervisor Collin Bleak for bringing this problem to his attention and for helpful comments on  numerous versions of the paper. {The author would also like to thank Laurent Bartholdi for helpful conversations, in particular for drawing attention to the results of Ching Chou and Shmuel Rosset.} He also gratefully acknowledges the financial support of the Carnegie Trust.

\section{Preliminaries}\label{preliminaries}

\subsection{Transducers}
 Throughout the paper fix $X_n := \{0,\ldots, n-1\}$, for $2 \le n$ a natural number. For an arbitrary finite set $X$ of symbols we shall let $X^{\ast}$ be the set of all finite strings, including the empty string (which we shall always denote by $\epsilon$), and $X^+$ to be the set of all finite strings excluding the empty string. We call a word $\Gamma \in X_n^{+}$ which cannot be written as $\Gamma =  (\gamma)^r$ for a stricly smaller word $\gamma \in X_n^{+}$ and $0< r \in \N$ a \emph{prime word}. For a natural number $k \in \mathbb{N}$ we shall let $X^k$ be the set of all strings of length precisely $k$; $X^{\mathbb{N}}$ and $X^{\mathbb{Z}}$ shall denote the set of infinite and bi-infinite strings respectively. We shall represent a point $x \in X^{\mathbb{N}}$ as a sequence $(x_i)_{i \in \mathbb{N}}$ which we will normally write as $x:= x_0x_1x_2 \ldots$, likewise we shall represent a point $x \in X^{\Z}$ as a sequence $(x_i)_{i \in \Z}$, and we will normally write this as $x:= \ldots
 x_{-1}x_0x_{1}\ldots$. Sometimes it shall be convenient to consider an element $x \in X_n^{\mathbb{Z}}$ as being composed of words in $X_n^{k}$ for some natural number $k >0$ and in this case we shall write: 
 \[
    x:= \ldots\Gamma_{-1}\dot{\Gamma}_0\Gamma_{1}\ldots
 \]
where $\Gamma_{i}$, $i \in \mathbb{Z}$ are all in $X_n^{k}$. The dot above $\Gamma_0$ is used to communicate that $x_0\ldots x_{k-1} = \Gamma_0$. Then $x_{k}\ldots x_{2k-1} = \Gamma_1$ and so on. Equipping $X$ with the discrete topology, and taking the product topology on $X^{\mathbb{N}}$ and $X^{\mathbb{Z}}$ we have that both of these spaces are homeomorphic to Cantor space. It is a standard result that both $X^{\mathbb{N}}$ and $X^{\mathbb{Z}}$ with the topologies defined above, are metrisable. On $X^{\mathbb{N}}$ we shall take the metric $d_{\omega}: X^{\mathbb{N}} \times X^{\mathbb{N}} \to (0, \infty)$ given by 
 \begin{equation}\label{metriconX^omega}
 d_{\omega}(x,y) = \begin{cases}
                      \frac{1}{j+1} \mbox{ where } j \in \mathbb{N} \mbox{ is minimal so that } x_j \ne y_j \\
                      0 \mbox{ if } x = y. 
                    \end{cases} 
 \end{equation}
 
 On $X^{\Z}$ we shall take the metric $d_{\infty}: X^{\Z} \times X^{\Z} \to (0, \infty)$ given by 
  \begin{equation}\label{metriconX^Z}
  d_{\infty}(x,y) = \begin{cases}
                       \frac{1}{j+1} \mbox{ where } j \in \mathbb{N} \mbox{ is minimal so that } x_j \ne y_j \mbox{ or } x_{-j} \ne y_{-j} \\
                       0 \mbox{ if } x = y. 
                     \end{cases} 
  \end{equation} 
   For a string $\Gamma \in X^{\ast}$, we shall use $|\Gamma|$ to denote the length (or size) of $\Gamma$; the empty string has length zero. We shall also use $|X|$ to denote the cardinality of the set $X$; if $i \in \Z$ then $|i|$ shall denote the absolute value of $i$. We shall let the context determine which meaning of $| \cdot |$ is being taken.

\begin{Definition}
In our context a \emph{transducer} $A$ is a tuple $A = \gen{X_n, Q, \pi, \lambda}$ where:

\begin{enumerate}[label= (\roman*)]
\item $X_n$ is both the input and output alphabet.
\item $Q$ is the set of states of $A$.
\item $\pi$ is the \emph{transition} function, and is a map:
   \[
     \pi: X_n\sqcup\{\epsilon\} \times Q \to Q
   \]
 \item $\lambda$ is the \emph{output} or \emph{rewrite} function, and is a map:
 \[
   \lambda: X_n\sqcup\{\epsilon\} \times Q \to X_n^{\ast}
 \]
\end{enumerate}
\end{Definition}

We shall take the convention that $\pi(\epsilon, q) = q$ for any $q \in Q$, and also $\lambda(\epsilon, q) = \epsilon$, this will allow single-state transducers to be synchronizing at level 0. 
If $|Q| < \infty$ then we say the transducer $A$ is finite. $A$  is said to be \emph{synchronous} or a \emph{Mealy automaton} if $\lambda$ obeys the rule $|\lambda(x,q)| = |x|$ for any $x \in X_n\sqcup\{\epsilon\}$ and $q \in Q$. If the map $\lambda(\cdot,q) : X_n \to X_n$ is the identity map on $X_n$ then we say that $q$ acts \emph{locally as the identity}, if the map $\lambda(\cdot,q) : X_n \to X_n$ is a permutation of $X_n$ then we say that $q$ acts \emph{locally as a permutation.} If it is clear from the context then we may sometimes omit the prefix `locally'.

If we specify a state $q \in Q$ from which we start processing inputs then we say $A$ is \emph{initialised at $q$} and shall denote this $A_{q}$. The transducer $A_{q}$ is then called an initial transducer.

We can extend the domain  of $\pi$ and $\lambda$ to $X_n^{\ast} \times Q$ using the rules below and  induction:

\begin{IEEEeqnarray}{rCl}
\pi( \Gamma x, q) &=& \pi(x, \pi(\Gamma,q)) \\
\lambda( \Gamma x, q) &=& \lambda(\Gamma,q) \lambda(x,\pi(\Gamma,q)) 
\end{IEEEeqnarray}

where $\gamma \in X_n^{\ast}$, $x \in X_n$  and  $q \in Q$. Given a word in $\Gamma \in X_n^{\ast}$ and $q, p \in Q$ such that $\pi(\Gamma, q) = p$, then we shall say that we \emph{ read $\Gamma$ from state $q$ into $p$}, we shall sometimes supplement this by adding, \emph{ and the output is $\Delta$} if $\Delta = \lambda(\Gamma, q)$. 

Endowing $X_n$ with the discrete topology,
the above now means that each state $q \in Q$ induces a continuous map from Cantor space $X_n^{\mathbb{N}}$ to itself. If this map is a homeomorphism then we say that $q$ is a \emph{homeomorphism state} Two states $q_1$ and $q_2$ are then said to be $\omega$-equivalent if they induce the same continuous map. (This is can be checked in finite time.) A transducer, therefore, is called \emph{minimal} if no two states are $\omega$-equivalent. Two minimal transducers, $A = \gen{X_n, Q_A, \pi_A, \lambda_A}$ and $B= \gen{X_n, Q_B, \pi_B, \lambda_B}$, are said to be $\omega$-equivalent if there is a bijection $f: Q_A \to Q_B$ such that $q$  and $f(q)$ induce the same continuous map for $q \in Q_A$. In the case where $A$ and $B$ are $\omega$-equivalent then we write $A \wequal B$, otherwise we write $A \nwequal B$.

Given two transducers $A = \gen{X_n, Q_A, \pi_A, \lambda_A}$ and $B= \gen{X_n, Q_B, \pi_B, \lambda_B}$, the product $A*B$  shall be defined in the usual way. The set of states of $A*B$ will be $Q_A \times Q_B$, and the transition and rewrite functions, $\pi_{A*B}$ and $\lambda_{A*B}$ of $A*B$ are defined by the rules:

\begin{IEEEeqnarray}{rCl}
 \pi_{A*B}(x, (p,q)) &=& (\pi_{A}(x,p), \pi_B(\lambda_A(x,p),q)) \\
 \lambda_{A*B}(x,(p,q)) &=& \lambda_{B}(\lambda_A(x,p),q)
\end{IEEEeqnarray}

Where $x \in X_n\sqcup{\epsilon}$, $p \in Q_A$ and $q \in Q_B$. As usual $A^{i} = A_{1} \ast A_{2}\ast \ldots \ast A_{i}$ where $A_j = A$ $1\le j \le i$ and $i \in \mathbb{N}$, and $A^{-i} = (A^{-1})^{i}$. If $A:= \gen{X_n, Q_A, \lambda_A, \pi_{A}}$ then we shall set $A^{i} = \gen{X_n, Q_A^{i}, \lambda_{Ai}, \pi_{Ai}}$. 

If $A = \gen{X_n, Q_A, \pi_A, \lambda_A}$ is a synchronous transducer, then as each state $q$ of $A$ induces a continuous function of $X_n^{\mathbb{Z}}$ we may consider the subsemigroup (or group in the case that $A$ is invertible) of the endomorphisms of $X_n^{\Z}$ generated by the set $\{A_{q} | q \in Q_{A}\}$. We shall refer to this semigroup or group as the \emph{automaton semigroup or automaton group generated by $A$}. Alternatively can also consider the monogenic semigroup $\gen{A} = \{A^i | i \in \mathbb{N}\}$ or cyclic group $\gen{A} = \{A^i | i \in \Z \}$, we shall call these the \emph{the semigroup or  group generated by $A$}. When there is any ambiguity we shall make it explicit that $\gen{A}$ refers either to  the semigroup or group generated by $A$.

\begin{Definition}
Given a non-negative integer $k$ and an automaton $A=\gen{X_n,Q,\pi,\lambda}$, we say that \emph{$A$ is synchronizing at level $k$} if there is a map $\mathfrak{s}:X_n^k\to Q$, so that for all $q\in Q$ and any word $\Gamma\in X_I^k$ we have $\mathfrak{s}(\Gamma)=\pi(\Gamma,q)$. That is, the location in the automaton is determined by the last $k$ letters read. We call $\mathfrak{s}$ the synchronizing map for $A$, the image of the map $\mathfrak{s}$ the \emph{core of $A$}, and for a given $\Gamma \in X_n^k$, we call $\mathfrak{s}(\Gamma)$ the \emph{state of $A$ forced by $\Gamma$}. If $A$ is invertible, and $A^{-1}$ is synchronizing at some level $0\le l \in \mathbb{N}$, then  we say that \emph{$A$ is bi-synchronizing at level $\max(k,m)$}. If $A$ is synchronizing but not bi-synchronizing then we shall say $A$ is \emph{one-way synchronizing}.
\end{Definition}

\begin{Remark}
\begin{enumerate}[label = (\roman*)] \leavevmode
\item It is an easy observation that for a synchronizing transducer the core of $A$ is a synchronizing transducer in its own right. We shall denote this transducer by $\core(A)$, and if $A = \core(A)$ then we say that \emph{$A$ is core}.
\item For a synchronous, synchronizing transducer, $A$, $\core(A)$ induces a continuous map from $X_n^{\mathbb{Z}}$ to itself. This follows since if $\core(A)$ is synchronizing at level $k$ with synchronizing map $\mathfrak{s}$, and given a bi-infinite string $(x_i)_{i \in \mathbb{Z}}$, then $\core(A)(x_i) = \pi(x_i,\mathfrak{s}(x_{i-k}\ldots x_{i-1}))$. That is we look at the preceding $k$ symbols to determine from which state the subsequent symbol is to be processed. Let $f_{A}$ be the map induced on $X_n^{\Z}$ induced by $\core(A)$, then as $A$ is synchronous, the map $f_{A}$ preserves indices and so is a well defined map on $X_n^{\Z}$. This is not always the case, for instance if $A$ was asynchronous, and synchronizing then the indices for the action of $f_{A}$ are not well-defined. 
\end{enumerate}
\end{Remark}

Before defining the monoid $\spn{n}$ it is necessary to establish the following standard result.

\begin{claim} \label{claim-SyncLengthsAdd}
Let $A = \gen{X_n,Q_A,\pi_{A},\lambda_{A}}$ and $B = \gen{X_n,Q_{B},\pi_B,\lambda_{B},}$ be synchronizing, synchronous transducers. If $A$ is synchronizing at level $j$ and $B$ is synchronizing at level $k$ then $A*B$ is synchronizing at level $j+k$.
\end{claim}
\begin{proof}
We show that any word of length $j+k$ is synchronizing for $A*B$, where we recall the set of states of $A*B$ is the product set $Q_A\times Q_B$ (of course, not all of these states are in the core of $A*B$).

Let $\Gamma$ be a word of length $j$ and $\Delta$ a word of length $k$, both $\Gamma, \Delta \in X^\ast_{n}$. Suppose that $\Gamma$ forces us into a state $q_A$ of $A$. Let $\bar{\Gamma}$ be the length $j$ prefix of  $\Delta\Gamma$, and let $\bar{\Delta}$ be the complimentary length $k$ suffix. Let $(p_A,p_B)$ any pair in $Q_A \times Q_B$. Then we have that $\pi_A(\Delta\Gamma,p_A) = q_A$. Let $\bar{q}_A$ be the state which we are in after processing $\bar{\Gamma}$. Let $q_B$ be the state we are forced to after reading $\lambda_{A}(\bar{\Delta},\bar{q_{A}})$ from any state in $B$. Then $\lambda_A(\Delta\Gamma,p_A) = \lambda_{A}(\bar{\Gamma},p_A) \lambda_A(\bar{\Delta},\bar{q}_{A})$. Therefore $\pi_{B}(\lambda_{A}(\bar{\Gamma},p_A) \lambda_A( \bar{\Delta},\bar{q}_{A}),p_B) = q_B$, thus reading $\Delta\Gamma$ from any state pair of $Q_A \times Q_B$, the active state becomes $(q_A,q_B)$. 
\end{proof}
\begin{Remark}
From the above it follows that the set $\widetilde{\mathcal{P}}_{n}$ of core, synchronizing, synchronous transducers forms a monoid \cite{BleakCameronOlukoya}.
\end{Remark}
With this in place, given $1 \le n \in \mathbb{N}$, we can describe the monoid $\spn{n}$ as those continuous functions on $\{0,1,\ldots,n-1\}^{\Z}$ whose elements are given as finite, synchronizing, synchronous, core transducers.  (We mention core since we are restricting our attention only to those states in the core.  The product is the automaton product, where after taking this product, one removes non-core states as they are irrelevant to the action. This is always possible by the above claim.) The submonoid  $\pn{n}$ shall consist of those elements of $\spn{n}$ which  induce homeomorphisms of $X_n^{\mathbb{Z}}$. The group $\hn{n}$ is the subset of $\pn{n}$ consisting of those transducers $H$ for which  there is a state $q$ of $H$ such that the initial transducer $H_{q}$ is a homeomorphism of $X_n^{\mathbb{N}}$. Notice that since $H$ is synchronous, then this means all of its states are homeomorphism states. It is a result in the forthcoming paper \cite{BleakCameronOlukoya} that $\hn{n}$ is isomorphic to the group of automorphisms of the one-sided shift on $n$ letters. Finally define $\shn{n}$ to be those elements of $\spn{n}$ which have  a homeomorphism state and which are synchronizing but not bi-synchronizing. 

Given two elements $A, B \in \spn{n}$, then we shall denote the minimal transducer representing the core of the product of $A$ and $B$  by $\min(\core(A*B))$. Since the operations of minimising and reducing to the core commute with each other, the order in which we perform these operations is irrelevant.

Most of our results shall be for $\hn{n}$ and $\shn{n}$, though some of our results also apply to $\spn{n}$ and $\pn{n}$. We  give below  a table listing these various groups and monoid together with their defining properties.

\begin{center}
	
\begin{TAB}(r,1cm,1cm){|c|c|}{|c|c|c|c|}

	$\spn{n}$ & synchronous, core, synchronizing transducers \\ 
	$\pn{n}$  & subset of $\spn{n}$ which induce homeomorphisms of $X_{n}^{\Z}$ \\ 
	$\shn{n}$ & subset of $\spn{n}$ consisting  of elements with a homeomorphism state \\ 
	$\hn{n}$ & Subset of $\pn{n}$ containing elements with a homeomorphism state.
\end{TAB}
\end{center}

The following should aid the reader in remembering which groups/monoids are denoted by which symbols: $\hn{n}$ and $\shn{n}$ are those elements of $\pn{n}$ and $\spn{n}$, respectively, with a homeomorphism state, thus, $\mathcal{H}$ for homeomorphism; a tilde above a symbol means that the corresponding set  contains elements that do not induce homeomorphisms of $X_{n}^{\Z}$.

\section{Properties of \texorpdfstring{$\spn{n}$}{Lg}}\label{propertiesofpn}

Let $A =  \gen{Q,X_n,\pi,\lambda}$ be an element of $\shn{n}$. Let $Q^{-1}$ be a set in bijective correspondence with $Q$, such that for $q\in Q$, $q^{-1}$ denotes the corresponding state of $Q^{-1}$, and ${(q^{-1})}^{-1} = q$.  Then as $A$ is synchronous, the inversion algorithm of \cite{GriNekSus} produces an inverse automaton $A^{-1}=\gen{Q^{-1},X_n,\pi^{-1},\lambda^{-1}}$ so that for states $p,q\in Q$, and $x,y\in X_n$ such that $\pi(x,p)=q$ with $\lambda(x,p)=y$, we have $\pi^{-1}(y,p^{-1})=q^{-1}$ and $\lambda^{-1}(y,p^{-1})=x$. If $A$ represents an element of $\hn{n}$ there is a constant $k$ so that both $A$ and $A^{-1}$ are synchronizing at level $k$.  Note also that all words $\Gamma$ of length at least $k$ are synchronizing words for $A$ and for $A^{-1}$.

We make the following claims about $A$ which will be useful later on. The first two shall apply to to all elements $A \in \spn{n}$.

\begin{claim}\label{inequivstates}
Let $A$ and $B$ be elements of $\widetilde{\T{P}}_n$, and let $m \in \mathbb{N}\backslash\{0\}$ be minimal such that both $A$ and $B$ are synchronizing at level $m$. Then if  $A \nwequal B$, there is a word $\Gamma$, $|\Gamma| = k \ge m$, and states $p$ and $q$ of $A$ and $B$, respectively, such that: 
\begin{enumerate}[label = (\roman*)]
\item $p$ is the state in $A$ forced by $\Gamma$  and $q$ is the state in $B$ forced by $\Gamma$. 
\item  $p$ and $q$ are not $\omega$-equivalent. 
\end{enumerate}  
\end{claim}
\begin{proof}
Since $A \nwequal B$ they induce different homeomorphisms  of $X_{n}^{\mathbb{Z}}$, and so there is a bi-infinite word $w = \ldots x_{-2}x_{-1}x_0x_1\ldots$ which they process differently. 

Let $w_1 = \ldots y_{-2}y_{-1}y_{0}y_1y_2\ldots$ and $w_2 = \ldots z_{-2}z_{-1}z_{0}z_1z_2\ldots$ be the outputs from $A$ and $B$ respectively. Let $k \in \mathbb{N}\backslash\{0\}$ be such that $A$ and $B$ are synchronizing at level $k$. Note that $k \ge m$. Let $l \in \mathbb{N}$ be minimal such that $y_{l} \ne z_l$ or $y_{-l} \ne z_{-l}$. Then one of the words $x_{l-k}\ldots x_{l-2}x_{l-1}$ or $x_{-l-k}\ldots x_{-l-2}x_{-l-1}$ satisfies the premise of the claim. 
\end{proof}

\begin{lemma}\label{corehasinfiniteorderimpliesdifferentpowersnotequivalent}
Let $A \in \spn{n}$ be such that  $\min\core{(A^{i})} \nwequal \min\core{(A^{j})}$ for any pair $i, j \in \mathbb{N}$ . Then for $i \ne j \in \mathbb{N}$ and two distinct states $u$ and $v$ of $A^{i}$ and $A^{j}$ respectively, the initial transducers $A^{i}_{u}$ and $A^{j}_{v}$ are not $\omega$-equivalent.
\end{lemma}

\begin{proof}
Let $A, i, j, u$ and $v$ be as in the statement of the lemma above. Observe that since $\min\core(A^{i}) \nwequal \min\core(A ^{j})$, by Claim \ref{inequivstates} there is a word  $\Gamma$ of size greater than or equal to the maximum of the minimum synchronizing levels of $A$ and $B$ such that the state of $\min\core(A^{i})$ forced by $\Gamma$ is not $\omega$-equivalent to the state $\min\core(A^{j})$ forced by $\Gamma$. Now since  $A^{i}$ and $A^{j}$ are synchronizing, the initial transducers $A^{i}_{u}$ and $A^{j}_{v}$ are also synchronizing. Moreover $\core(A^{i}_{u}) \wequal \min\core(A^{i})$, likewise $\core(A^{j}_{v}) \wequal \min\core(A^{j})$.  Therefore let $\Lambda$ be a long enough word such that when read from the state $u$ of $A^{i}$ and state $v$ of $A^{j}$ the resultant state is in the core of $A^{i}$ and $A^{j}$ respectively. Now let $u'$ and $v'$ be the states of $A^{i}_{u}$ and $A^{j}_{v}$ respectively reached after reading $\Lambda\Gamma$ in $A^{i}_{u}$ and $A^{j}_{v}$. Then $u'$ and $v'$ are not $\omega$-equivalent since $\core(A^{i}_{u}) \wequal \min\core(A^{i})$, and $\core(A^{j}_{v}) \wequal \min\core(A^{j})$. Therefore there exits a word $\delta \in X_n^{\N}$ such that $(\delta)A^{i}_{u'} \ne (\delta)A^{i}_{v'}$ therefore we have that $(\Lambda\Gamma\delta)A^{i}_{u} \ne (\Lambda\Gamma\delta)A^{j}_{v}$. The result now follows.
\end{proof}

We introduce some terminology for the claim below. Let $X$ be a finite set and $\rho$ a permutation of $X$. For any $a_1 \in X$, there is a sequence $(a_1 a_2 \ldots a_m)$ such that $(a_i) \rho= a_{i+1}$ for $1 \le i <m$ and $(a_m)\rho = a_1$. We call such a sequence a \emph{disjoint cycle}. Note that up to cyclically permuting the elements, the disjoint cycle containing a given element $a \in X$ is unique, therefore we typically do not distinguish between a disjoint cycle $(a_1 a_2 \ldots a_m)$ and a cyclic permutation $(a_i a_{i+1} \ldots a_{m} a_1 \ldots a_m)$ of its elements for some $1 \le i \le m$. Another way of saying this, is that for $a \in X$, a disjoint cycle $(a_1 a_2 \ldots a_m)$ containing $a$ corresponds uniquely to the permutation $\gamma$ of $X$ which maps $a_{i} \to a_{i+1}$ for all $1 \le i \le m$, maps $a_{m} \to a_1$ and fixes every other point. Clearly the map $\gamma$ is invariant under cyclically permuting the elements of the disjoint cycle containing $a$. We do not distinguish between a disjoint cycle containing $a$ and the permutation $\gamma$. The length of a disjoint cycle is the number of points in its support i.e the number of points it moves. Let $\gamma_1, \ldots, \gamma_{l}$ be disjoint cycles such that every  element of $X$ is contained in exactly one of the $\gamma_{i}$. Then, we say that $\gamma_1 \gamma_2 \ldots \gamma_{l}$ is an expression of $\rho$ as a \emph{product of disjoint cycles}. Since  cycles of length $1$ correspond to the identity map, it is conventional to omit disjoint cycles of length $1$ when expressing a permutation as a product of disjoint cycles. Two permutation $\rho$ and $\phi$ are said to have the same disjoint cycle structure  if there are disjoint cycle decompositions $\rho = \gamma_1 \gamma_2 \ldots \gamma_{l}$ and $\phi= \eta_1 \eta_2 \ldots \eta _{l}$ such that for $1 \le i \le l$, the length of $\gamma _i$ is equal to the length f $\eta_i$.

\begin{claim}\label{cycls<k}
Let $A$ be a finite, invertible, synchronous transducer which is bi-synchronizing at level $k$. Then for any non-empty word $\Gamma_1$  there is a unique state $q_{1} \in Q_A$ such that $\pi(\Gamma_1,q_1) = q_1$. Moreover,  there is a disjoint cycle $(\Skip{\Gamma_1}{}{\Gamma_m})$ such that we have the following:
\begin{enumerate}[label= (\roman{*})]
\item Let $q_i$, $1\le i < m$ be such that $\pi(\Gamma_i,q_i) = q_i$ then $\Gamma_{i+1} = \lambda(\Gamma_i,q_i)$
\item We have $\Gamma_1 = \lambda(\Gamma_m,q_m)$
\end{enumerate}
\end{claim}
\begin{proof}
Through out the proof let $\Gamma$ be any non-empty word of length $j \ge 1$. We observe first that if there is a state $q$ such that $\pi(\Gamma, q) = q$ then this state must be unique. Since if there was a state $q'$ such that $\pi(\Gamma, q') = q'$ then $\pi(\Gamma^k, q) = q$ while $\pi(\Gamma^{k}, q')=q'$, and since $\Gamma^k$ has length at least $k$ we see it is a synchronizing word and so can conclude that $q=q'$.

To see that such a state $q$ exists, consider again the word $\Gamma^k$. Since $\Gamma$ is non-empty, $|\Gamma^k| \ge k$, so there is a unique state $q$ such that $\Gamma^k = q$.   Now consider the state $p$ so that $\pi(\Gamma,q)=p.$  Since $\pi(\Gamma^k,q)=q$ it is the case that $\pi(\Gamma^{k+1},q) =p$, but $\Gamma^k$ and $\Gamma^{k+1}$ have the same length $k$ suffix, so that $p=q$.  In particular, we have $\pi(\Gamma,q)=q$.

We now free the symbol $\Gamma$.  We want to show the map defined on $X_n^j$ (words of length exactly $j$), by $\Gamma \mapsto \lambda(\Gamma,q)$, where $\pi(\Gamma,q) = q$, is a bijection (and so is decomposable into disjoint cycles as indicated in the statement of the claim).

To prove this map is injective, suppose there are two words, $\Gamma$ and $\Delta$, with associated states $q$ and $r$ respectively, of length $l$, such that $\lambda(\Gamma,q) = \Gamma' = \lambda(\Delta,r)$. Now, as $q$ is the state forced by $\Gamma^k$ as above, while $r$ is the state forced by $\Delta^k$ (again as above), we see that $\pi^{-1}((\Gamma')^k,q)=q$ while $\pi^{-1}((\Gamma')^k,r)=r$, but as $(\Gamma')^k$ is synchronizing for $A^{-1}$ we must have that $q=r$, and then, by injectivity of $A_q$, that $\overline{\Gamma}=\overline{\Delta}$, so that in particular $\Gamma=\Delta$.

Therefore for each $j \in \mathbb{N}$, $j \ge 1$ the map induced by $A$, from the set of words of length $j$ to itself, is injective.  Therefore as this set of words is finite, the map is actually a bijection and so can be represented as product of a set of disjoint cycles (of words of length $j$).
\end{proof}

\begin{Remark}
Notice that we have only used the full bi-synchronizing condition in arguing invertibility. The existence and uniqueness of the state $q \in Q$ such that for $1 \le j \in \mathbb{N}$ and $\Gamma \in X_n^j$, $\pi(\Gamma, q) = q$ holds for all elements of the monoid $\widetilde{\mathcal{P}}_{n}$.
\end{Remark}
We illustrate the above claim with the example below.

\begin{example} \label{example 1}
Let $C$ be the following transducer:
\begin{figure}[h!]
\begin{center}
\begin{tikzpicture}[shorten >=0.5pt,node distance=3cm,on grid,auto] 

   \node[state] (q_0)   {$q_0$}; 
   \node[state] (q_1) [below left=of q_0] {$q_1$}; 
   \node[state] (q_2) [below right=of q_0] {$q_2$}; 
    \path[->] 
    (q_0) edge [in=105,out=75,loop] node [swap]{$0|1$} ()
          edge  [out=335,in=115]node [swap]{$2|0$} (q_2)
          edge  [out=185,in=85]node [swap]{$1|2$} (q_1)
    (q_1) edge[out=65,in=205]  node  [swap]{$0|0$} (q_0)
          edge [in=240,out=210, loop] node [swap] {$1|2$} ()
          edge [out=330,in=210] node [swap]{$2|1$} (q_2)
    (q_2) edge[out=95,in=355]  node [swap]{$1|1$} (q_0)
          edge [out=195,in=345]  node [swap]{$0|2$} (q_1) 
          edge [in=330,out=300, loop] node [swap] {$2|0$} ();
\end{tikzpicture}
\caption{An example}
\end{center}
\end{figure}
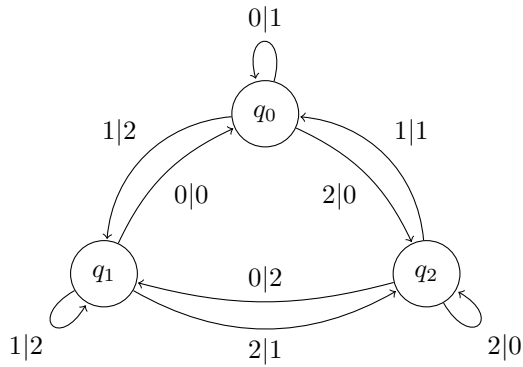
It is easily verified that this transducer is bi-synchronizing at level 2. The sets $\{00,10,21\}$, $\{01,11,20\}$ and $\{02,12,22\}$ are, respectively, the set of words which force the  states $q_0$, $q_1$ and $q_2$. The permutation of words of length 2 associated to this transducer in the manner described above is given by: $(00 \ 11 \ 22) (10 \ 20 \ 12 ) ( 21 \ 01 \ 02)$. The curious reader might have observe that these disjoint cycles have an interesting structure: if we consider the states forces by each element of a cycle then the result is a cyclic permutation of $(q_0 \ q_1 \ q_2)$. We shall later see how such behaviour plays a role in understanding the order an element.
\end{example}
\begin{Remark}
If we have found a permutation (as above) for a transducer $A$  for words of  length $j \ge 1$, then the disjoint cycle structure of this permutation will be present in all permutations associated to $A$ for words of length  $mj$, for $m \in \mathbb{N} \backslash \{0\}$. This is seen for example  if $(\Skip{\Gamma_1 \ }{}{\ \Gamma_l})$ is a disjoint cycle in the permutation associated to words of length $j$ , then $(\Skip{\Gamma_1\Gamma_1 \ }{}{ \ \Gamma_l\Gamma_l})$ is a disjoint cycle in the level $2j$ permutation. This is because each $\Gamma_i$ is processed from the state of $A$ it forces and the output is $\Gamma_{i+1}$. Generalise in the obvious way for the permutation of words of length $mj$. For instance in the example above $(0\ldots 0 \ 1\ldots1 \ 2 \ldots 2)$ will be present in the permutation of words of length $2m$  associated to $C$ (where each $i\ldots i$ is of length $2m$, $i \in \{0,1,2\}$).
\end{Remark}

We establish some further notation. For $A \in \hn{n}$ bi-synchronizing at level $k$, and $1 \le j \in \mathbb{N}$, let $\overline{A_j}$ represent the permutation of $X_n^j$ indicated in Claims \ref{cycls<k}.

\begin{Remark}\label{transformations}
Observe that a similar proof to that given in Claim \ref{cycls<k} will show that we can analogously associate to each element of $\widetilde{\mathcal{P}}_{n}\backslash \pn{n}$ a map from $X_n^j \to X_n^j$ for every $1\le j \in \mathbb{N}$. However this map need not be invertible for every such $1 \le j$, (we shall later see that for one-way synchronizing transducers there is some $j$, where the map so defined is not invertible). In light of this, for each  $A = \gen{X_n,Q, \pi, \lambda} \in \widetilde{\mathcal{P}}_{n}$ and $1 \le j \in \mathbb{N}$ let $\overline{A}_{j}: X_n^j \to X_n^j$ be the transformation given by $ \Gamma \mapsto \lambda(\Gamma,q)$ where $q \in Q$ is the unique state such that $\pi(\Gamma,q) = q$. We observe that if $A \in \pn{n}$ then $\overline{A}_{j}$ is a permutation for every $j \in \mathbb{N}$. This is because  $\pn{n}$ induces a homeomorphism of $X_n^{\mathbb{Z}}$. Since if for some $j \in \mathbb{N}$, $\overline{A}_{j}$ is not injective, then there are words $\Gamma, \Delta$ for which $(\Gamma)A_{j} = (\Delta)A_j$, this means that the bi-infinite strings $(\ldots\Gamma\dot{\Gamma}\Gamma\ldots)$ and $(\ldots\Delta\dot{\Delta}\Delta\ldots)$ are mapped to the same element of $X_n^{\mathbb{Z}}$ by $A$ contradicting injectivity.  
\end{Remark}
The following lemma shows that these maps  behave well under multiplication. 

\begin{claim}\label{cyclswellbhvd}
Let $A = \gen{X_{n}, Q_{A}, \pi_{A}, \lambda_{A}}$ and $B = \gen{Q_{B},X_{n}, \pi_{B}, \lambda_{B}}$ be elements of $\widetilde{\T{P}}_n$. Let $A*B = \gen{Q_{A},S, \pi_{A*B},\lambda_{A*B}}$ be the core of the product of $A$ and $B$, where $S \subset Q_A \times Q_B$ is the set of states in the core of $A\ast B$. Then $\overline{(A*B)}_{l} = \overline{A}_{l}*\overline{B}_{l}$.
\end{claim}
\begin{proof}
Let $\Gamma$ be a word of length $l$ in $X_{n}$, and let $p \in Q_A$ be such that $\pi_{A}(\Gamma, p) = p$. Let $\Delta:= \lambda_{A}(\Gamma, p)$, and let $q \in Q_{B}$ be such that $\pi_{B}(\Delta, q) = q$. Then $(p,q)$ is a state of $A*B$ such that $ \pi_{A*B}(\Gamma,(p,q)) = (p,q)$. If $\Lambda = \lambda_{B}(\Delta, q)$, the we have in $\overline{(A*B)}_{l}$ that $\Gamma \mapsto \Lambda$. However $\overline{A} * \overline{B}$ sends $\Gamma$ to $\Lambda$ also. Since $\Gamma$ was an arbitrary word of length $l$, this gives the result.
\end{proof}

 Let $\tau_{l}: \spn{n} \to \sym{(X_n^{l})}$ be the map defined by $A \mapsto \overline{A}_{l}$ for every $l \in \mathbb{N}$. Below we demonstrate the usefulness of these maps.

\begin{proposition} \label{usefulnessofperms}
Let $A$ and $B$ be elements of $\spn{n}$ then the following hold:
\begin{enumerate}[label =  (\roman*)]
\item $A$ and $B$  commute if and only if for every $l \ge 1$ $\overline{A}_{l}$ and $\overline{B}_{l}$ commute.
\item $A$ and $B$ are conjugate by an invertible element of $\spn{n}$  if and only if there is an  invertible, $h \in \spn{n}$, such that for  every $l \ge 1$ $\overline{h}_{l}^{-1}\overline{A}_{l} \overline{h}_{l} = \overline{B}_{l}$.
\item $A$ and $B$ are equal if and only if for every $l \ge 1$ $\overline{A_l} = \overline{B_l}$.
\end{enumerate}
\end{proposition}
\begin{proof}
The forward direction in all cases follows by Claim \ref{cyclswellbhvd} above which shows that the map $\tau_{l}: \spn{n} \to \sym{(X_n^{l})}$ is a monoid homomorphism. We need only prove the reverse implications.

We proceed by contradiction. 

For (i) suppose that $\overline{A}_l$ and $\overline{B}_{l}$ commute for every $l$ however $\core(B*A) \nwequal \core(A*B)$. Let $m \in \mathbb{N}\backslash\{0\}$ be  such that both $\core(A*B)$ and $\core(B*A)$ are bi-synchronizing at level $m$. Let $\Gamma$ be a word of length $m$ as in Claim \ref{inequivstates} such that $p$ is the state of $\core(A*B)$ forced by $\Gamma$ and $q$ is the state of $\core(B*A)$ forced by $\Gamma$. 

Let $\lambda_{AB}$ and $\lambda_{BA}$ denote, respectively, the output function of $\core(A*B)$ and $\core(B*A)$. Since $p$ is not $\omega$-equivalent to $q$ there is a word  $\Delta$, of length $l \ge 1$ say, such that $\Lambda: = \lambda_{AB}(\Delta,p) \ne \lambda_{BA}(\Delta,q) =: \Xi$. This now means that in $\overline{\core(A*B)_{l+m}}$, $\Delta\Gamma \mapsto \Lambda W_1$ and in $\overline{\core(B*A)}_{l+m}$, $\Delta\Gamma \mapsto \Xi W_2$ (for some words $W_1$ and $W_2$ of length $l$).  Therefore we conclude that $\overline{\core(A*B)_{l+m}} \ne \overline{\core(B*A)_{l+m}}$ which is a contradiction.

Part (ii) proceeds in a analogous fashion. Suppose $A,B$ and $h$ are as in the statement of Proposition 1.1 (ii), but $\core(A \ast h) \ne \core(h \ast B)$. Let $m \in \mathbb{N}\backslash\{0\}$ be such that $\core(A \ast h)$ and $\core(h \ast B)$ are bi-synchronizing at level $m$. Let $\Gamma$ be a word as in Claim \ref{inequivstates} and let $p$ be the state of $\core(A h)$ forced by $\Gamma$  and $q$ the state of $\core(h \ast B)$ forced by $\Gamma$ and $p$ and $q$ are not $\omega$-equivalent. Now we are able to construct a word as in part (i) demonstrating that $\overline{\core(A \ast h)}_{l} \ne \overline{\core(h \ast B)}_{l}$ for some $l$ yielding a contradiction.

Part (iii) follows from Part (ii) with $h$ the identity transducer.
\end{proof}

The following is a corollary of Proposition \ref{usefulnessofperms}.

\begin{corollary}\label{corollaryusefulnessofperms}
Let $A$ and $B$ be elements of $\spn{n}$, and let $k \ge 1 \in \mathbb{N}$ be such that both $A$ and $B$ are synchronizing at level $k$ then the following hold:
\begin{enumerate}[label = (\roman*)]
\item $A=B$ if and only if  $\overline{A}_{k+1} = \overline{B}_{k+1}$.
\item Let $BA$ and $AB$ denote the minimal transducers representing the products $\core(A*B)$ and $\core(B*A)$ and $l \ge 1 \in \mathbb{N}$ be such that both  $AB$ and $BA$ are synchronizing at level $l$. Then $AB = BA$ if and only if $\overline{A}_{l+1}\overline{B}_{l+1}= \overline{B}_{l+1}\overline{A}_{l+1}$.
\item  $A$ and $B$ are conjugate in  $\spn{n}$ if and only if there is an invertible $h \in \spn{n}$ such that $h^{-1}Ah$ (where this is the minimal transducer representing the product) is synchronizing at level $k$ and $\overline{h}^{-1}_{k+1}\overline{A}_{k+1}\overline{h}_{k+1} = \overline{B}_{k+1}$.
\end{enumerate}
\end{corollary}
\begin{proof}
Throughout the proof all products indicated shall represent the minimal transducer under $\omega$-equivalence representing the product. 

Observe that parts $(ii)$ and $(iii)$ are consequences of part $(i)$. Since for part $(ii)$ $AB$ and $BA$ are synchronizing at level $l$; for part $(iii)$ $B$ and $h^{-1}Ah$ are synchronizing at level $k$ (where $h$ is the conjugator). Therefore it suffices to prove only part $(i)$.

The forward implication follows by Proposition \ref{usefulnessofperms}, so we need only show the reverse implication. Let $k$ be as in the statement of part $(i)$ and assume that $\overline{A}_{k+1} = \overline{B}_{k+1}$. Denote by a triple $(\Xi,u,v)$ for $\Xi \in X_{n}^{k}$ and $u$ and $v$ states of $A$ and  $B$ respectively, that $u$ is the state of $A$ forced by $\Xi$ and $v$ is the state of $B$ forced by $\Xi$. Notice that for each such  $\Xi \in X_n^k$ such a triple is unique.

Let $\Gamma \in X_{n}^{k}$, belong to a triple $(\Gamma,p,q)$. Let $i \in X_n$ be arbitrary. Since $\overline{A}_{k+1} = \overline{B}_{k+1}$, we must have that $A_{p}(i) = B_q(i)$ since $\overline{A}_{k+1}(\Gamma i) = \overline{B}_{k+1}(\Gamma i)$.

Free the symbols $\Gamma$, $p$, and $q$.

Now let $w=\ldots w_{-k}\ldots w_{-1}w_{0}\Skip{w_1}{}{w_{k}}\ldots$ be a bi-infinite word. We show that $A$ and $B$ process this word identically. Let $w_i$ $i \in \mathbb{Z}$ denote the $i$\textsuperscript{th} letter of $w$. Then the $i$\textsuperscript{th} letter of $A(w)$ is $A_p(w_i)$ where $p$ is the state of $A$ forced by $\Gamma = \Skip{w_{i-k}}{}{w_{i-1}}$, the word of length $k$ immediately to the left of  $w_i$. Likewise the $i$\textsuperscript{th} letter of $B(w)$ is $B_{q}(w_i)$ where $q$ is the state of $B$ forced by $\Gamma$. Therefore $(\Gamma, p, q)$ is an allowed triple. However from above  we know that $A_{p}(w_i) = B_{q}(w_i)$. Since $i \in \mathbb{Z}$ was arbitrary, $A(w) = B(w)$, and $A = B$ since $w$ was arbitrary and $A$ and $B$ are assumed minimal.
\end{proof}

\begin{Remark}
Recall that a group is said to be \emph{residually finite} if for any non-identity element $g$ of the group, there is a homomorphism onto a finite group mapping $g$ to a non-trivial element. Corollary \ref{corollaryusefulnessofperms} part (i) demonstrates that the group $\hn{n}$ is residually finite. This is because given $A \in \hn{n}$ a non-identity element that is synchronizing at level $k$, the map sending $B \in \hn{n}$ to $\overline{B}_{k+1}$ in the symmetric group on $n^{k+1}$ points, is a homomorphism that maps $A$ to a non-trivial element.
\end{Remark}

\begin{Remark}
Part (ii) of Corollary \ref{corollaryusefulnessofperms}  demonstrates that if $B \in \spn{n}$ is synchronizing at level $j \ge 1 \in \mathbb{N}$, and $A \in \spn{n}$ is synchronizing at level $k \ge 1 \in \mathbb{N}$, then $B$ commutes with $A$ if and only if $\overline{A}_{j+k+1}$ commutes with $\overline{B}_{j+k+1}$ by  Claim \ref{claim-SyncLengthsAdd}.
\end{Remark}

\begin{Remark}
In order to restate Corollary \ref{corollaryusefulnessofperms} $(iii)$ for a non-invertible $h \in \spn{n} \backslash \pn{n} $  showing that the equation $\overline{A}_{k+1}\overline{h}_{k+1} = \overline{h}_{k+1}\overline{B}_{k+1}$ holds might no longer suffice. Instead we might have to check that $\overline{A}_{j+1}\overline{h}_{j+1} = \overline{h}_{j+1}\overline{B}_{j+1}$ where $j \in \mathbb{N}$ is  a level such that $\core(A*h)$ and $\core(B*h)$  are synchronizing at level $j$.
\end{Remark}


The result distinguishes between elements of $\T{H}_{n}$ and $\shn{n}$. However, we require the following definitions first.

\begin{Definition}
	Let $\Gamma =  \gamma_0\gamma_2\ldots\gamma_k-1$ be a word in $X_n^{k}$ for some natural number $k >0$. Define the $i$\textsuperscript{th} rotation of $\Gamma$ to be the word: $\Gamma'=\gamma_{k-i}\gamma_{k-i+1}\ldots\gamma_0\gamma_1\ldots\gamma_{k-i-1}$.
\end{Definition}

\begin{Remark}
	One can think of $\Gamma$ as decorating a circle divided into $k$ intervals (counting from zero), and $\Gamma'$ is the result of rotating the circle clockwise by $i$. Then  the 0\textsuperscript{th} rotation of $\Gamma$ is simply $\Gamma$.
\end{Remark}

\begin{proposition}\label{automorphismiffbi-synch}
Let $A = \gen{X_n,Q_{A},\pi_{A},\lambda_{A}}$ be an element of $\shn{n} \backslash \hn{n}$  with synchronizing level $k$ and let $A^{-1} = \gen{X_n,Q_{A^{-1}},\pi_{A^{-1}},\lambda_{A^{-1}}}$ the inverse of $A$. Then there is an $l \in \mathbb{N}$ with $0< l \le k(|Q_A|^2 +1)$ such that $\overline{A}_{l}$ is not a permutation. In particular, the action of $A$ on $X_n^{\mathbb{Z}}$ is non-injective. Moreover there exists words $\Delta$ and $\Lambda$ in $X_n^+$ such that $\Delta$ is not a cyclic rotation of $\Lambda$ and the bi-infinite strings $(\ldots\Delta\Delta\ldots)$ and $(\ldots\Lambda\Lambda\ldots)$ have the same image under $A$.
\end{proposition}
\begin{proof}
Suppose $A$ is synchronizing at level $k$.
Since $A^{-1}$ is not synchronizing it follows that $|Q_{A}|= |Q_{A^{-1}}|>1$. Moreover, there is a pair of states $(r_1,r_2)$ such that there is an infinite set $W_1$ of words $w_i \in X_n^{+}$ for which $\pi_{A^{-1}}(w_i,r_1) \ne \pi_{A^{-1}}(w_i,r_2)$. This follows since  $A^{-1}$ is not synchronizing at level  $l$ for any $l \in \mathbb{N}$. Therefore for each $l \in \mathbb{N}$ there is a pair states $(r^{l}_1, r^l_2)$ and a word $w_l \in X_n^{l}$ such that $\pi_{A^{-1}}(w_l,r^l_1) \ne \pi_{A^{-1}}(w_l,r^l_2)$. Since $A$ is a finite automaton there is a pair of states $(r_1, r_2)$  such that for infinitely many $l \in \mathbb{N}$, $(r^l_1, r^l_2) = (r_1, r_2)$, therefore taking $W_1 := \{w_l | l \in \mathbb{N} \mbox{ and } (r^l_1, r^l_2) = (r_1, r_2) \}$, $(r_1, r_2)$ and $W_1$ satisfy the conditions.

Now since $W_1$ is infinite, by an argument similar to that above, there is a pair of states $(s_1,s_2)$ such that $\pi_{A^{-1}}(w_i,r_1) =s_1$ and $ \pi_{A^{-1}}(w_i,r_2) = s_2$ and  $s_1 \ne s_2$ for infinitely many $w_i \in W_1$. Let $W_2$ denote the set of words $w_i$ such that $\pi_{A^{-1}}(w_i,r_1) =s_1$ and $ \pi_{A^{-1}}(w_i,r_2) = s_2$.

Let $w_i \in W_2$ be such that $|w_i| \ge k(|Q_A|^2 +1)$. Now since $s_1 \ne s_2$, then for any prefix $\varphi$ of $w_i$ we must have $\pi_{A^{-1}}(\varphi,r_1) \ne \pi_{A^{-1}}(\varphi,r_2)$. Moreover since $|w_i| \ge k(|Q_A|^2 +1)$ there are prefixes $\varphi_1$ and $\varphi_2$ of $w_i$ such that $\left||\varphi_1| - |\varphi_2|\right| = jk \le k(|Q_A|^2 +1)$ ($j \in \mathbb{N}\backslash\{0\}$) satisfying $\pi_{A^{-1}}(\varphi_1,r_1) = \pi_{A^{-1}}(\varphi_2,r_1) = p^{-1}$ and $\pi_{A^{-1}}(\varphi_1,r_2) = \pi_{A^{-1}}(\varphi_2, r_2) = q^{-1}$ with $p^{-1} \ne q^{-1}$, and  $p^{-1},q^{-1} \in Q_{A^{-1}}$.

Assume $\varphi_1$ is a prefix of $\varphi_2$ and let $v$ be the such that $\varphi_1v = \varphi_2$. By construction $v$ satisfies $\pi_{A^{-1}}(v,p) = p$ and  $\pi_{A^{-1}}(v,q) = q$ such that $p^{-1} \ne q^{-1}$. Let $\Lambda = \lambda_{A^{-1}}(v,p^{-1})$ and $\Delta = \lambda_{A^{-1}}(v,q^{-1})$. Since $A$ is synchronizing at level $k$ and synchronous, $\Lambda \ne \Delta$, otherwise $p = q$ and since $A$ is synchronous $|\Lambda| = |\Delta|$. 

Therefore in $A$ we have, $\pi_{A}(\Lambda, p) = p$ and $\pi_{A}(\Delta,q) = q$ moreover, $\lambda_{A}(\Lambda,p) = \lambda_{A}(\Delta,q) = v$.  This shows that  $\overline{A}_{\Lambda}$ is not a permutation of $X_n^{|\Lambda|}$. We now make the assumption that $\Lambda$ and $\Delta$ are the smallest words such that $\pi_{A}(\Lambda, p) = p$ and $\pi_{A}(\Delta,q) = q$ moreover, $\lambda_{A}(\Lambda,p) = \lambda_{A}(\Delta,q)$. Let $v \in X_n^{|\Lambda|}$ be such that $\lambda_{A}(\Lambda,p) = \lambda_{A}(\Delta,q) = v$.

In order to show that $A$ represents a non-injective map on $X_n^{\mathbb{Z}}$ observe that the bi-infinite strings $(\ldots\Lambda\Lambda\ldots)$ and $(\ldots\Delta\Delta\ldots)$ are mapped to the bi-infinite string $(\ldots vv\ldots)$ under $A$. Therefore taking $(\ldots\Theta\dot{\Theta}\Theta\ldots)$ for $\Theta \in X_n^+$ to represent the element $y \in {X_n^{\mathbb{Z}}}$ defined by $y_{j|\Theta|}y_{j|\Theta|+1}\ldots y_{j|\Theta|+ |\Theta|-1} := \Theta$ for any $j \in \mathbb{Z}$, we see that $(\ldots\Lambda\dot{\Lambda}\Lambda\ldots)$ and $(\ldots\Delta\dot{\Delta}\Delta\ldots)$ are distinct elements of $X_n^{\mathbb{Z}}$ which have the same image under $A$. This shows $A$ is non-injective.  

To conclude the proof we now need to argue that there exists words $\Lambda'$ and $\Delta'$ which are not cyclic rotations of each other such that $(\ldots\Lambda'\dot{\Lambda}'\Lambda'\ldots)$ and $(\ldots\Delta'\dot{\Delta}'\Delta'\ldots)$ are mapped by $A$ to the same place.

Suppose that $\Lambda$ is a cyclic rotation of $\Delta$, since we are done otherwise. 

Since $\pi_{A}(\Lambda, p) = p$ we must have that $v$ is equal to a non-trivial cyclic rotation of itself. This is the case if and only if $v$ is equal to some power of a third word $\nu$ strictly smaller than $v$ (see for instance \cite[Theorem 1.2.9]{MVSapir}). In fact if $v = v' v'' = v''v'$ then both $v''$ and $v'$ are powers of this word $\nu$.

 We may assume that $\nu$ is a prime word (that is, it cannot be written as a powers of a strictly smaller word). Let $r \in \N$ be such that $\nu^r = v$. Notice that $r|\nu| = |v| = |\Lambda|$.

First suppose that there is word $u \in X_n^{|\nu|}$ such that $A_{|\nu|}(u) = \nu$ and $u^r$ is a rotation of $\Lambda$. If a non-trivial suffix $u_1 \ne u$ of $u$ is a prefix of $\Lambda$, then since $\lambda_{A}(\Lambda,p) = \nu^{r} =v$, we must have that $\nu$ is equal to a non-trivial cyclic rotation of itself contradicting that $\nu$ is a prime word. Therefore $\Lambda = u^r$. However, since $A_{|\Lambda|}(\Delta) = A_{|\Lambda|}(\Lambda)$ and $\Delta$  is a  cyclic rotation  of $\Lambda$ then $u^r$ is also a cyclic rotation of $\Delta$. Therefore by the same argument we must have that $\Delta = u^r$. However this now implies that $\Delta = \Lambda$ yielding a contradiction since we assumed that $\Delta \ne \Lambda$.

Now since $|\nu| < |v|$, then either there is a word $u$, such that $|u| = |\nu|$ for which $\overline{A}_{|\nu|}(u) = \nu$ or $\overline{A}_{|\nu|}$ is not surjective from $X_n^{|\nu|}$ to itself, and so it is also not injective (since $X_n^{|\nu|}$ is finite). If the latter occurs, then there are strictly smaller distinct words $\Lambda'$ and $\Delta'$ and states ${p}'$ and ${q}'$ such that $\pi_{A}(\Lambda', {p}') = {p}'$ and $\pi_{A}(\Delta',{q}') = {q}'$ so that, $\lambda_{A}(\Lambda',{p}') = \lambda_{A}(\Delta',{q}')$. Notice that since $A \in \hn{n}$ all its states are homeomorphism states, therefore $p'$ and $q'$ cannot be equal or $A$ would have a non-homeomorphism state. However this is a contradiction since we assumed that $\Lambda$ and $\Delta$ were the smallest such words. Therefore there is a word $u$ so that $|u| = |\nu|$ and $\overline{A}_{|\nu|}(u) = \nu$. Notice that $u^r$ cannot be a  rotation of $\Lambda$ by an argument above.  Moreover the bi-infinite sequences  $(\ldots \dot{u^r}u^r\ldots)$ and $(\ldots \dot{\Lambda}\Lambda\ldots)$ are mapped by $A$ to the same  bi-infinite string $(\ldots \dot{v}v\ldots)$. 
\end{proof}

\begin{Remark}
Let $A$ be an element of $\shn{n} \backslash \hn{n}$ which is invertible as an automaton, then $A$ represents a surjective map from the Cantor space $X_n^{\mathbb{Z}}$ to itself. In particular as a consequence of the proposition above an element $A \in \shn{n}$  is injective on $X_n^{\Z}$  if and only if it is a homeomorphism if and only if it is bi-synchronizing.
\end{Remark}
\begin{proof}
Our argument shall proceed as follows, we shall make use of the well known results that the continuous image of a compact topological space is compact, and that a compact subset of a Hausdorff space is closed. This means it suffices to argue that the image of $A$ is dense in $X_n^{\mathbb{Z}}$.

Let $k \in \mathbb{N}$ be the minimal synchronizing level for $A$.

Notice that since $A$ is invertible as an automaton each state of $A$ defines an invertible map from $X_n^{\mathbb{N}}$ to itself. Therefore given an element $y \in X_n^{\mathbb{Z}}$, let  $p$ be a state of $A$ and fix an index $i \in \mathbb{Z}$, then defining $z:=y_{i}y_{i+1}y_{i+1}\ldots$ in $X_n^{\mathbb{N}}$, there exists $x \in X_n^{\mathbb{N}}$ such that the initial automaton $A_{p}: X_n^{\mathbb{N}} \to X_n^{\mathbb{N}}$ maps $x$ to $z$. 

Now let $y$, $p$, $z$ and $x$ be as in the previous  paragraph, and let $\Gamma \in X_n^k$ be a word such that the state of $A$ forced by $\Gamma$ is $p$. Let $u \in X_n^{\Z}$ be defined by $u_i u_{i+1}\ldots := x$, $u_{i-k}u_{i-k+1}\ldots u_{i-1} := \Gamma$, and $u_j:= 0$ for all $j < i-k$.

If $w \in X_n^{\Z}$ is the image of $u$ under $A$, then $w_{i}w_{i+1}\ldots = z$. Therefore for any $y \in X_n^{\Z}$ we can find an element in $A(X_n^{\Z})$ as arbitrarily close to $y$ with respect to the metric given by equation \ref{metriconX^Z}. 
\end{proof}
\begin{Remark}
Given an element, $A$, of $\shn{n}$ Proposition \ref*{automorphismiffbi-synch} gives an algorithm for determining if $A \in \hn{n}$ or if $A \in \shn{n} \backslash \hn{n}$, since we have only to check if $\overline{A}_{j}$ is a permutation for all $1 \le j \le kM(A)$, where $k$ is the synchronizing level of $A$ and $M(A)$ is quadratic in the states of $A$.
\end{Remark}

\begin{Remark} \label{primewordsmaptoprimewords}
It is a consequence of the proof of the proposition above that for $A \in \pn{n}$, $\overline{A}_{l}$ maps prime words to prime words for every $l \in \mathbb{N}$. This is because if $(\Gamma)\overline{A}_{l} = (\gamma)^r$ for $\Gamma$ a prime word, for $|\gamma| < |\Gamma|$ and $r \in \mathbb{N}$. Then either $\overline{A}_{l}: X_n^{|\gamma|}$ is not surjective and so it is not injective either, or there is a word $\delta \in X_n^{\gamma}$ such that $(\delta)\overline{A}_{l} = \gamma$. Since $\Gamma$ is a prime word it follows in either case, as in the proof of Proposition \ref{automorphismiffbi-synch}, that $A$ does not induce a homeomorphism of $X_n^{\Z}$. An alternative proof of this fact can be found in \cite{BlkYMaisANav}.
 
\end{Remark}

Proposition \ref{usefulnessofperms} indicates that if two elements $A$ and $B$ in $\hn{n}$ are such that $\overline{A}_{j}$ and $\overline{B}_{j}$ have the same disjoint cycle structure for all $j \in \mathbb{N}$ then $A$ and $B$ are likely to be conjugate. This however need not be the case, as will be seen below.  First we make the following  definition.

\begin{Definition}[Rotation]
	Let $A \in \pn{n}$ and let $l \in \mathbb{N}$. Given a prime word $\Gamma \in X_n^l$ let $C$ be the disjoint cycle of $\overline{A}_{l}$ containing $\Gamma$. Notice that $C$ consists only of prime words by Remark \ref{primewordsmaptoprimewords}. Let $1 \le s \le \mathrm{length}(C)$ be minimal in $\mathbb{N}$ such that $(\Gamma)\overline{A}_{l}^{s}$ is a rotation of $\Gamma$. Let $0 \le i < l$ be such that $\overline{A}^{s}_{l}(\Gamma)$ is the $i$\textsuperscript{th} rotation of $\Gamma$, then we say that $C$ has minimal \emph{rotation $i$} of $\Gamma$. We call the triple $(\mathrm{length}(C),s, i)_{\Gamma}$ \emph{the triple associated to $C$ for $\Gamma$}. 
\end{Definition}

\begin{lemma}
Let  $C \in \overline{A}_{l}$ be a disjoint cycle  with associated triple $(\mathrm{length}(C),s_C, r_C)_{\Gamma_0}$ for $\Gamma_0$ a prime word belonging to $C$. Then we have the following: 
\begin{enumerate}[label = (\roman*)]
\item for any other word $\Gamma$ belonging to $C$ we have: 
\[
(\mathrm{length}(C),s_C, r_C)_{\Gamma_0} = (\mathrm{length}(C),s'_C, r'_C)_{\Gamma},
\]  
\item and  $\mathrm{Length}(C) = o \cdot s_C$ where $o$ is the order of $r_C$ in the additive group $\mathbb{Z}_{l}$, if $r_C = 0$ then take $o =1$.
\end{enumerate}
\end{lemma}
\begin{proof}
Let $C= (\Gamma_0 \ldots \Gamma_{j})$ and let $(\mathrm{length}(C), s_C, r_C)_{\Gamma_0}$ be the triple associated to $C$ for $\Gamma_0$, where $\Gamma_0$ is a prime word . Then $s_C$ is minimal such that $\Gamma_{s_{C}}$ is the $r_C$\textsuperscript{th} rotation of $\Gamma_0$. Now since $\Gamma_1$ is the output of the unique loop of $A$ labelled by $\Gamma_0$, then $\Gamma_{s_{C}+1}$ is also a $r_C$\textsuperscript{th} rotation of $\Gamma_1$. This is because the unique loop of $A$ labelled by $\Gamma_{s_{C}+1}$ is the $r_{C}$\textsuperscript{th} rotation of the loop labelled  by $\Gamma_{0}$. We can now replace $C$ with the disjoint cycle $(\Gamma_1 \ldots \Gamma_{j}\Gamma_{1})$ and repeat the argument, until we have covered all rotations of $C$. This shows that the triple $(\mathrm{length}(C),s_C, r_C)_{\Gamma_1}$ is independent of the choice of $\Gamma_1$.

For the second part of the lemma, first observe that if $s_C = \mathrm{length}(C)$, then $r_C = 0$ and we are done. Therefore we may assume that $1 \le s_c < \mathrm{length}(C)$. 

Now observe that by minimality of $s_C$ and the above argument, $\Gamma_{s_{C}+ s_{C}}$ is the $2r_C$\textsuperscript{th} rotation of $\Gamma_0$,  moreover no $\Gamma_k$ for $s_C< k < 2s_C$ is a rotation of $\Gamma_0$. Notice that $r_C$ has finite order in the additive group $\mathbb{Z}_{l}$. Let $o$ be the order of $r_C$. Then $\Gamma_{os_{C}}$ is the  $or_C$\textsuperscript{th} rotation of $\Gamma$ which is just $\Gamma$. Moreover by minimality of $s_C$, and repetitions of the argument in the previous paragraph, $o$ is minimal such that $\Gamma_{os_{C}} = \Gamma_0$. However by the first part of the lemma, we must also have $\overline{A}_{l}^{os_{C}}(\Gamma_k) = \Gamma_k$ $1\le k \le j$. Minimality now ensures that $os_C = j$.
\end{proof}

As a consequence of the remark above for a given disjoint cycle $C \in \overline{A}_{l}$ we shall simple refer to $(\mathrm{length}(C),s_C, r_C)$ as \emph{the triple associated to $C$}.

\begin{Definition}[Spectrum]
Let $A \in \pn{n}$, and let $k \in \mathbb{N}$. For each triple $(L_C, S_C, T_C)$ associated to a disjoint cycle of prime words in the disjoint cycle structure of $\overline{A}_{k}$, let $d_C$ denote the multiplicity with which it occurs as we consider all such triples associated to the disjoint cycles of $\overline{A}_{k}$. Then define $Sp_{k}(A):= \{(k,d_C,(L_C,S_C,T_C))\}$ as $C$ runs over all disjoint cycles of $\overline{A}_{k}$. Define $Sp(A) := \bigcup_{k \in \mathbb{N}}Sp_{k}(A)$. 
\end{Definition}

\begin{Theorem}\label{spectrumisconjinv}
Let $A \in \pn{n}$, and let $k \in \mathbb{N}$, then $Sp_{k}(A)$ is a conjugacy invariant of $A$ in $\pn{n}$. 
\end{Theorem}
\begin{proof}
Let $C$ be a cycle in the disjoint cycle structure of $\overline{A}_{k}$ and let  $(L_C, S_C, T_C)$ be its associated triple. Let $J \in \pn{n}$ be arbitrary and invertible.

That $L_C$ is preserved under conjugation by $J$ follows from Proposition \ref{usefulnessofperms}, and standard results about permutation groups.

That $S_C$ is preserved under conjugation is a consequence of the fact that  $J \in \pn{n}$. To see this first suppose that $C = (\Gamma_1 \ldots \Gamma_{j})$ for some $j \in \mathbb{N}$. Let $\Delta_i = (\Gamma_i)\overline{J}_{k}$. Then $(\Delta_1 \ldots \Delta_{j})$ is a cycle of $\overline{J}_{k}^{-1}\overline{A}_{k}\overline{J}_{k}$. Since $\Delta_i$ is the output of the unique loop of $J$ labelled by $\Gamma_i$ ($1 \le i \le j$), and since $S_C$ is minimal so that $\Gamma_{S_{C}}$ is a rotation of $\Gamma_1$, then $S_C$ is also the minimal position so that $\Delta_{S_{C}}$ is a rotation of $\Delta_1$.

That $T_C$ is preserved under conjugation is once more a consequence of the fact that $J \in \pn{n}$. Let $\Gamma_i$ and $\Delta_i$ for $1\le i \le j$ be as in the previous paragraph. Since $\Gamma_{S_{C}}$ is the $T_C$\textsuperscript{th} rotation of $\Gamma_1$, then as $\Delta_1$ is the output of the unique loop of $J$ labelled by $\Gamma_1$, $\Delta_{S_{C}}$ is the $T_C$\textsuperscript{th} rotations of $\Delta_1$. 
\end{proof}

\begin{corollary}
Let $A \in \pn{n}$, then $Sp(A)$ is a conjugacy invariant of $A$ in $\pn{n}$. 
\end{corollary}

It is possible to construct elements $M \in \hn{n}$ for which $Sp(M) \ne Sp(M^{-1})$. Clearly $\overline{M}_{j}$ and $\overline{(M^{-1})}_{j}$ have the same disjoint cycle structure for all $j \in \N$.

\section{The structure of the dual automaton for elements of \texorpdfstring{$\spn{n}$}{lg}}\label{tools}

In this section we shall start to introduce the tools needed to understand the order problem in the group $\hn{n}$. However as the techniques are applicable to $\spn{n}$ and will be relevant in later sections, we shall work in this set for majority of this section restricting our attention to $\hn{n}$ only when necessary.  It is standard in the literature to tackle the order problem by investigating the structure of the dual automaton, see for instance \cite{KlimmannPicantinSavchuk, AKLMP}, and this is what we do below.

\subsection{The dual at level \texorpdfstring{$k$}{Lg}}

Let $A = \gen{X,Q,\pi,\lambda}$ be a synchronous  transducer and let $k \in \mathbb{N}$.

We form the level $k$ dual, 

\[A^{\vee}_{k} = \gen{\dual{X}_k,\dual{Q}_k,\dual{\pi}_k,\dual{\lambda}_k}\]
 of $A$ as follows. The state set $\dual{Q}_k$ of $\dual{A}_{k}$ is the set of all words of length  $k$ in the input alphabet $X$. This dual automaton has its input alphabet equal to its output alphabet and they are both equal to  $\dual{X}_k:= Q$ the set of states of $A$. The transition function $\dual{\pi}_k$ is defined as follows: for states $q,q'\in Q$, and $\Gamma,\Gamma'\in \dual{Q}_k$ we have: 
\begin{enumerate}
\item  $\dual{\pi}_k(q,\Gamma) = \Gamma'$ if and only if $\lambda(\Gamma,q) = \Gamma'$, and
\item  $\dual{\lambda}_k(q, \Gamma) = q'$ if and only if $\pi(\Gamma, q) = q'$ .
\end{enumerate}

We observe that $\dual{A}_{k+1} = \dual{A}_{k} * \dual{A}_{1}$. For suppose that $\Gamma i$  a word of length  $k+1$ is a state in the dual, and $q$ is any state symbol of $A$ such that after reading $q$ from $\Gamma i$ in $\dual{A}_{k+1}$ we are in state $ \Delta j$ and the output is $p$. Then in $A$ we  have $\pi(\Gamma i,q)= p$ and $\lambda(\Gamma i,q) = \Delta j$. We can break up this transition into two steps. Suppose $\pi(\Gamma,q) = p'$, then we have $\lambda(\Gamma,q) = \Delta$, $\lambda(i,p') = j$ and $\pi(i,p') = p$. Hence in $\dual{A}_k$ we  read $q$ from $\Gamma$ and transition to $\Delta$ and $p'$ is the  output $p'$. Moreover in $\dual{A}_{1}$ we read $p'$ from $i$ and transition to $j$ with output $p$. Therefore the state $(\Gamma, i)$ of $\dual{A}_{k}*\dual{A}_{1}$, is such that  we  read $q$ from this state and transition to the state $(\Delta,j)$ and the output produced is $p$.

The following definition gives a tool which connects the the synchronizing level of powers of an element of $\spn{n}$ to a  property of the dual automaton. First we introduce some notation. 

Let $A \in \spn{n}$ be a transducer and $t \in Q_{A}$ be a state of $A$. Let $k \in \N$ be the synchronizing length of $A$, then denote by $W_{t}$ the synchronizing words for the state $t$. That is, $W_{t}$ is the set of all words $\Gamma \in X_{n}^{\ast}$ of length at least $k$  such that $\pi_{A}(\Gamma,q)  = t$ for some (and so any) state of $A$.

\begin{Definition}[Splits]
Let $A$ be an element of $\spn{n}$, with synchronizing level $k$. Then, for $r \ge k$, we say that $\dual{A}_{r}$ \emph{splits} if  there is a word $\Gamma \in X_n^{r}$, elements $(p_1, p_2, \ldots, p_l), (q_1, q_2, \ldots, q_{l}), (s_1, s_2, \ldots, s_l) \in Q_{A}^{l}$, where  $\Gamma \in W_{s_1}$, and distinct elements $t_1,t_2 \in Q_{A}$, such that the sequences $\Gamma_1, \Gamma_2, \ldots, \Gamma_{l}$ and $\Lambda_{1},\Lambda_{2}, \ldots, \Lambda_{l}$ defined by $\Gamma_1 = \lambda_{A}(\Gamma, p_1)$, $\Lambda_{1}= \lambda_{A}(\Gamma, q_1)$  and for $1 < i \le l$, $\Gamma_i = \lambda_{A}(\Gamma_{i-1}, p_i)$ and $\Lambda_{i} =  \lambda_{A}(\Lambda_{i-1}, q_i)$, satisfy: $\Gamma_i, \Lambda_{i} \in W_{s_{i+1}}$ for all $1 \le i \le l-1$, $\Gamma_{l} \in W_{t_1}$ and $\Lambda_{l} \in W_{t_2}$. In other words, the following picture depicting the transitions in $\dual{A}_{r}$ at the state $\Gamma$ is valid:
\begin{figure}[H]
\begin{center}
\begin{tikzpicture}[shorten >=0.5pt,node distance=3cm,on grid,auto] 
   \node[state] (q_0) [yshift = 1.5cm]   {$\Gamma$}; 
   \node[state] (q_1) [yshift=3cm, xshift=1.5cm] {$\Gamma_1$}; 
   \node[state] (q_2) [yshift=0cm,xshift=1.5cm] {$\Lambda_1$};
   \node[state] (q_a) [yshift=0cm, xshift= 3.5cm]{$\Lambda_2$};
   \node[state] (q_b) [yshift=3cm, xshift= 3.5cm]{$\Gamma_2$};
    \node[state] (q_3) [yshift=0cm, xshift=5cm] {$\Lambda_{l-1}$};
    \node[state] (q_4)  [yshift=3cm,xshift=5cm] {$\Gamma_{l-1}$};
    \node[state] (q_6) [yshift=0cm, xshift=7cm] {$\Lambda_l$};
    \node[state] (q_5)  [yshift=3cm,xshift=7cm] {$\Gamma_l$};
    \node		 (q_7)  [yshift=3cm,xshift=8.5cm] {};
    \node		 (q_8)  [yshift=0cm,xshift=8.5cm] {};
    \path[->] 
    (q_0) edge node {$q_1|s_1$} (q_1)
          edge node[swap] {$p_1|s_1$} (q_2)
    (q_1) edge node {$q_2|s_2$} (q_b)
    (q_2) edge node[swap]{$p_2|s_2$} (q_a);      
    \path[-,dotted]      
     (q_a) edge node{} (q_3)
     (q_b) edge node{} (q_4);
     \path[->]
     (q_3) edge node [swap] {$p_l| s_l$} (q_6)
     (q_4) edge node {$q_l|s_l$} (q_5)
     (q_5) edge node  {$\ast| t_1$} (q_7)
     (q_6) edge node [swap] {$\sharp| t_2$} (q_8);
\end{tikzpicture}
\end{center}
\caption{A split; the symbols $\ast$ and $\sharp$ represent arbitrary elements of $Q_{A}$.}
\label{split}
\end{figure}

We say that the $l$-tuples $(\Skip{p_1}{,}{p_l})$ and $(\Skip{q_1}{,}{q_l})$ \emph{split} $\dual{A}_{r}$. We shall call $\{p_1, q_1\}$ the \emph{top of the split}, $\{t_1, t_2\}$ the \emph{bottom of the split}, and the triple $((q_1 \ldots, q_l), (p_1, \ldots, p_l), \Gamma)$ a \emph{split} of $\dual{A}_{r}$.
\end{Definition}

\begin{Definition}\label{bottomdependsonlyontop}
Let $A$ be an element of $\spn{n}$, with synchronizing level $k$. Let $r \ge k$ and let $((q_1 \ldots, q_l), (p_1, \ldots, p_l), \Gamma)$ be a split of $\dual{A}_{r}$ for $\Gamma \in X_n^{r}$ and $(q_1 \ldots, q_l), (p_1, \ldots, p_l) \in Q_A^{l}$. Let $\{t_1, t_2\}$ be the bottom of this split. Then we say that \emph{the bottom of the split $((q_1 \ldots, q_l), (p_1, \ldots, p_l), \Gamma)$ depends only on the top} if for any other tuples $U_1, U_2 \in Q_A^{l-1}$ we have that $((q_1,U_1), (p_1, U_2), \Gamma)$ is also a split with bottom $\{t_1, t_2\}$ and, for any $u, u' \in Q$, $\pi_{Al}(\Gamma,(p_1, \ldots, p_l, u)) = \pi_{Al}(\Gamma, (p_1, U_2, u'))$ and $\pi_{Al}(\Gamma,(q_1, \ldots, q_l, u)) = \pi_{Al}(\Gamma, (q_1, U_2, u'))$. The last condition means that if $\lambda_{Al}(\Gamma,(q_1,\ldots,q_l)) \in W_{t_1}$ then so also is $\lambda_{Al}(\Gamma, (q_1, U_1))$ and likewise for $(p_1,\ldots, p_l)$, $(P_1, U_2)$ and $W_{t_2}$.
\end{Definition}

\begin{Definition}
For a transducer $A$, we define the \emph{ r splitting length of $A$} (for $r$ greater than or equal to the minimal synchronizing length) to be minimal $l$ such that there is a pair of $l$-tuples of states which split $\dual{A}_{r}$. If there is no such pair the we set the r splitting length of $A$ to be $\infty$.
\end{Definition}

\begin{Remark}\label{badpairsmakesense}
Let $A$ be a transducer with minimal $r$ splitting length $l < \infty$, by minimality of $l$ it follows that for a given pair in $Q^{l}\times Q^l$ which  splits $\dual{A}_{r}$, then the bottom of the split depends only on the top. Therefore the top and bottom of the split have cardinality two. In particular, for any split whose bottom depends only on its top, the top and bottom of the split both have cardinality two.
\end{Remark}
\begin{Remark}\label{rsplittinglengthzeromeansr+1splittinglenghtzzero}
Let $A$ be a transducer such that the minimal $r$ splitting length of $A$ is infinite for some $r$ then the minimal $r+1$ splitting length of $A$ is also infinite.
\end{Remark}

The following lemma demonstrates that for $A \in \spn{n}$  an $r >2$ the $r$ splitting length of $A$ is bigger than the $r-1$ splitting length of $A$.  

\begin{lemma}\label{minimalsplittinglengthstrictlyincreasing}
Let $A \in \spn{n}$ be synchronizing at level $k$, and suppose that the $mk$ splitting length of $A$ is finite for $m \in \N$, $m >0$, then the $(m+1)k$ splitting length of $A$ is  strictly greater than the $mk$ splitting length of $A$.
\end{lemma}
\begin{proof}
Let $A$, $m$ and $k$ be as in the statement of the lemma. Suppose that $A$ has $mk$ splitting length $l$.  It suffices to show that for any word $\Gamma \in X_n^{(m+1)k}$, and any $l+1$-tuple $P$ in $Q_A^{l+1}$, the output of $P$ through $\Gamma$ depends only on $\Gamma$.

First we set up some notation. Let $A^{j} := \gen{X_n, Q_A^{j}, \lambda_{j}, \pi_{j}}$ and let $\dual{A}_{j} :=\gen{Q_A, X_n^{j}, \dual{\lambda}_{j}, \dual{\pi}_{j}}$ for $j \in \mathbb{N}$. For a word $\gamma \in X_n^{k}$ let $q_{\gamma}$ denote the state of $A$ forced by $\Gamma$.

Now since $A$ has $mk$ splitting length $l$, it follows that for any $P:= (p_1, \ldots, p_l)$ and $T:= (t_1, \ldots, t_l)$ in $Q_A^{l}$ and $\Gamma \in X_n^{mk}$ we have that $\dual{\lambda}_{mk}(P, \Gamma) = \dual{\lambda}_{mk}(T, \Gamma)$. By definition of the dual, $\dual{\lambda}_{mk}(P, \Gamma) = \pi_{l}(\Gamma, P)$. 

Now let $\gamma \in X_n^{k}$ be arbitrary and let $p \in Q_A$ and $P \in Q_A ^{l}$ also be arbitrary. Consider $\dual{\lambda}_{(m+1)k}(Pp, \Gamma \gamma)$, we have:

\[
\dual{\lambda}_{(m+1)k}(Pp, \Gamma \gamma) = \pi_{l+1}(\Gamma \gamma, Pp) =\pi_{l}(\gamma,\pi_{l}(\Gamma, P))\pi_{1}(\lambda_{l}(\gamma,\pi_{l}(\Gamma, P) ), \pi_{1}(\lambda_{l}(\Gamma, P),p))
\] 

However observe that since  $A$ is synchronizing at level $k$ that the suffix  $\pi_{1}(\lambda_{l}(\gamma,\pi_{l}(\Gamma, P) ), \pi_{1}(\lambda_{l}(\Gamma, P),p))$ depends only on $\lambda_{l}(\gamma,\pi_{l}(\Gamma, P) )$. However since $\dual{A}_{mk}$ has minimal splitting length $l$ we have that $\pi_{l}(\Gamma, P)$ depends only on $\Gamma$. Therefore we have that $\dual{\lambda}_{(m+1)k}(Pp, \Gamma \gamma)$ depends only on $\Gamma\gamma$.

\end{proof}

\begin{Remark}
It follows from the lemma above that if $A \in \spn{n}$ is synchronizing at level $k$ then the $mk$ splitting length of $A$, if it is finite, is at least $m$ for $m \in \N$, $m >0$.
\end{Remark}

The following lemma shows that the minimal splitting length is connected with the synchronizing level of powers of a transducer.

\begin{lemma} \label{minsplittinglengthandsynchlevel}
Let $A$ be  a transducer with synchronizing level less than or equal to $k$, then if $A$ has $k$ splitting length $l$, then $\min(\core(A^{l+1}))$ has minimal  synchronizing level $m \ge k+1$.
\end{lemma}
\begin{proof}
Let $l$ be as in the statement of the lemma.
Now consider $\core(A^l)$. The states of $\core(A^l)$ will consist of all length $l$ outputs of $\dual{A}_{k}$. Moreover by choice of $l$, $\core(A^l)$ is also synchronizing at level $k$.

Let $\Gamma$ be a word for which there is a split $( (q_1, \ldots, q_l), (p_1, \ldots, p_l),\Gamma)$  of the minimal length  $l$ giving rise to the picture below:

\begin{figure}[h!]
\begin{center}\label{figure 1}
\begin{tikzpicture}[shorten >=0.5pt,node distance=3cm,on grid,auto] 
   \node[state] (q_0) [yshift = 1.5cm]   {$\Gamma$}; 
   \node[state] (q_1) [yshift=3cm, xshift=1.5cm] {$\Gamma_1$}; 
   \node[state] (q_2) [yshift=0cm,xshift=1.5cm] {$\Lambda_1$};
   \node[state] (q_a) [yshift=0cm, xshift= 3.5cm]{$\Lambda_2$};
   \node[state] (q_b) [yshift=3cm, xshift= 3.5cm]{$\Gamma_2$};
    \node[state] (q_3) [yshift=0cm, xshift=5cm] {$\Lambda_{l-1}$};
    \node[state] (q_4)  [yshift=3cm,xshift=5cm] {$\Gamma_{l-1}$};
    \node[state] (q_6) [yshift=0cm, xshift=7cm] {$\Lambda_l$};
    \node[state] (q_5)  [yshift=3cm,xshift=7cm] {$\Gamma_l$};
    \path[->] 
    (q_0) edge node {$q_1|s_1$} (q_1)
          edge node[swap] {$p_1|s_1$} (q_2)
    (q_1) edge node {$q_2|s_2$} (q_b)
    (q_2) edge node[swap]{$p_2|s_2$} (q_a);      
    \path[-,dotted]      
     (q_a) edge node{} (q_3)
     (q_b) edge node{} (q_4);
     \path[->]
     (q_3) edge node [swap] {$p_l| s_l$} (q_6)
     (q_4) edge node {$q_l|s_l$} (q_5);
\end{tikzpicture}
\end{center}
\caption{A minimal split}
\end{figure}
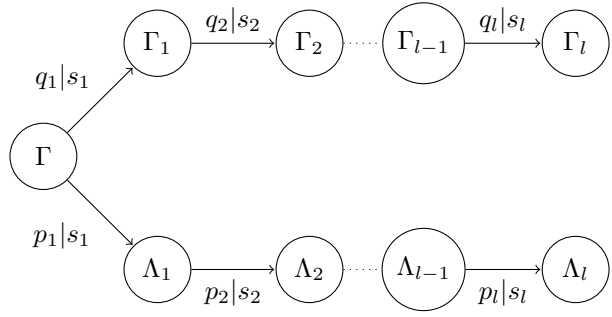
where $\Gamma_l \in W_{t_1}$ and $\Lambda_l \in W_{t_2}$ for distinct states $t_1$ and $t_2$.

 We now consider $ B:=\core(A \times \core(A^{l}))$. It is easy to see that there are states in $B$ of the form $(p_1, P)$, $(q_1,Q)$ for appropriate $P,Q \in \core(A^l)$. Therefore in $B$  when we have read $\Gamma$ through $(p_1,P)$, we are in state $(\Skip{s_1}{,}{s_l},t_1)$, and when we have read $\Gamma$ through state $(q_1,Q)$ we go to state $(\Skip{s_1}{,}{s_l},t_2)$. Since $t_1 \ne t_2$ these states are not $\omega$ equivalent. This concludes the proof.
\end{proof}

We begin to make connections to the order problem.

\begin{lemma}
Let $A \in \hn{n}$. Then either $A$ has finite order or for all $k \in \mathbb{N}$ there is an $N \in \mathbb{N}$ such that for all $m \in \N$ $A^{Nm}$ is bi-synchronizing at level greater than $k$. Moreover this $N$ is computable.
\end{lemma}
\begin{proof}
 Suppose that $A$ does not have finite order. Since $A$ has infinite order then $\dual{A}_{j}$ splits for every $j \in \mathbb{N}$, therefore there is an $m_0 \in \mathbb{N}$ and $N_1 \in \mathbb{N}$ such that $\min\core(A^{N_1})$ is bi-synchronizing at level $m_0$. In order to simplify the notation we shall identify  $A^{N_1}$ with the minimal, core transducer $\min \core(A^{N_1})$. 
 
 Now consider the permutation $\overline{A}^{N_1}_{m_0}$, we shall assume that it is written as a product of disjoint cycles. Let $d$ be the order of this permutation. Let $\Gamma \in X_n^{m_0}$, then $\Gamma$ belongs to a cycle $(\Gamma_1 \Gamma_2 \ldots \Gamma_{d_{\Gamma}})$ where  $\Gamma_1 := \Gamma$ and $d_{\Gamma}| d$. Let $e_{\Gamma} = d/d_{\Gamma}$.  To this cycle there is associated a tuple of states $(q_{\Gamma_1}, q_{\Gamma_2}, \ldots, q_{\Gamma_{d_{\Gamma}}})$ where $q_{\Gamma_i}$ is the state of $A^{N_1}$ forced by $\Gamma_i$ for $1 \le i \le d_{\Gamma}$.  Now  observe that $(q_{\Gamma_1}, q_{\Gamma_2}, \ldots, q_{\Gamma_{d_{\Gamma}}})$ is a state of $A^{N_1 d_{\Gamma}}$, moreover since $\lambda_{N_1}(\Gamma_i, q_{\Gamma_i}) = \Gamma_{i+1}$ for $1 \le i < d_{\Gamma}$, and $\lambda_{N_1}(\Gamma_{d_{\Gamma}}, q_{\Gamma_{d_{\Gamma}}}) = \Gamma_{1}$, then we have that $$\pi_{N_1d_{\Gamma}}(\Gamma_1, (q_{\Gamma_1}, q_{\Gamma_2}, \ldots, q_{\Gamma_{d_{\Gamma}}})) = (q_{\Gamma_1}, q_{\Gamma_2}, \ldots, q_{\Gamma_{d_{\Gamma}}}).$$ Now let $t_{\Gamma}$ be the order of the permutation induced by   $(q_{\Gamma_1}, q_{\Gamma_2}, \ldots, q_{\Gamma_{d_{\Gamma}}})^{e_{\Gamma}}$ on  $X_n^{m_0}$. 
 
 Let $t$ be the lowest common multiple of the set  $\{t_{\Gamma} | \gamma \in X_n^{m_0}\}$. In order to keep the notation concise let $P_{\Gamma}$ represent the state $(q_{\Gamma_1}, q_{\Gamma_2}, \ldots, q_{\Gamma_{d_{\Gamma}}})^{e_{\Gamma}t}$. Notice that $P_{\Gamma}$ acts locally as the identity for all $\Gamma \in X_n^{m_0}$. Moreover $P_{\Gamma}$ is the state of $A^{N_1 dt}$ such that $\pi_{N_1 dt}(\Gamma, P_{\Gamma}) = P_{\Gamma}$.
 
 Now let $N = N_1 dt$, and let $m \in N$. Suppose that $A^{mN}$, (where again $A^{mN}$ is identified with  $\min\core(A^{mN})$) is synchronizing at level $m_0$. Since $m \ge N$ we may write $m = rN + s$ for some $r \in \N$ and $ 0 \le s < N$. Therefore states of $A^{m}$ can be identified with states $P^{m}$ for $P$ a state of $A^{N}$, since for a word $\Gamma \in X_{n}^{m_0}$, the state $P_{\Gamma}$ satisfies $\pi_{N}(\Gamma, P_{\Gamma}) = \Gamma$ and $\lambda_{N}(\Gamma,P_{\Gamma}) = \Gamma$. Furthermore, since $A^{mN}$ is synchronizing at level $m_0$, it must be the case that the state of  $A^{mN}$ forced by a word $\Gamma \in x_{n}^{m_0}$ is precisely $P_{\Gamma}^{m}$.  Now as all of these states are locally identity, it follows that $A^{mN}$ is the identity. However, this is contradicts the initial assumption that $A$ does not have finite order. Therefore $A^{mN}$ must be synchronizing at level greater than $m_0$.
\end{proof}

Now suppose that $A \in \spn{n}$ and the semigroup $\gen{A} :=  \{ A^{i} | i \in \N \}$ is finite. Notice that if $A \in \shn{n}$ and the semigroup $\gen{A}$ is finite then it coincides with the group generated by $A$. The next result demonstrates that in the case where the semigroup $\gen{A}$ is finite, there must be some $j \in \N$, $j >0$ for which the $j$ splitting length of $A$ is infinite.

\begin{lemma}\label{lemma 1}
Let $A \in \spn{n}$ be synchronizing at level $k$. Suppose that the semigroup $\gen{A}$ is finite, and that $j$ is the maximum of the minimal synchronizing level of the elements of $\gen{A}$. Then $A$ has infinite $j$ splitting length. 
\end{lemma}
\begin{proof}
This is a consequence of Lemma \ref{minsplittinglengthandsynchlevel}. Since if $A$ has $j$ splitting length $l$, then by Lemma \ref{minsplittinglengthandsynchlevel} $\min(\core(A^{l+1}))$ has minimal synchronizing level $j+1$, which is a contradiction.
\end{proof}

\begin{Remark} \label{remark 1}
The above means that we can partition $\dual{A}_{j}$ into components $D_1, \ldots, D_i$ and for each component there is a pair of words $W_{i,1}$ and $W_{i,2}$ in the states of $A$ such that for any input into a state of a component of $D_i$, the only possible output has the form $u(W_{i,2})^{l}v$ where $u$ is a possibly empty suffix of $W_{i,1}W_{i,2}$, and $v$ is a possibly empty prefix of $W_{i,2}$.
\end{Remark}

Below we illustrate Lemma~\ref{lemma 1} with some examples.

\begin{example}
Consider the transducer $C$ from example \ref{example 1}. This is a transducer of order $3$, in particular, it is a conjugate of the single state transducer which can be identified with the permutation $(0,1,2)$.
\begin{figure}[h!]
\begin{center}
\begin{tikzpicture}[shorten >=0.5pt,node distance=3cm,on grid,auto] 
    \node[state] (q_0)   {$q_0$}; 
   \node[state] (q_1) [below left=of q_0] {$q_1$}; 
   \node[state] (q_2)[below right=of q_0] {$q_2$}; 
    \path[->] 
    (q_0) edge [in=105,out=75,loop] node [swap] {$0|1$} ()
          edge [out=335,in=115]  node [swap] {$2|0$} (q_2)
          edge [out=185,in=85]  node [swap]  {$1|2$} (q_1)
    (q_1) edge [out=65,in=205]  node [swap]  {$0|0$} (q_0)
          edge [in=240,out=210, loop] node [swap] {$1|2$} ()
          edge [out=330,in=210]  node [swap] {$2|1$} (q_2)
    (q_2) edge[out=95,in=355]  node [swap] {$1|1$} (q_0)
          edge[out=195,in=345]  node [swap] {$0|2$} (q_1) 
          edge [in=330,out=300, loop] node [swap] {$2|0$} ();
\end{tikzpicture}
\end{center}
\caption{An element of order 3 in $\hn{n}$}
\end{figure}

This transducer, as noted before, is bi-synchronizing at the second level. The level 3 dual has 27 nodes and so we shall not give this below. However utilising either the AAA package or the AutomGrp package \cite{AutomGrp1.2.4} in GAP \cite{GAP4}, together with (in AutomGrp) the function ``MinimizationOfAutomaton( )'' which returns an $\omega$-equivalent automaton, applied to the third power of the dual automaton, we get the following result:
\begin{figure}[H]
\begin{center}
\begin{tikzpicture}[shorten >=0.5pt,node distance=3cm,on grid,auto] 
   \node[state] (q_0)   {$a_0$}; 
   \node[state] (q_1) [below left=of q_0] {$a_1$}; 
   \node[state] (q_2) [below right=of q_0] {$a_2$}; 
    \path[->] 
    (q_0) edge [bend right] node [swap] {$q_0|q_0, \ q_1|q_0,\ q_2|q_0$} (q_1)
    (q_1) edge [bend right] node [swap]  {$q_0|q_1, \ q_1|q_1, \ q_2|q_1 $} (q_2)
    (q_2) edge [bend right] node [swap] {$q_0|q_2,\ q_1|q_2, \ q_2|q_2 $} (q_0);
\end{tikzpicture}
\end{center}
\caption{The level 3 dual of $C$.}
\end{figure}
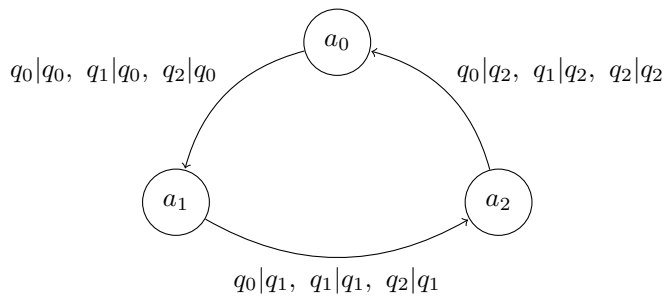

Since the original transducer $C$ has order 3 we can see from its level $3$ dual above that the states in the core will be cyclic rotations of $(q_0,q_1,q_2)$ all of which are locally identity. 
\end{example}

We illustrate another example below, but now with an element  of order 2.

\begin{example}
Consider the automaton of order two given below.
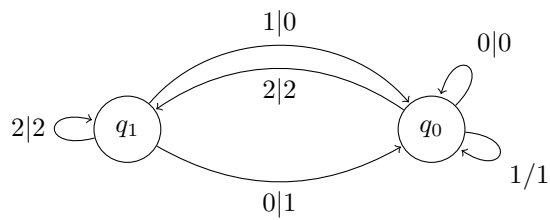
\begin{figure}[H]
\begin{center}
\begin{tikzpicture}[shorten >=0.5pt,node distance=3cm,on grid,auto] 
   \node[state] (q_0)   {$q_0$}; 
   \node[state] (q_1) [xshift=-4cm,yshift=0cm] {$q_1$};  
    \path[->] 
    (q_0) edge [in=75,out=45,loop] node[swap] {$0|0$} ( )
          edge [in=325,out=355,loop] node  {$1/1$} ( )
          edge [out=145,in=35] node  {$2|2$} (q_1)
    (q_1) edge [bend right] node [swap]  {$0|1$} (q_0)
          edge [out=50,in=130] node  {$1|0$} (q_0)
          edge [loop left] node [swap]  {$2|2$} ( );
\end{tikzpicture}
\end{center}
\caption{An element of order 2}
\end{figure}
This automata is synchronizing on the first level. We give the dual below.

\begin{figure}[h!]
\begin{center}
\begin{tikzpicture}[shorten >=0.5pt,node distance=3cm,on grid,auto] 
   \node[state] (q_0)   {$0$}; 
   \node[state] (q_1) [below left=of q_0] {$1$}; 
   \node[state] (q_2) [below right=of q_0] {$2$}; 
    \path[->] 
    (q_0) edge [bend left] node {$q_1|q_0$} (q_1)
          edge [in=45,out=75,loop] node {$q_0|q_0$} ()
    (q_1) edge [bend left] node {$q_1|q_0$} (q_0)
          edge [out=225,in=255,loop] node[swap] {$q_0|q_0$} ()
    (q_2) edge [loop above] node {$q_0|q_1,\ q_1|q_1$} ( );
\end{tikzpicture}
\end{center}
\caption{The level 1 dual.}
\end{figure}
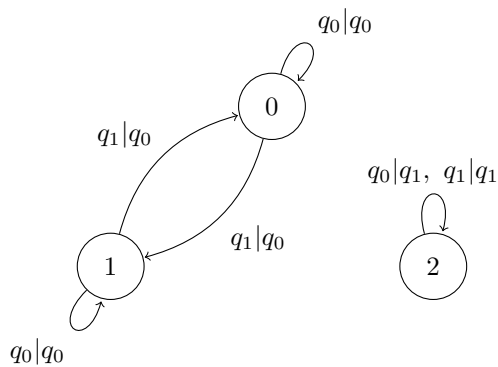

It is easy to see that the states $0$ and $1$ are $\omega$-equivalent, and can be identified to a single node that produces $q_0$ for all inputs. The states in the core of the square of the original automaton will be $(q_0,q_0)$ and $(q_1,q_1)$.
\end{example}

For a transducer of finite order, $A$, as above, we have the following result about the semigroup $\gen{\dual{A}}$.
\begin{Theorem} \label{lemma 2}
Let $A \in \spn{n}$ be synchronizing at level $k$ and suppose that the semigroup  $\gen{A}$ is finite and let $j$ be the maximum of the minimal synchronizing levels of the elements of $\gen{A}$. Then $\dual{A}_{j} = (\dual{A})^{j}$ is a zero in $\gen{\dual{A}}$, the semigroup generated by $\dual{A}$.
\end{Theorem}

\begin{proof}
It suffices to show that $\dual{A}_{j}$ is a right zero of the semigroup since the semigroup $\gen{\dual{A}}$ is commutative.

Our strategy shall be to to show that for any state $\Gamma$ of $\dual{A}_{j}$ and any state $x$ of $\dual{A}$, that the state $\Gamma x$ of $\dual{A}_{j+1}$ is $\omega$-equivalent to a state of $\dual{A}_{j}$. To this end let  $\Gamma  \in X_n^{j}$ be a word of length $j+1$. By Lemma \ref{minimalsplittinglengthstrictlyincreasing} and Remark \ref{remark 1} there is a pair of words $W_{1,\Gamma_1}$ and $W_{2, \Gamma_1}$ such that any input read from $\Gamma_1$ has output of the form $W_1 (W_2)^{l} v$ for $l \in \mathbb{N}$ and $v$ a prefix of $W_2$, otherwise the output is a prefix of $W_1$. Let $\gamma \in X_n^{j}$ be the length $j$ suffix of $\Gamma_{1}$. Observe that the outputs of the state $\gamma$ of $\dual{A}_{j}$ must also all be of the form $W_1 (W_2)^{l} v$ for $l \in \mathbb{N}$ and $v$ a prefix of $W_2$, otherwise the output is a prefix of $W_1$ and the output depends only on the length of the input word. Therefore we must have that $\Gamma$ and $\gamma$ are $\omega$-equivalent.

On the other hand, given a word $\gamma \in X_n^j$, then a similar argument demonstrates that the state  $x \gamma$ for any $x \in X_n^{j}$ is $\omega$-equivalent to $\gamma$.
\end{proof}

The next lemma observes that Lemma \ref{lemma 1} gives a complete characterisation of elements of $\T{H}_{n}$ with finite order.

\begin{proposition}\label{lemma 1.3}
Let $A$ be an element of $\T{P}_{n}$ and suppose $A$ is synchronizing at level $k$. Then the semigroup $\gen{A}$ generated by  $A$  is finite if and only if there is some $m \in \mathbb{N}$ such that the following holds:
\begin{enumerate}[label = (\roman*)]
\item $\dual{A}_{m}$ is a zero of the semigroup $\langle \dual{A} \rangle$  
\item $\dual{A}_{m}$ is $\omega$-equivalent to a transducer with $r$  components $D_i$ $1\le i \le r$. For each component $D_i$ there is a fixed pair of words $w_{i,1}, w_{i,2}$ (in the states of $A$) associated to $D_i$ such that whenever we read any input from a state in the $D_i$, the output is of the form $w_{i,1}w_{i,2}^{l}v$ for $l \in \mathbb{N}$ and $v$ a prefix of $w_{i,2}$ or has the from $u$ for some prefix $u$ of $w_{i,1}$. Moreover the output depends only on the state in the component $D_i$ from which the input is processed. \end{enumerate}
\end{proposition}
\begin{proof}

$\Rightarrow$: This direction follows from Lemma \ref{lemma 1}, Remark \ref{remark 1} and Theorem \ref{lemma 2}.

\hspace{0.6cm}$\Leftarrow$: Assume that $\dual{A}_{m}$ has $r$ components and let $w_{i,1}$ and $w_{i,2} \ 1\le i\le r$ be the pair of words in the states of $A$ associated with each component $D_i$. To see that the semigroup $\gen{A}$ is finite observe that the assumptions that $\dual{A}_{m}$ is a zero of the semigroup $\gen{\dual{A}}$ and that the output depends only on which state in the component $D_i$ of $\dual{A}_{m}$ we begin processing inputs means that $A^{l}$ is synchronizing at level  $m$ for all $l \in \mathbb{N}$. Therefore the set $\{A^{l}| l \in \N \}$ is finite, since there are only finitely many automata which are synchronizing at level $l$.
\end{proof}

\begin{Remark}
In the case where $A$ is an element of $\hn{n}$ in the above proposition, then each component $D_i$ is a strongly connected component. In particular one of $w_{i,1}$ or $w_{i,2}$ will be the empty string for any component $D_i$.
\end{Remark}

\begin{Remark}
Given a transducer $A \in \spn{n}$ which is synchronizing at level $k$, by increasing the alphabet size to $n^{k}$, one can identify $A$ with an element of $A'$ of $\spn{n^{k}}$. It is an easy exercise to verify that $\dual{(\dual{A}_{k})} = A'$. Therefore  $\dual{A'} = \dual{A}_{k}$, hence by Proposition \ref{lemma 1.3}, $A'$ has finite order if and only if $A$ has finite order. In order to simplify calculations, we shall often assume that our transducer if bi-synchronizing at level 1.
\end{Remark}

\subsection{Applications to the order problem}

On the surface Lemma \ref{lemma 1.3} appears to reduce the order problem in $\T{H}_{n}$ to an equivalent problem of deciding whether the semigroup generated by the dual has a zero. However a consequence of the above lemmas (in particular Lemma \ref{lemma 1}), is that for certain transducers where the dual at the bi-synchronizing level has some property we introduce below, we are able to conclude that this transducer will be an element of infinite order.
We shall need a few definitions first. Once more we shall make these definitions for elements of  $\spn{n}$ restricting to $\hn{n}$ as required.

\begin{Definition}[Bad pairs]
Let $A \in \spn{n}$ be a transducer which is synchronizing at level $k$, and let $r \ge k$. Let $l$ be the minimal splitting length of $\dual{A}_{r}$. Let $\mathscr{B}_{r}$ be the set of tops of those pairs $((q_1 \ldots, q_m), (p_1, \ldots, p_m))$ of $m$ tuples, $m \ge l$, which split $\dual{A}_{r}$ and for which there is a split $((q_1 \ldots, q_m), (p_1, \ldots, p_m), \Gamma)$ such that the bottom of the split depends only on the top. Then we call $\mathscr{B}_{r}$ the set of \emph{bad pairs} associated to $\dual{A}_{r}$. Notice that if $B \in \mathscr{B}_{r}$ then $ B \subset Q$ and $|B| = 2$. Furthermore observe that by minimality of $l$, $\mathscr{B}_{r}$ contains the tops of all splits consisting of a pair of $l$ tuples and a word in $X_n^{r}$. Let $B_{r} \subset \mathscr{B}_{r}$ be this subset. We call $B_{r}$ the \emph{minimal bad pairs} associated to $\dual{A}_{r}$.
\end{Definition}

\begin{Definition}[Graph of Bad pairs]
For a transducer $A \in \spn{n}$, and for $r$ greater than or equal to the minimal synchronizing level, such that $\dual{A}_{r}$ has minimal splitting length $l$, form a directed graph $G_{r}(A)$ associated to $\dual{A}_{r}$ as follows:
\begin{enumerate}[label = (\roman{*})]
\item The vertex set of $G_r(A)$ is the  set $\mathscr{B}_r$ of bad pairs.
\item Two elements $\{x_1, x_2\}$, and $\{{y}_1, {y}_2\}$ of $\mathscr{B}_r$ are connected by an arrow going from $\{x_1, x_2\}$ into $\{{y}_1, {y}_2\}$, if there are pairs $(T_1, T_2) \in Q^{m} \times Q^{m}$, for some $m \ge l$, splitting $\dual{A}_{r}$, with top $\{x_1, x_2\}$ and bottom $\{y_1, y_2\}$ and such that the bottom depends only on the top. 
\end{enumerate}

We call $G_r(A)$ the \emph{graph of bad pairs} associated to $\dual{A}_{r}$. By Remark \ref{badpairsmakesense} this definition makes sense. 

The graph $G_{r}(A)$ possesses an interesting subgraph $\overline{G}_{r}(A)$ whose vertices are elements of $B_r$ the set of minimal bad pairs, with an edge from $\{x_1, x_2\}$ to $\{y_1, y_2\}$, $\{x_1, x_2\}, \{y_1, y_2\} \in B_{r}$ if there is a split of minimal length $l$, with top $\{x_1, x_2\}$ and bottom $\{y_1, y_2\}$. We call $\overline{G}_{r}(A)$ the \emph{minimal graph of bad pairs}.

\end{Definition}

\begin{Remark}
There is a much larger graph containing $G_{r}(A)$ which we do not consider here. This graph has all subsets of $Q_A \times Q_A$ with size exactly two, and there is a directed edge between two such vertices if there is a split of $\dual{A}_{r}$ with top the initial vertex and bottom the terminal vertex. The reason we do not consider this larger graph is due to the existence of elements of finite order in $\shn{n}$ whose dual at the synchronizing level splits (see Example \ref{exampleelelementoffiniteorderwhosegraphofbadpairssplits}). This means that in the larger graphs contain information which is not carried by powers of the transducers.
\end{Remark}

The following results link graph theoretic properties of $G_{r}(A)$ and the order of $A$ when $A \in \shn{n}$. All of these results apply also to the minimal graph of bad pairs $\overline{G}_{r}(A)$. In most cases the information given by $G_{r}(A)$ can already be seen in $\overline{G}_{r}(A)$, however this is not always.

\begin{lemma} \label{lemmagphofbdprs}
Let $A \in \shn{n}$ be a transducer, and suppose that $k$ is the minimal synchronizing level of $A$. Let $r \ge k$ and let $G_{r}(A)$ be the  graph of bad pairs associated to $\dual{A}_{r}$. If $G_{r}(A)$ is non-empty and contains a circuit i.e there is a vertex which we can leave and return to, then  $A$ has infinite order.
\end{lemma}
\begin{proof}
Let $l$ be the minimal splitting length of $\dual{A}_{r}$. The proof will proceed as follows, for every $m \ge 1$, we construct a word, $w(rm)$, of length  $rm$, such that there are two distinct elements of $Q^{ml+1}$ which have different outputs when processed through $w(rm)$ (in $\dual{A}_{rm}$). This will contradict $A$ having finite order, since by Lemma \ref{lemma 1} above, if $A$ has finite order, then there will be a $j$ such that any two sequences of states of any length will have the same output when processed through a word of length $rj$ (see Remark \ref{remark 1}).

Since $G_{r}(A)$ is non-empty, and has a circuit, there exists a circuit: $\{x_1,y_1\} \to \{x_2, y_2\} \to \ldots \to \{x_j, y_j\}\to \{x_1, y_1\}$. For $i \in \mathbb{N}$ let $A^{i} = \gen{X_n, Q^{i}, \pi_{i}, \lambda_{i}}$.
 
Now for $m =1$, since $\{x_1, y_1\}$ is a vertex of $G_{r}(A)$ with at least one incoming edge, there is a state $\Gamma_1$ of $\dual{A}_{r}$, and a pair $(S_1,T_1) \in Q^{l_1}\times Q^{l_1}$ (for $l_1 \ge l$) such that $(S_1,T_1, \Gamma_1)$ is a split of $\dual{A}_{r}$ with bottom $\{x_1, y_1\}$ and such that the bottom depends only on the top. We may assume  that the top of this split is $\{x_j,y_j\}$. Therefore for any $p \in Q$, the output of $S_1p$ when processed through $\Gamma_1$ is not equal to the output of $T_1p$ when processed through $\Gamma_1$. However the output of $S_1$ and $T_1$ are equal when processed through $\Gamma_1$ since the bottom of the split depends only on its top. Hence the following picture is valid, for appropriate $U \in Q^{l_1}$.
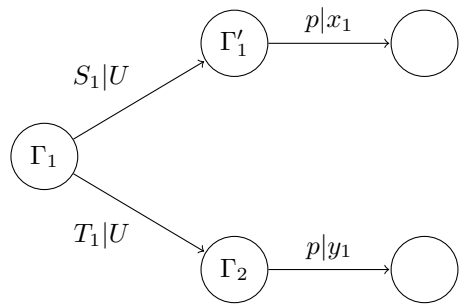
\begin{figure}[h!]
\begin{center}
\begin{tikzpicture}[shorten >=0.5pt,node distance=3cm,on grid,auto] 
   \node[state] (q_0) [yshift = 1.5cm]   {$\Gamma_1$}; 
   \node[state] (q_1) [yshift=3cm, xshift=2.5cm] {$\Gamma_1'$}; 
   \node[state] (q_2) [yshift=0cm,xshift=2.5cm] {$\Gamma_2$}; 
   \node[state] (q_3) [yshift=0cm, xshift=5cm] {};
   \node[state] (q_4)  [yshift=3cm,xshift=5cm] {};
    \path[->] 
    (q_0) edge node {$S_1|U$} (q_1)
          edge node[swap] {$T_1|U$} (q_2)
     (q_1) edge node {$p| x_1$} (q_4)
     (q_2) edge node {$p|y_1$} (q_3);
\end{tikzpicture}
\end{center}
\caption{Stage 1 of construction}
\end{figure}

Now since $\{x_1, y_1\}$ is connected to $\{x_2, y_2\}$ there is a word $\Lambda_{1} \in X_n^{r}$ such that there is a pair $(S_2,T_2) \in Q^{l_2} \times Q^{l_2}$ and $(S_2,T_2, \Lambda_{1})$ is a split with top $\{x_1, y_1\}$ and bottom $\{x_2, y_2\}$ such that the bottom depends only on the top. Let $\Lambda_1'$ be the word of length $r$ such that $\lambda_{l_1}(\Lambda_1', U) = \Lambda_1$ (such a word exists since $A$ is invertible). Since the bottom of the split depends only on the top there is $V \in Q^{l_2}$ such that  for any $P \in Q^{l_2 -1}$ we have $\pi_{l_2}(\Lambda_1, (x_1,P))= V = \pi_{l_2}(\Lambda_1, (y_1,P)) $. Let $V'$ be the state  of $A^{l_1}$ such that $\ \pi_{l_1}(\Lambda_1',U) = V'$. Then let $w(k2) = \Gamma_1\Lambda_1'$. Now by the Remark \ref{badpairsmakesense}, we have the following transition 
\begin{figure}[h!]
\begin{center}
\begin{tikzpicture}[shorten >=0.5pt,node distance=3cm,on grid,auto] 
   \node[state] (q_0) [yshift = 1.5cm]   {$\Gamma_1\Lambda_1'$}; 
   \node[state] (q_1) [yshift=3cm, xshift=2.5cm] {$\Gamma_1'\Lambda_1$}; 
   \node[state] (q_2) [yshift=0cm,xshift=2.5cm] {$\Gamma_2\Lambda_1$}; 
   \node[state] (q_3) [yshift=0cm, xshift=5cm] {$\#\Lambda_2$};
   \node[state] (q_4)  [yshift=3cm,xshift=5cm] {$\ast\Lambda_2'$};
    \node[state] (q_6) [yshift=0cm, xshift=7.5cm] {};
    \node[state] (q_5)  [yshift=3cm,xshift=7.5cm] {};
    \path[->] 
    (q_0) edge node {$T_1|V'$} (q_1)
          edge node[swap] {$S_1|V'$} (q_2)
     (q_1) edge node {$P| V$} (q_4)
     (q_2) edge node[swap] {$P|V$} (q_3)
     (q_3) edge node [swap] {$p| x_2$} (q_6)
     (q_4) edge node {$p|y_2$} (q_5);;
\end{tikzpicture}
\end{center}
\caption{Stage 2 of construction.}
\end{figure}
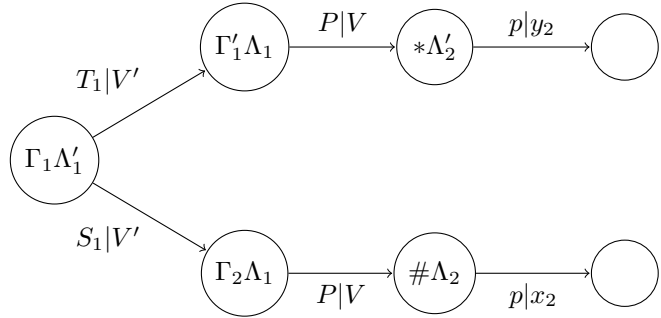
for some $P \in Q^{l_2}$ and $p \in Q$. Now we can iterate the above process, since $\{x_2, y_2\}$ is a vertex of $G_{r}(A)$ with an outgoing edge to another vertex of $G_{r}(A)$, and the output of $S_1P$ and $T_1P$ when processed through $\Gamma_1\Lambda_1'$ are the same. 

Label the levels of the above picture by $1,2,3,4$. We grow our word from bottom to top. Let $\Delta_1$ be the word such that there is a pair $(S_3, T_3)$ in $Q^{l_3} \times Q^{l_3}$ and $(S_3,T_3, \Delta_1)$ is a split of $\dual{A}_{r}$ with top $\{x_2, y_2\}$ and bottom $\{x_3, y_3\}$ such that the bottom depends only on the top. Attach $\Delta_1$ to right end of both words representing the states of level  $3$. There is a word $\Delta_1'$ such that $\lambda_{l_1+l_2}(\Delta_1', V'V) = \Delta_1$. Hence $W(r3) = \Gamma_1\Lambda_1'\Delta_1'$. Moreover since the bottom of the split $(S_3,T_3, \Delta_1)$ depends only on its top we see that  $\dual{A}_{3r}$ has a split of length $l_1 + l_2 + l_3$  with bottom $\{x_3, y_3\}$ whose bottom depends only on its top.  We repeat the above process and in this way construct the words $w(rm)$ demonstrating that $\dual{A}_{rm}$ has a  split with bottom $\{x_i, y_i\}$ where $1 \le i \le j$, and $i \equiv m \mod j$ such that the bottom depends only on the top.
\end{proof}

\begin{Remark}\label{circuitimpliesloop}
The proof above actually demonstrates that if \(A \in \widetilde{\T{H}}_{n}\) is such that for some $r$ bigger than or equal to its synchronizing level, the graph $G_{r}(A)$ of bad-pairs has a circuit, then there is an  $r' >r$ depending on the length of the circuit in $G_{r}(A)$, such that $G_{r'}(A)$ has a loop.
\end{Remark}

The above gives a sufficient condition for determining when an element of $\shn{n}$ has infinite order, although it does not produce  a witness. The next results show that when the graph of bad pairs contains a circuit for an element of $\shn{n}$ then we are also able to generate a rational word on an infinite orbit under the action of the transducer. 

We start with the case where the graph of bad pairs is contains a loop and prove the general case by reducing to the loop case.

\begin{proposition}\label{loopgiveswitness}
Let $A \in \widetilde{\T{H}}_{n}$ be synchronizing at level $k$, and let $r \ge k$. If the graph $G_{r}(A)$ of bad pairs has a loop, then there is rational word in $X_n^{\Z} $ on an infinite orbit under the action of $A$.
\end{proposition}
\begin{proof}
Let $(p,q)$ be a vertex of $G_{r}$ with a loop. Furthermore assume that $m \in \mathbb{N}$ is the minimum splitting length of $\dual{A}_{r}$. Let $\Gamma_1 \in X_n^{r}$, and $P,Q \in Q_A^{l-1}$ be such that $(pP, qQ, \Gamma_1)$ is a split with top and bottom equal to $\{p,q\}$ such that the bottom depends only on the top. 

Let  $A^{l} = \gen{X_n Q^{l}, \lambda_{l}, \pi_{l}}$, then since $(pP, qQ, \Gamma_1)$ is a split with top and bottom equal to $\{p,q\}$ such that the bottom depends only on the top, there is a state $S_0$ such that $\pi_{l}(pP, \Gamma_1)= S_0 = \pi_{l}(qQ, \Gamma_1)$ for any $P$ and $Q$ in $Q_A^{l-1}$. Therefore let $S_1 := \pi_{l}(\Gamma_1, S_0)$. Let $\Gamma_2 \in X_n^r$ be the unique word such that $\lambda_{l}(\Gamma_2, S_1) = \Gamma_1$. Assume that $\Gamma_i$ is defined and $S_i$ is equal to $\pi_{l}(\Gamma_i, S_{i-1})$, then let $\Gamma_{i+1}$ be the unique word in $X_n^r$ such that $\lambda_{l}(\Gamma_{i+1}, S_{i}) = \Gamma_i$. Eventually we find there are $i \le j \in \mathbb{N}$ such that $\Gamma_{i} = \Gamma_{j+1}$.

Suppose that $\lambda_{l}(\Gamma_1, pP) = \Delta$  and $\lambda_{l}(\Gamma_1, qP) = \Lambda$. Consider the bi-infinite word: 
\[
\ldots \Delta\overset{\centerdot}{\Delta}\Gamma_1\ldots\Gamma_{i-1}(\Gamma_i\ldots\Gamma_j)(\Gamma_i\ldots\Gamma_j)\ldots
\]

where `$\stackrel{\centerdot}{\Delta}$' indicates that $\Delta$ starts at the zero position.
There are two cases to be considered.

{\bf{Case 1:}} $\Delta \in W_{p}$ and $\Lambda \in W_{q}$. We consider how powers of $A^{l}$ act on this word. Since $\Delta \in W_{p}$ and for any $T \in pQ^{l-1}$, $\lambda_{l}(\Gamma_1, T) \in W_{p}$, the bottom of the split $(pP, qQ, \Gamma_1)$ depends only on the top, we must have that:

\[
 \ldots \Delta\overset{\centerdot}{\Delta}\Gamma_1\ldots\Gamma_{i-1}(\Gamma_i\ldots\Gamma_j)(\Gamma_i\ldots\Gamma_j)\ldots \overset{A^{l}}{\rightarrow}\ldots \overset{\centerdot}{\ast}_{0}\Delta_1\Gamma_1\ldots\Gamma_{i-1}(\Gamma_i\ldots\Gamma_j)(\Gamma_i\ldots\Gamma_j)\ldots
\]

Now since $\Delta_{1} \in W_{p}$, we can repeat the above
\[
 \ldots \overset{\centerdot}{\ast}_{0}\Delta_1\Gamma_1\ldots\Gamma_{i-1}(\Gamma_i\ldots\Gamma_j)(\Gamma_i\ldots\Gamma_j)\ldots
  \overset{A^{l}}{\rightarrow}\ldots \overset{\centerdot}{\ast}_{0}'\ast_{1}\Delta_2\Gamma_1\ldots\Gamma_{i-1}(\Gamma_i\ldots\Gamma_j)(\Gamma_i\ldots\Gamma_j)\ldots
\]

Therefore after applying $A^{l}$ $t$ times for some $t \in \mathbb{N}$ we see that from the  position $i$ onwards the output is of the form $\Delta_{r}\Gamma_1\ldots\Gamma_{i-1}(\Gamma_i\ldots\Gamma_j)(\Gamma_i\ldots\Gamma_j)\ldots$, and $\Delta_m \in W_p$. Therefore if $\Gamma_i \ne \Gamma_1$, \(\ldots \Delta\smash{\overset{\centerdot}{\Delta}}\Gamma_1\ldots\Gamma_{i-1}(\Gamma_i\ldots\Gamma_j)(\Gamma_i\ldots\Gamma_j)\ldots \) is on an infinite orbit under the action of $A^{l}$. This follows for if $t, t' \in \mathbb{N}$ such that $t \ne t'$, then we have:

\[
( \ldots \Delta\overset{\centerdot}{\Delta}\Gamma_1\ldots\Gamma_{i-1}(\Gamma_i\ldots\Gamma_j)(\Gamma_i\ldots\Gamma_j)\ldots)A^{l*t} \ne ( \ldots \Delta\overset{\centerdot}{\Delta}\Gamma_1\ldots\Gamma_{i-1}(\Gamma_i\ldots\Gamma_j)(\Gamma_i\ldots\Gamma_j)\ldots)A^{l* t'}
\]
otherwise:

\[
\ldots \Delta\overset{\centerdot}{\Delta}\Gamma_1\ldots\Gamma_{i-1}(\Gamma_i\ldots\Gamma_j)(\Gamma_i\ldots\Gamma_j)\ldots= ( \ldots \Delta\overset{\centerdot}{\Delta}\Gamma_1\ldots\Gamma_{i-1}(\Gamma_i\ldots\Gamma_j)(\Gamma_i\ldots\Gamma_j)\ldots)A^{l* |t-t'|}
\]

However by minimality of $i$, we have that $\Gamma_1 \ne \Gamma_{t}$ for $1< t \le j$, yielding a contradiction.

If $\Gamma_i = \Gamma_1$ then our original word $\ldots \Delta\overset{\centerdot}{\Delta}\Gamma_1\ldots\Gamma_{i-1}(\Gamma_i\ldots\Gamma_j)(\Gamma_i\ldots\Gamma_j)\ldots$, becomes 
\[
\ldots \Delta\overset{\centerdot}{\Delta}(\Gamma_1\ldots\Gamma_{i-1})(\Gamma_1\ldots\Gamma_{i-1})\ldots.
\]

 Notice that the infinite word corresponding to the coordinates $\mathbb{N}\backslash\{0\}$ is periodic with period $i-1$. Now if $\Gamma_{i-1} \notin W_{p}$, Then for any $t \in \mathbb{N}$ such that $t > 0 $ we have:
\[
(\ldots \Delta\overset{\centerdot}{\Delta}(\Gamma_1\ldots\Gamma_{i-1})(\Gamma_1\ldots\Gamma_{i-1})\ldots)A^{lt(i-1)} \ne \ldots \Delta\overset{\centerdot}{\Delta}(\Gamma_1\ldots\Gamma_{i-1})(\Gamma_1\ldots\Gamma_{i-1})\ldots.
\]

since $\Delta_{t(i-1)} \ne \Gamma_{t(i-1)} = \Gamma_{i-1}$. Therefore for any $t, t' \in \mathbb{N}$ such that $t \ne t'$, we must have that:
\[
(\ldots \Delta\overset{\centerdot}{\Delta}(\Gamma_1\ldots\Gamma_{i-1})(\Gamma_1\ldots\Gamma_{i-1})\ldots)A^{lt(i-1)} \ne (\ldots \Delta\overset{\centerdot}{\Delta}(\Gamma_1\ldots\Gamma_{i-1})(\Gamma_1\ldots\Gamma_{i-1})\ldots)A^{lt'(i-1)}.
\]

If $\Gamma_i = \Gamma_1$ and $ \Gamma_{i-1} \in W_{p}$, then consider the bi-infinite word $\ldots \Lambda\overset{\centerdot}{\Lambda}\Gamma_1\ldots\Gamma_{i-1}(\Gamma_1\ldots\Gamma_{i-1})$, since $q \ne p$ and $\Lambda \in W_q$ and $\Gamma_{i-1} \in W_p$ we have $\Gamma_{t(i-1)} =\Gamma_{i-1} \ne \Lambda_{t(i-1)}$ for any $m \in \mathbb{N}$. Here $\Lambda_{t(i-1)}$ is defined analogously to $\Delta_{t(i-1)}$. Therefore by the argument above $\ldots \Lambda\overset{\centerdot}{\Lambda}\Gamma_1\ldots\Gamma_{i-1}(\Gamma_{1}\ldots\Gamma_{i-1})\ldots$ is on an infinite orbit under the action of $A^{l}$.

{\bf{Case 2 :}} We assume now that $\Delta \in W_{q}$ and $\Lambda \in W_{p}$. As in the previous case we consider how powers of $A^{l}$ act on the word \(\ldots \Delta\overset{\centerdot}{\Delta}\Gamma_1\ldots\Gamma_{i-1}(\Gamma_i\ldots\Gamma_j)(\Gamma_i\ldots\Gamma_j)\ldots\). 

Making an argument similar to case 1, we have that: 

\[
 \ldots \Delta\overset{\centerdot}{\Delta}\Gamma_1\ldots\Gamma_{i-1}(\Gamma_i\ldots\Gamma_j)(\Gamma_i\ldots\Gamma_j)\ldots \overset{A^{l}}{\rightarrow}\ldots \overset{\centerdot}{\ast}_{0}\Delta_1\Gamma_1\ldots\Gamma_{i-1}(\Gamma_i\ldots\Gamma_j)(\Gamma_i\ldots\Gamma_j)\ldots
\]

However, in this case $\Delta_1 \in W_{p}$. Applying  $A^{l}$ again we have:
\[
 \ldots \overset{\centerdot}{\ast}_{0}\Delta_1\Gamma_1\ldots\Gamma_{i-1}(\Gamma_i\ldots\Gamma_j)(\Gamma_i\ldots\Gamma_j)\ldots
  \overset{A^{l}}{\rightarrow}\ldots \overset{\centerdot}{\ast}_{0}\ast_{1}\Delta_2\Gamma_1\ldots\Gamma_{i-1}(\Gamma_i\ldots\Gamma_j)(\Gamma_i\ldots\Gamma_j)\ldots
  \]
where $\Delta_2 \in W_q$. Therefore given $t \in \mathbb{N}$ we know that after applying $A^{l}$ $t$ times, the resulting word, from the $t$\textsuperscript{th} position onwards is of the form:
$\Delta_{t}\Gamma_1\ldots\Gamma_{i-1}(\Gamma_i\ldots\Gamma_j)(\Gamma_i\ldots\Gamma_j)\ldots$ where $\Delta_{t} \in W_q$ if $t$ is even, and $\Delta_{t} \in W_{p}$ if $t$ is odd.

By considering the bi-infinite word: $\ldots \Lambda\overset{\centerdot}{\Lambda}\Gamma_1\ldots\Gamma_{i-1}(\Gamma_i\ldots\Gamma_j)(\Gamma_i\ldots\Gamma_j)\ldots$, and similarly defining the $\Lambda_{t}$'s $t \in \mathbb{N}$, we see that after applying $A^{l}$ $t$-times to this word, the output, from the $t$\textsuperscript{th} position onwards is of the form: $\Lambda_{t}\Gamma_1\ldots\Gamma_{i-1}(\Gamma_i\ldots\Gamma_j)(\Gamma_i\ldots\Gamma_j)\ldots$ where $\Lambda_{t} \in W_q$ if $t$ is odd, and $\Lambda_{t} \in W_{p}$ if $t$ is even.

Now we go through the subcases as in Case 1. If $\Gamma_{i} \ne \Gamma_1$, then the argument proceeds exactly as in Case 1.

Hence consider the case \( \Gamma_{i} = \Gamma_1\). Again we split into subcases. Now either $\Gamma_{i-1} \in W_p \sqcup W_q$ or it is disjoint from these two sets. We assume $\Gamma_{i-1} \in W_q$ (the other case is proved analogously). Since $\Lambda_{2t(i-1)} \in W_{p}$, for $t \in \mathbb{N}$ then by similar arguments to Case 1 above we conclude that $\ldots \Lambda\overset{\centerdot}{\Lambda}\Gamma_1\ldots\Gamma_{i-1}(\Gamma_i\ldots\Gamma_j)(\Gamma_i\ldots\Gamma_j)\ldots$ is on an infinite orbit under the action of $A^{2l}$ and so under the action of $A^{l}$.

If $\Gamma_{i-1} \cap (W_p \sqcup W_q) = \emptyset$ then \(\ldots \Delta\overset{\centerdot}{\Delta}\Gamma_1\ldots\Gamma_{i-1}(\Gamma_1\ldots\Gamma_{i-1})\ldots\) and $\ldots \Lambda\overset{\centerdot}{\Lambda}\Gamma_1\ldots\Gamma_{i-1}(\Gamma_1\ldots\Gamma_{i-1})\ldots$ are on infinite orbits under the action of $A^{l}$ by repeating the argument in Case one. 

\end{proof}

\begin{Remark}
In the proof above, in showing that the witness is on an infinite orbit under the action  of the transducer, our argument has made use only of the  the right infinite portion corresponding to the coordinates $\mathbb{N}\sqcup\{-1,\ldots, -r \}$.  In particular this means that we can replace the left infinite portion  corresponding to the coordinates $\{\ldots,-r-3,-r-2,-r-1\}$ by any infinite word in $X_n^{\mathbb{N}}$.
\end{Remark}
\begin{corollary}\label{circuitgiveswitness}
Let $A \in \widetilde{\T{H}}_{n}$ be synchronizing at level $k$, and let $r \ge k$. If the graph $G_{r}(A)$ of bad pairs has a circuit, then there is rational word in $X_n^{\Z} $ on an infinite orbit under the action of $A$.
\end{corollary}
\begin{proof}
This is a consequence of the proposition above and Remark \ref{circuitimpliesloop}.
\end{proof}

Below is an example of a transducer $B$ whose graph of bad pairs at level 1 $G_{1}(B)$ satisfies the conditions of  Lemma \ref{lemmagphofbdprs}, we also construct a witness as in Proposition \ref{loopgiveswitness} which demonstrates that the transducer has infinite order.

\begin{example}
Let $B$ be the transducer in Figure \ref{exampleillustratingrationalwitness}. Its dual is given by Figure \ref{exampleofdualwithloop}.
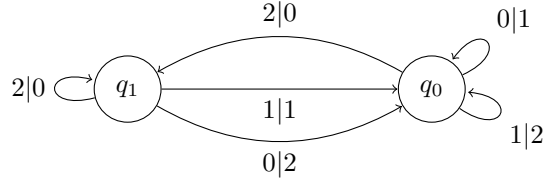
\begin{figure}[h!]
\begin{center}
\begin{tikzpicture}[shorten >=0.2pt,node distance=3cm,on grid,auto] 
   \node[state] (q_0)   {$q_0$}; 
   \node[state] (q_1) [ xshift=-4cm,yshift=0cm] {$q_1$}; 
    \path[->] 
    (q_0) edge[in=55,out=25,loop] node[swap] {$0|1$} ()
          edge[in=355,out=325,loop] node[swap] {$1|2$} ()
          edge[bend right] node [swap] {$2|0$} (q_1)
     (q_1)edge node [swap] {$1|1$} (q_0)
          edge[bend right] node[swap] {$0|2$} (q_0)
          edge[loop left] node {$2|0$} ();
\end{tikzpicture}
\end{center}
\caption{An element of infinite order in $\hn{n}$}
\label{exampleillustratingrationalwitness}
\end{figure}
\begin{figure}[h!]
\begin{center}
\begin{tikzpicture}[shorten >=0.2pt,node distance=3cm,on grid,auto] 
   \node[state] (q_0) [xshift = 0cm, yshift=0cm]   {$0$}; 
   \node[state] (q_1) [xshift=-3cm,yshift=0cm] {$1$};
   \node[state] (q_2) [xshift = 3cm, yshift= 0cm] {$2$}; 
    \path[->] 
    (q_0) edge node[swap] {$q_0|q_0$} (q_1)
          edge node {$q_1|q_0$} (q_2)
     (q_1)edge[out=30,in=150] node {$q_0|q_0$} (q_2)
          edge[loop left] node {$q_1|q_0$} ()
     (q_2)edge[bend left] node {$q_0|q_1$,$q_1|q_1$} (q_0);
\end{tikzpicture}
\end{center}
\caption{The level 1 dual of $B$}
\label{exampleofdualwithloop}
\end{figure}
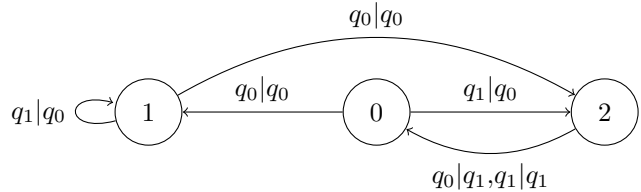
From this it is easy to see that  the pair $\{q_0, q_1\}$ is a vertex of $G_{1}$ and there is a directed edge with initial and terminal vertex $\{q_0, q_1\}$. Therefore the conditions of Lemma \ref{lemmagphofbdprs} are satisfied and $B$ has infinite order. Going through the construction in the proof of Proposition \ref{loopgiveswitness}, we see that \smash{$\ldots11\overset{\centerdot}{1}(02)(02)\ldots$} is on an infinite orbit under the action of $B$.
\end{example}

 \begin{example}
 The transducer $A$ shown in Figure  \ref{elementofinifiniteorderwhosegrphofbadpairshasnocircuit} demonstrates that though the  graph of bad pairs may contain a circuit at some level, the minimal graph of bad pairs at the same level may not do so. 
 
 \begin{figure}[H]
 \begin{center} \label{figure 3}
 \begin{tikzpicture}[shorten >=0.5pt,node distance=5cm,on grid,auto] 
    \node[state] (q_1) [xshift = 0cm]   {$q_1$}; 
    \node[state] (q_2) [xshift=3cm, yshift= -3cm] {$q_2$}; 
    \node[state] (q_3) [xshift=6cm,yshift=0cm] {$q_3$}; 
     \node[state] (q_4) [xshift=3cm, yshift=0cm] {$q_4$};
     
     \path[->] 
     (q_1) edge[loop left] node  {$1|2$} ()
           edge [bend right] node {$2|3$} (q_2)
           edge node  {$3|1$} (q_4)      
      (q_2) edge [in=270,out=180] node {$1|2$} (q_1)
            edge [in=275 ,out=0] node[swap] {$2|3$} (q_3)
            edge node {$3|1$} (q_4)
      (q_3) edge [loop right] node {$2|3$} ()
            edge [in=60,out=120] node [swap] {$1|1$} (q_1)
            edge node [swap] {$3|2$} (q_4)
      (q_4) edge [in=45,out=75,loop] node {$3|1$} (q_4)
            edge [in=340,out=200] node{$2|2$} (q_1)
            edge [in=65,out=295] node{$1|3$} (q_2);
 \end{tikzpicture}
 \end{center}
 \caption{An element of infinite order whose minimal graph of bad pairs has no circuits.}
 \label{elementofinifiniteorderwhosegrphofbadpairshasnocircuit}
 \end{figure}
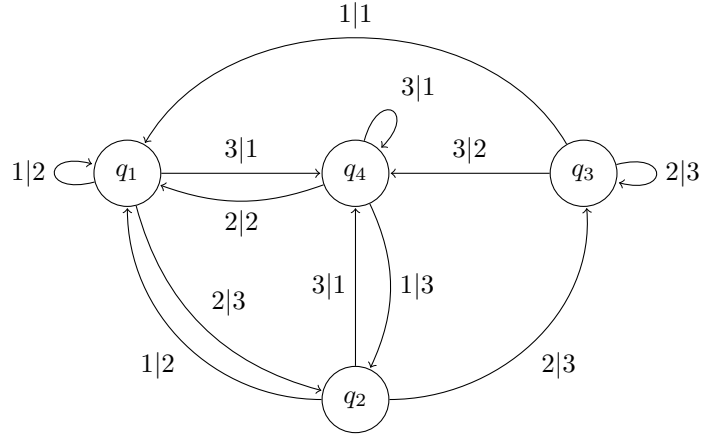
 
 It is easy to see that this transducer is bi-synchronizing at level $3$ using the minimisation procedure outlined in \cite{BlkYMaisANav}, or by direct computation in GAP. The graph $G_3(A)$ of bad pairs has a loop at the vertex $\{q_1, q_2\}$. The minimal graph of bad pairs $\overline{G}_{3}(A)$ is as shown in Figure \ref{exampleofgrphofbadpairswithnocircuit}:
 \begin{figure}[h!]
 \begin{center} \label{figure 4}
  \begin{tikzpicture}[shorten >=0.5pt,node distance=3cm,on grid,auto] 
  
     \node[state, accepting] (q_1) [xshift = 0cm]   {$(q_2, q_3)$}; 
     \node[state] (q_2) [xshift=-3cm, yshift= 0cm] {$(q_1,q_2)$}; 
     \node[state] (q_3) [xshift= 3cm,yshift=0cm] {$(q_1,q_3)$}; 
     \node[state] (q_4) [xshift=0cm, yshift= -3cm] {$(q_1,q_4)$};
     \node[state] (q_5) [xshift=-3cm, yshift= -3cm] {$(q_3,q_4)$};
     \node[state] (q_6) [xshift=3cm, yshift= -3cm] {$(q_2,q_4)$};

      \path[->] 
      (q_2) edge node  {} (q_1)
      (q_3) edge node {} (q_1)
      (q_5) edge node {} (q_2)
            edge node {} (q_1)
      (q_4) edge node {} (q_1)
            edge node {} (q_2)
      (q_6) edge node {} (q_1)
            edge node {} (q_2);
  \end{tikzpicture}
  \end{center}
  \caption{The graph $\overline{G}_{3}(A)$ of minimal bad pairs}
  \label{exampleofgrphofbadpairswithnocircuit}
 \end{figure}
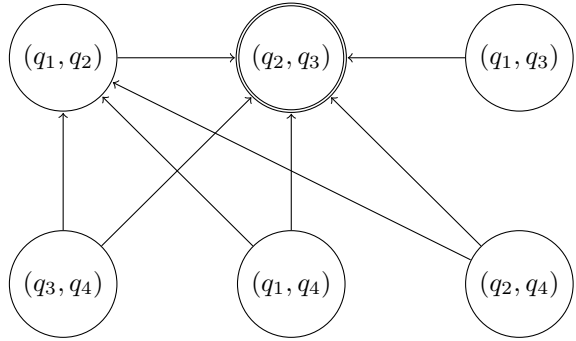
Here the state pair $(q_2,q_3)$ is not a minimal bad pair, and in fact reads any word  of length 3 into a pair of the form $(p,p)$ (it acts like a sink through which we escape the minimal bad pairs).
\end{example}

\begin{example}\label{exampleelelementoffiniteorderwhosegraphofbadpairssplits}
The transducer $H \in \hn{5}$ shown in Figure \ref{elementoffiniteorderwhosegrphofbadpairsplits} is an element of finite order whose dual at its minimal bi-synchronizing level splits. However the next power of its dual is the zero of the semigroup generated by the dual. This means that the splits can be fixed by taking powers of the dual.

 \begin{figure}[h!]
 \begin{center} 
 \begin{tikzpicture}[shorten >=.5pt,node distance=1cm,on grid,auto] 
    \node[state] (q_1) [xshift = -2cm]   {$q_1$}; 
    \node[state] (q_2) [xshift=3.5cm, yshift= -2.5cm] {$q_2$}; 
    \node[state] (q_3) [xshift=10cm,yshift=-0.5cm] {$q_3$}; 
     \node[state] (q_4) [xshift=3.5cm, yshift=0cm] {$q_4$};
     \node[state] (q_5) [xshift=6cm, yshift=-2cm] {$q_5$};
     
     \path[->] 
     (q_1) edge[out=140, in=110, loop] node  {$1|0$} ()
           edge node[swap] {$0|4$} (q_4)
           edge[out=260, in=270] node  {$2|2$} (q_5)
           edge[in=100, out = 100] node {$3|3$} (q_3)
           edge[in=165, out=325] node[swap] {$4|1$} (q_2)      
      (q_2) edge node[swap] {$0|0$} (q_1)
            edge [out=100, in=260] node {$1|4$} (q_4)
            edge [out=290, in=270 ] node  {$2|3$} (q_3)
            edge node {$3|2$} (q_5)
            edge[in=350, out=320, loop] node[swap] {$4|1$} (q_2)
      (q_3) edge[in=85, out=120] node {$0|1$} (q_1)
            edge [bend right] node[swap] {$1|3$} (q_4)
            edge[in=60, out= 90, loop] node  {$2|2$} ()
            edge node [swap] {$3|4$} (q_5)
            edge[out = 290, in=270] node {$4|0$} (q_2)
      (q_4) edge [loop above] node {$0|4$} (q_4)
            edge [in=10,out=170] node[swap]{$1|1$} (q_1)
            edge [in=105, out=350] node{$2|3$} (q_5)
            edge [in=150, out=15] node[swap] {$3|2$} (q_3)
            edge node {$4|0$} (q_2)
      (q_5) edge [out=110, in=340] node {$0|4$} (q_4)
            edge[in= 280, out =250] node {$1|1$} (q_1)
            edge[out=15, in=210] node[swap] {$2|2$} (q_3)
            edge[out=345, in=315, loop] node {$3|3$} ()
            edge [in=0, out=200] node {$4|0$} (q_2);
 \end{tikzpicture}
 \end{center}
 \caption{An element of finite order whose graph of bad pairs splits.}
 \label{elementoffiniteorderwhosegrphofbadpairsplits}
 \end{figure}
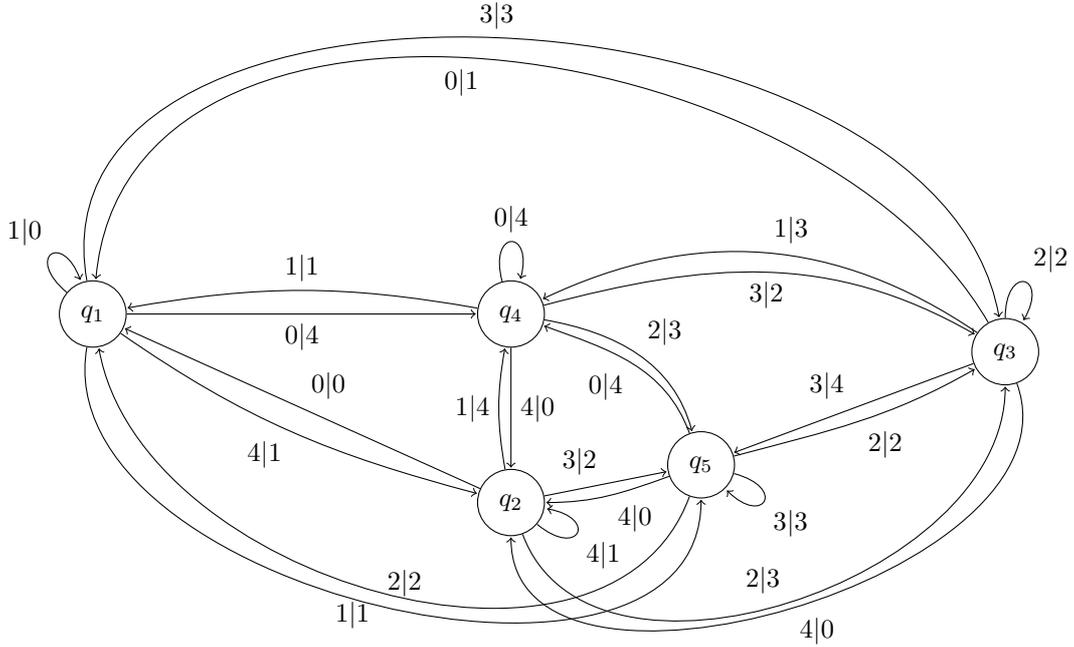

\end{example}
\subsection{Alternative conditions for having infinite order} 

In this section we give an algebraic condition which implies that the graph of bad pairs for a transducer for some power of its dual has a loop. 

To this end let $A \in \widetilde{\T{H}}_{n}$, and let $r \in \mathbb{N}$ be greater than or equal to the minimal synchronizing level of $A$. Let $l \in \mathbb{N}$ be the minimal splitting length of $\dual{A}_{r}$.

To each state $\Gamma$ of $\dual{A}_{r}$ we associate a transformation $\sigma_{\Gamma}$ of the set $Q_A$ of states of $A$. We do this as follows. For each state $q \in Q_{A}$ and for any $l-1$ tuple $S \in Q^{l-1}$, there is a unique state $p \in Q_{A}$, such that if $\Delta$ is the output when $\Gamma$ is processed through $qS$ in $A^{l}$, then $\Delta \in W_{p}$. Since $l$ is the minimal splitting length of $\dual{A}_{r}$, $p$ is independent of which $l-1$ tuple $S$ we chose. Therefore define $\sigma_{\Gamma}$ such that \smash{$q \overset{\sigma_{\Gamma}}{\mapsto} p$}. 

For $j \in \mathbb{N}$ let $\mathfrak{S}_{r,j}$ be the set of all products of length  $j$ of elements of the set $\{\sigma_{\Gamma} | \Gamma \in X_n^{r}\}$. We have the following result:

\begin{proposition}
$A \in \widetilde{\T{H}}_{n}$, and let $r \in \mathbb{N}$ be greater than or equal to the minimal synchronizing level of $A$. Let $l \in \mathbb{N}$ be the minimal splitting length of $\dual{A}_{r}$. Then $\mathfrak{S}_{r,|Q_{A}|^2+1}$ contains a transformation of $Q_{A}$ which is not a right zero if and only if the graph $\overline{G}_{r}(A)$  of minimal bad pairs contains a circuit.
\end{proposition}
\begin{proof}
Let $A$, $r$, and $l$ be as in the statement of the proposition.

Now $\mathfrak{S}_{r,|Q_{A}|^2+1}$ contains a  transformation of $Q_{A}$ which is not a right zero if and only if there is a product $\sigma_{\Gamma_{1}}\sigma_{\Gamma_{2}}\ldots\sigma_{\Gamma_{|Q_{A}|^2+1}}$, for $\Gamma_{i} \in X_n^{r}$, $1 \le i \le |Q_{A}|^2+1$, whose image set has size at least 2. This occurs if and only if there are $p_0,q_0 \in Q_{A}$ which map to distinct elements under $\sigma_{\Gamma_{1}}\sigma_{\Gamma_{2}}\ldots\sigma_{\Gamma_{|Q_{A}|^2+1}}$. Let $p_i := (p_0)\sigma_{\Gamma_{1}}\sigma_{\Gamma_{2}}\ldots\sigma_{\Gamma_{i}}$ and $q_i := (q_0)\sigma_{\Gamma_{1}}\sigma_{\Gamma_{2}}\ldots\sigma_{\Gamma_{i}}$ for $1 \le i \le |Q_A|^2+1$. Notice that $p_i \ne q_i$ since $p_0$ and $q_0$ have distinct images under $\sigma_{\Gamma_{1}}\sigma_{\Gamma_{2}}\ldots\sigma_{\Gamma_{|Q_{A}|^2+1}}$.  

By the pigeon hole principle there exists $1 \le i,j \le |Q_A|^2+1$ such that $\{p_i, q_i\}:=\{p_j, q_j\}$.

This implies that in the graph $\overline{G}_{r}(A)$ of minimal bad pairs we have: $\{p_i, q_i\} \to  \{p_{i+1}, q_{i+1}\} \to \ldots \to \{p_j, q_j\} \to \{p_i, q_i\}$. This follows by definition of the $\sigma_{\Delta}$, $\Delta \in X_n^{r}$ and of the graph $\overline{G}_{r}(A)$.

Now suppose the graph $\overline{G}_{r}(A)$ contains a circuit. Let $j \in \mathbb{N}$ be the length of the circuit, and let $\{p_i,q_i\}$ $1\le i \le j$ be the vertices on the circuit. 

Let $1 \le i < j$ be arbitrary. Now an edge $\{p_i,q_i\} \to \{p_{i+1}, q_{i+1}\}$ corresponds to the existence of some $\Gamma_i \in X_n^{r}$ and $S_i, T_i \in Q_{A}^{l-1}$  such that $(p_iS_i, q_iT_i, \Gamma_i)$ is a split of $\dual{A}_{r}$ with bottom $\{p_{i+1}, q_{i+1}\}$. It then follows that the product $\sigma_{\Gamma_1}\ldots\sigma_{\Gamma_j}$ maps $\{p_1, q_1\} \to \{p_1, q_1\}$. This means that $\mathfrak{S}_{r,|Q_{A}|^2+1}$ contains a transformation of $Q_{A}$ which is not a right zero.
\end{proof}

\begin{corollary}
$A \in \widetilde{\T{H}}_{n}$, and let $r \in \mathbb{N}$ be greater than or equal to the minimal synchronizing level of $A$. Let $l \in \mathbb{N}$ be the minimal splitting length of $\dual{A}_{r}$. If $\mathfrak{S}_{r,|Q_{A}|^2+1}$ contains a transformation of $Q_{A}$ which is not a right zero then $A$ has infinite order. Moreover there is a rational word in $X_n^{\Z}$ on an infinite orbit under the action of $A$.
\end{corollary}

\begin{Remark}\label{ifinfiniteorderwithnocircuitsthenmusthavezero}
The above now implies that if $A \in \shn{n}$ has infinite order, but none of its graph of minimal bad pairs,  $\overline{G}_{r}(A)$, for $r \in \mathbb{N}$ greater than or equal to the minimal synchronizing length, has a loop, then $\mathfrak{S}_{r, |Q_A|^2 +1}$ consists entirely of right zeroes.
\end{Remark}

We have already seen that given an element  $A \in \spn{n}$ we can associate a transformation $\overline{A}_{j}$ of the set $X_n^{j}$ to  $A$. We shall now introduce a new transformation, which is defined only for elements of $\shn{n}$.

\begin{Definition}
Let $H \in \shn{n}$, and let $ j \in \mathbb{N}$, we shall define a transformation $\underline{H}_{j}$ of $X_n^{j}$ by  $$\Gamma \mapsto (\Gamma)q_{\Gamma}^{-1}$$

where $q_{\Gamma}$ is the unique state of $H$ forced by $\Gamma$, and $(\Gamma)q_{\Gamma}^{-1}$ is the unique element of $X_n^{j}$ such that $\lambda_{H}((\Gamma)q_{\Gamma}^{-1}, q_{\Gamma}) = \Gamma$. If  $j$ is zero, then $\underline{H}_{j}$ is simply the identity map on the set containing the empty word. 
\end{Definition}

\begin{Remark}
Given an element $H \in \shn{n}$ and a $ j \in \mathbb{N}$ such that $j >1$, then $\underline{H}_{j}$ is not injective in general. One can check that for the transducer $B$ of Figure \ref{exampleillustratingrationalwitness}, $\underline{B}_{1}$ is not injective. The map $\phi_{j}$  from $\shn{n}$ to the full transformation semigroup on $X_n^{j}$ which maps $H$ to $\underline{H}_{j}$ is not a homomorphism. 
\end{Remark}

\begin{lemma}
Let $H \in \shn{n}$ and let $j \in \N$ be greater than or equal to the minimal synchronizing level of $H$, then $\underline{H}_{j}$ is not injective if and only if  $\dual{H}_{j}$ has  a  split of length one such that the top and bottom of the split are equal.
\end{lemma}
\begin{proof}
($\Rightarrow$): suppose that, for $j \in \N$  and $H$ as in the statement of the lemma, $\underline{H}_{j}$ is not injective. This means that there are two distinct elements $\Gamma$ and $\Delta$ of $X_n^{j}$ such that $(\Gamma)\underline{H}_{j} =   (\Delta)\underline{H}_{j}$. Let $\Lambda := (\Gamma)\underline{H}_{j}$. Let $q_{\Gamma}$ and $q_{\Delta}$ be the states of $H$ forced by $\Gamma$ and $\Delta$ respectively. Then by definition of $\underline{H}_{j}$ we have: $\lambda(\Lambda, q_{\Gamma}) = \Gamma$ and $\lambda(\Lambda, q_{\Delta}) = \Delta.$

Observe that as  consequence it must be the case that $q_{\Gamma} \ne q_{\Lambda}$. Therefore it follows that we have the following split in $\dual{H}_{j}$:

\begin{figure}[H]
\begin{center}
\begin{tikzpicture}[shorten >=0.5pt,node distance=3cm,on grid,auto] 
   \node[state] (q_0) [yshift = 1.5cm]   {$\Lambda$}; 
   \node[state] (q_1) [yshift=3cm, xshift=2.5cm] {$\Gamma$}; 
   \node[state] (q_2) [yshift=0cm,xshift=2.5cm] {$\Delta$}; 
    \path[->] 
    (q_0) edge node {$q_{\Gamma}|q_{\Lambda}$} (q_1)
          edge node[swap] {$q_{\Delta}|q_{\Lambda}$} (q_2);
\end{tikzpicture}
\end{center}
\caption{Split in $\dual{A}_{j}$ with top equal to bottom}
\end{figure}
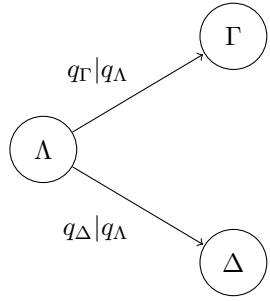
($\Leftarrow$) Suppose that $\dual{A}_{j}$ has a split of length $1$ such that the bottom of the split is equal to its top. This means that there exists $\Gamma$, $\Delta$ and $\Lambda$ in $X_n^{j}$  so that if $q_{\Delta}$ is the state of $H$ forced by $\Delta$ and $q_{\Gamma}$ is the state of $H$ forced by $\Gamma$, then we have $\lambda(\Lambda, q_{\Delta}) = \Delta$ and $\lambda(\Lambda, q_{\Gamma}) = q_{\Gamma}$. This now means that $\Lambda = (\Gamma)\underline{H}_{j} = (\Delta)\underline{H}_{j}$ and  $\underline{H}_{j}$ is not injective.
\end{proof}

\begin{corollary}
Let $H \in \shn{n}$ and let $j \in \N$ be greater or equal to the minimal synchronizing level of $H$. If $\underline{H}_{j}$ is not injective then $H$ has infinite order and there is a rational word in $X_n^{\Z}$ an infinite orbit under the action of $H$.
\end{corollary}

\section{Combining elements of \texorpdfstring{$\spn{n}$}{Lg} and some embedding results }\label{embeddings}

In this Section we describe a method for combining two elements of $\spn{n}$ and $\spn{m}$, into a single element of $\spn{m+n}$ in such a way that the property of having finite order is preserved.

Let $A = \gen{X_n, Q_A, \pi_A, \lambda_A}$ and $B = \gen{X_m, Q_B, \pi_B, \lambda_B}$ be elements of $\spn{n}$ and $\spn{m}$ respectively which have finite order. It is a consequence of Claim \ref{cycls<k} that for all $1 \le i \le n$ an $ 1 \le j \le m$ there is a state, $q_i$ of $A$ and $P_j$ of $A$ and $B$ respectively such that $\pi_A(i, q_i) = q_i$ and $\pi_B(j,p_j) = p_j$.

Form a new transducer $\widetilde{B} = \gen{\{n,\ldots,n+m-1\}, Q_{\widetilde{B}}, \pi_{\widetilde{B}}, \lambda_{\widetilde{B}}}$ with input and output alphabet $\{n,\ldots, n+m-1\}$ such that the states of $\widetilde{B}$ are in bijective correspondence with the states of $B$; we denote them  by $\widetilde{q}$, where $q$ is a state of $B$.

The transition and rewrite function of $\widetilde{B}$ are defined by the following rules for $i,j \in X_m$:

\begin{IEEEeqnarray*}{rCl}
\pi_{\widetilde{B}}(n+i, \widetilde{q}) = \widetilde{p} &\iff& \pi_{B}(i,q) = p \nonumber \\
\lambda_{\widetilde{B}}(n+i, \widetilde{q}) = n+j &\iff& \lambda_{B}(i,q) = j
\end{IEEEeqnarray*}

Now we form a new transducer, $A \sqcup B = \gen{ X_{n+m}, Q_{A\sqcup B}, \pi_{A \sqcup B}, \lambda_{A\sqcup B}}$ as follows: 

$Q_{A\sqcup B} = Q_A \sqcup Q_{\widetilde{B}}$; the states of $Q_A$ transition exactly as in $A$ for all inputs in $X_{n}$ and the states of $Q_{\widetilde{B}}$ transition as in $\widetilde{B}$ for all inputs in $X_{n+m}\backslash X_{n}$. Finally for all $i \in X_n$, and for any $\widetilde{q} \in Q_{\widetilde{B}}$, we have $\pi_{A\sqcup B} (i, \widetilde{q}) = p_i$, where $p_i$ is the state of $Q_A$ such that $\pi_{A}(i, p_i) = p_i$, furthermore, $\lambda_{A\sqcup B} (i, \widetilde{q}) = i$. An analogous condition holds in  $A\sqcup B$ for the states of $A$ on inputs in $X_{2n}\backslash X_{n}$. We shall demonstrate this in Example \ref{exampleofcombiningelements}.

We now argue that the order of $A \sqcup B$ is at least $lcm(O(A), O(B))$ --- the lowest common multiple of the orders of $A$ and $B$.

First we show  that $A \sqcup B$ has finite order.

Let $k$ be the maximal bi-synchronizing level of $A$ and $B$. We shall show that $A \sqcup B$ is synchronizing; similar arguments show that $(A\sqcup B)^{-1}$ is also synchronizing at level $k+1$.

\begin{claim} \label{newproductpreservessynch}
$A \sqcup B$ is bi-synchronizing at level $k+1$
\end{claim}
\begin{proof}[Proof 1]
Observe that as soon we read $i$, $i \in X_n$, we must be processing from a state of $A$ and if we read an $n+i$, $i \in X_n$ we must be processing from a state of $\widetilde{B}$.

Let $\Gamma \in X_{2n}^{k+1}$ a word of length $k+1$. 

If ${\Gamma} \in X_n^{k+1}$ or ${\Gamma} \in \{n,\ldots, 2n-1\}^{k+1}$, then the state of $A \sqcup B$ forced by $\Gamma$ is the state of $A$ or $\widetilde{B}$ forced by $\Gamma$. Since reading the first letter guarantees, by the observation in the first paragraph, that the active state is a state of $A$ (or $B$ if $\overline{\Gamma} \in \{n,\ldots, 2n-1\}^{k+1}$), and $A$ and $B$ are bi-synchronizing at level $k$.

Hence we need only consider the case that $\overline{\Gamma}$ contains at least one letter from $X_n$ and one letter from $X_{2n}\backslash X_{n}$. 

Let  ${\Gamma} = g_0 \ldots g_k$. Let $g_i \in X_n$ and assume $g_0 \in X_{2n}\backslash X_n$ (the other case follows by a similar argument) and suppose that $0 < j$ and $j$ is minimal bigger than $i$ such that $g_j \in X_{2n}\backslash X_n$ . By the observation in the first paragraph, regardless of the starting state, after processing the $g_0$, the active state must be some state of $A$. By the minimality of $j$, after processing $g_{j-1}$, the active state is still some state of $A$. Now, notice that every state of $A$ will read $g_j$ to a fixed state $\widetilde{q}_{g_j}$ of $B$. Therefore regardless of the starting position, we always process the final $k-j$ inputs from the state $q_{g_j}$. 

To see that $(A \sqcup B)^{-1}$ is also synchronizing at level  $k+1$ observe that the states corresponding to states $A^{-1}$ in $(A \sqcup B)^{-1}$ process words in $X_n^*$ exactly as $A^{-1}$ does. Moreover all states of $(A\sqcup B)^{-1}$ corresponding to states of $A^{-1}$ read a fixed letter $j$ in $\{n, \ldots, 2n-1\}$  to a unique state $\widetilde{q}_j^{-1}$ corresponding to the state $\widetilde{q}_{j}^{-1}$ of $\widetilde{B}^{-1}$. Analogously for the states of $(A\sqcup B)^{1}$ corresponding to the states of $\widetilde{B}^{-1}$  and  elements of $\{n, \ldots, 2n-1\}^{*}$ and letters in  $X_n$. Therefore we may repeat the argument already given for $(A \sqcup B)$ to show that $(A \sqcup B)^{-1}$ is synchronizing at level $k+1$ also.
\end{proof}

We now free the symbol $k$. To show that $A\sqcup B$ has finite order, we show that its level $k+1$ dual is $\omega$-equivalent to a disjoint union of cycles as in Lemma \ref{lemma 1.3}, where now $k$ is minimal such that both $\dual{A}_{k}$ and $\dual{B}_{k}$ are a disjoint union of cycles.

First we consider the case where $\Gamma \in X_n^{k+1}$ or $\Gamma \in \{n+1, \ldots, n+m-1\}^{k+1}$. Let $\Gamma = g_0 \ldots g_k$. By the assumption that both $A$ and $B$ (hence $\widetilde{B}$) have finite order, this implies that there is a state $q$ of $A$ (or $\widetilde{B}$ if $\Gamma \in \{n+1, \ldots, n+m-1\}^{k+1}$ ) such that the image of $\Gamma$ is in the set $W_q$. This is because after reading $g_0$ we are enter a state of $A$ (or $\widetilde{B}$ if $\Gamma \in \{n+1, \ldots, n+m-1\}^{k+1}$), and using the fact that $g_1\ldots g_k$ belongs to a cycle of states as in Lemma \ref{lemma 1.3} in $\dual{A}_{k}$ (or $\dual{\widetilde{B}}$ if $\Gamma \in \{n+1, \ldots, n+m-1\}^{k+1}$). Moreover if $\Gamma \in X_n^{k}$, then the image of $\Gamma$ through any state of $A\sqcup B$ is also in $X_n^{k}$ and analogously if $\Gamma \in \{n+1, \ldots, n+m-1\}^{k+1}$. Therefore the fact that $\Gamma$  belongs to such a cycle of states is a consequence of the fact that $A$ and $B$ have finite order.

Now we consider the case where $\Gamma$ contains a letter in $X_n$ and a letter in $X_{n+m}\backslash X_{n}$. Similarly to the proof of Claim \ref{newproductpreservessynch}, let $\Gamma = g_0g_1\ldots g_k$. Suppose that  a letter from $X_n$ (the other case being analogous) occurs first and let $j$ be minimal such that $j > 0$ and $g_j \in X_{n+m} \backslash X_n$. The proof of Claim \ref{newproductpreservessynch} shows that the state of $A\sqcup B$ forced by $\Gamma$ depends only on the suffix $g_jg_{j+1}\ldots g_k$. Let $q_{g_{j}}$  be the state of $\widetilde{B}$ such that $\pi_{\widetilde{B}}(g_{j},q_{g_{j}}) = Q_{g_{j}}$. Since every state of $A$ acts as the identity on $X_{n+m}\backslash X_{n}$, it is the case that processing $\Gamma$ from any state of $A \sqcup B$, the first letter of the output is an element of $X_n$ and the $j$\textsuperscript{th} letter is the minimal element in $X_{n+m}\backslash X_{n}$ and is in fact equal to $g_j$, moreover since we process the length $j-k$ suffix from the state $q_{g_{j}}$, the length $j-k$ suffix is independent of the state we begin processing from. However we can now repeat this argument to show that the output through any state of the set of images of $\Gamma$, have  the same $j-k$ suffix and the $j$\textsuperscript{th} letter equal to $g_j$. It now follows from the observation above that the state of $A\sqcup B$ forced depends only on the length $j-k$ suffix, and by induction, that $\Gamma$ belongs to a cycle of states as in Lemma \ref{lemma 1.3}.

The above two paragraphs  show that it is possible to decompose $A \sqcup B$ into a disjoint union of cycles as in Lemma \ref{lemma 1.3} and so $A \sqcup B$ has finite order.

The following alternative way of combining finite order elements of  $\spn{n}$ results in elements of finite order by mechanical substitutions in the arguments above. 

Let $A \in \spn{n}$ and  $B \in \spn{m}$ be as above and form $\widetilde{B}$ as before. For each  $1 \le i \le n$,  let $p_i$ be the state of $A$ such that $\pi_{A}(i,p_i) = p_i$, by the definition of the transformation (in the case where $A$ and $B$ have finite order a permutation), $\overline{A}_{1}$, we have $\lambda_{A}(i, p_i) = \overline{A}_{1}(i)$, likewise there is a state $q_j$ of $B$ such that $\lambda_{B}(j,q_j) = \overline{B}_{1}(j)$ for $1 \le j \le m$. 

We form $A \oplus B$ analogously to $A \sqcup B$. The set of states, and  the transition function, $\pi_{A \oplus B}$ are identical but we make some adjustments to the rewrite function. For any letter $i \in X_n$, and any state $\widetilde{q}$ of $\widetilde{B}$, we take  $\lambda_{A\oplus B}(i, \widetilde{q}) = \overline{A}_{1}(i)$, likewise for any letter $n+j \in \{n, \ldots n+m-1\}$ and any state $p$ of $A$, we have $\lambda_{A\oplus B}(j, p) = n + \overline{B}_{1}(j)$.

The methods described above of combining elements of $\spn{n}$ do not exhaust all possibilities, for instance we could fix a state  of $\widetilde{B}$ such that reading any letter $X_{n+m} \backslash X_n$ from $A$ goes into this state, and likewise we could fix such a state of $\widetilde{B}$. Similar arguments to those given above will show that these methods also give rise to elements of finite order whenever the initial elements have finite order. 

We remark also that as there are new cycles of states introduced in $\dual{(A\sqcup B)}_{k+1}$ and $\dual{(A\oplus B)}_{k+1}$ that are not present in $\dual{A}_{k}$ or $\dual{B}_{k}$ it might be that the order of $A \sqcup B$ is strictly greater than  $lcm(O(A),O(B))$ in some cases.

We give an example below.

\begin{example}\label{exampleofcombiningelements}
Consider the elements of $\T{H}_{3}$ and $\T{H}_{2} \cong C_2$ of order three and order 2 respectively,
\begin{figure}[h!] 
 \begin{center}
\begin{tikzpicture}[shorten >=0.5pt,node distance=3cm,on grid,auto] 

   \node[state] (q_0)   {$q_0$}; 
   \node[state] (q_1) [below left=of q_0] {$q_1$}; 
   \node[state] (q_2) [below right=of q_0] {$q_2$}; 
   \node[state] (q_3) [xshift = 5cm, yshift = -1.5cm]   {$p$}; 
    \path[->] 
    (q_0) edge [in=105,out=75,loop] node [swap]{$1|2$} ()
          edge  [out=335,in=115]node [swap]{$2|0$} (q_2)
          edge  [out=185,in=85]node [swap]{$0|1$} (q_1)
    (q_1) edge[out=65,in=205]  node  [swap]{$2|2$} (q_0)
          edge [in=240,out=210, loop] node [swap] {$0|1$} ()
          edge [out=330,in=210] node [swap]{$1|0$} (q_2)
    (q_2) edge[out=95,in=355]  node [swap]{$1|1$} (q_0)
          edge [out=195,in=345]  node [swap]{$0|2$} (q_1) 
          edge [in=330,out=300, loop] node [swap] {$2|0$} ()
    (q_3) edge [in=130,out=100,loop] node [swap]{$0|1$} ()
          edge [in=80,out=50,loop] node [swap]{$1|0$} ();
\end{tikzpicture}
 \end{center}
 \caption{Elements of $\hn{3}$ and $\hn{2}$.}
 \label{precombination}
\end{figure} 
we now combine them to give an element of order 6 in $\T{H}_{5}$. Using any of the methods describe above will yield an element of order 6, we only illustrate one such method.

\begin{figure}[h!]
\begin{center}
\begin{tikzpicture}[shorten >=0.5pt,node distance=3cm,on grid,auto] 

   \node[state] (q_0)   {$q_0$}; 
   \node[state] (q_1) [xshift = -4.5cm, yshift = -3.5cm] {$q_1$}; 
   \node[state] (q_2) [xshift = 4.5cm, yshift = -3.5cm] {$q_2$}; 
   \node[state] (q_3) [xshift = 0cm, yshift = -3cm]   {$p$}; 
    \path[->] 
    (q_0) edge [in=105,out=75,loop] node [swap]{$1|2$} ()
          edge  [out=342,in=108]node [swap]{$2|0$} (q_2)
          edge  [out=185,in=85]node [swap]{$0|1$} (q_1)
          edge  [out=260, in=100] node [swap] {$3|3$,$4|4$} (q_3)
    (q_1) edge[out=72,in=198]  node  [swap]{$2|2$} (q_0)
          edge [in=240,out=210, loop] node [swap] {$0|1$} ()
          edge [out=330,in=210] node [swap]{$1|0$} (q_2)
          edge [out=20, in=180] node  {$3|3$, $4|4$} (q_3)
    (q_2) edge[out=95,in=355]  node [swap]{$1|1$} (q_0)
          edge [out=202,in=338]  node [swap]{$0|2$} (q_1) 
          edge [in=330,out=300, loop] node [swap] {$2|0$} ()
          edge [out=160, in=0] node[swap] {$3|3$,$4|4$} (q_3)
    (q_3) edge [in=160,out=130,loop] node [swap]{$3|4$} ()
          edge [out=50,in=20,loop] node         {$4|3$} ()
          edge [out=80, in=280] node [swap] {$1|1$} (q_0)
          edge [out= 190, in= 0] node {$0|0$} (q_1)
          edge [in= 180, out= 350] node[swap] {$2|2$}(q_2); 
\end{tikzpicture}
 \end{center}
 \caption{The result of combining the elements in Figure \ref{precombination}.}
 \label{postcombination}
 \end{figure}
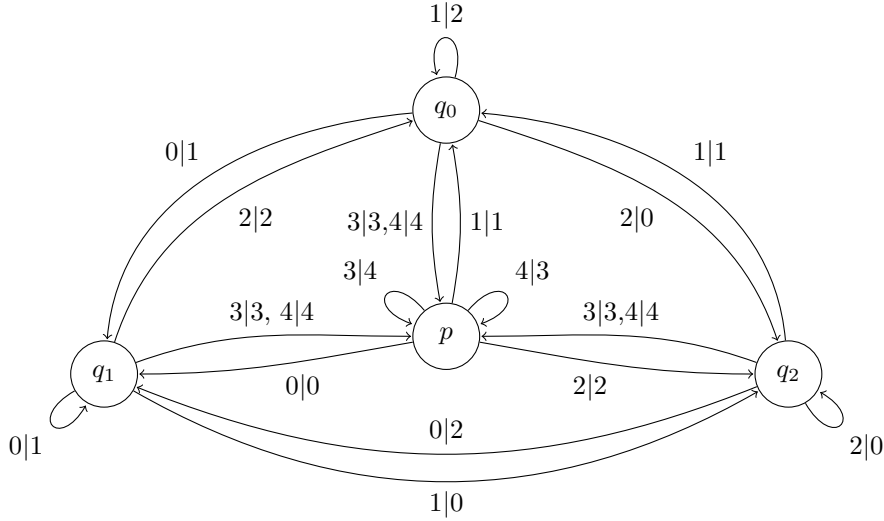
\end{example}

\subsection{Embedding direct sums of \texorpdfstring{$\spn{m}$}{Lg} in \texorpdfstring{$\spn{n}$}{Lg} for \texorpdfstring{$n$}{Lg} large enough}

The aim of this section is to show that for given $n \in \mathbb{N}$ and for an increasing sequence of non-zero natural numbers $d_1 \le d_2 \le \ldots \le  d_l$ such that $\sum_{i=1}^{l} d_{i} = n$, the semigroup $\spn{n}$ contains a subsemigroup isomorphic to $\spn{d_1} \times \spn{d_2} \times \ldots \times \spn{d_l}$. This will then yield, as a corollary, that $\hn{n}$ contains a subgroup isomorphic to $\hn{d_1} \times \hn{d_2} \times \ldots \times \hn{d_l}$. In order to do this, we first extend the results of the previous. Essentially we shall simultaneously merge the elements of  $\hn{d_i}$,  as opposed to inductively applying the construction in the previous section, this allows us better control the synchronizing level of the resulting transducers.

Let $n \in \mathbb{N}$ and let $d_i$ $1 \le i \le l$ be as in the previous paragraph. Let $X_0 := \{0,1,\ldots, d_1  -1\}$ and for  $2 \le i \le l$ let $X_i := \{\sum_{j=1}^{i-1}d_{j}, \sum_{j=1}^{i-1}d_{j} +1, \ldots, \sum_{j=1}^{i}d_{j} - 1\}$. Furthermore let $A_{i}$ $1 \le i \le l$ be synchronous synchronizing transducers on the alphabet $X_i$. We shall now describe how to form $\sqcup_{i=1}^{l} A_i \in \spn{n}$, which will simply be an extension of the 2 element case described in the previous Section.

For each $i \in \mathbb{N}$ let $\overline{A_{i}}$ denote the transformation of the words of length $1$ induced by $A_{i}$ as in Remark \ref{transformations}. The transducer $\sqcup_{i=1}^{l} A_i := \gen{X_n, \sqcup_{i=1}^{l}Q_{A_i}, \pi_{\sqcup}, \lambda_{\sqcup}}$ will consists of the disjoint union of copies of the $A_i$ which are connected in a specific way. Fix a $j$ such that $1 \le j \le l$, and consider the copy of $A_j$ in $\sqcup_{i=1}^{l} A_i$. Then the copy of $A_j$ in $\sqcup_{i=1}^{l} A_i \in \spn{n}$ transitions precisely as $A_j$ does when restricted to $X_{j}$; we now describe how $A_j$ acts for inputs not in  $X_j$. 

Let $1 \le i \le l$ be arbitrary, then for any $x_i \in X_i$, there is a unique state $q_{x_i} \in Q_{A_i}$ such that  $\pi_{A_{i}}(x_i, q_{x_i}) = q_{x_i}$ and $\lambda_{A_{i}}(x_i, q_{x_i}) = (x_i)\overline{A_{i}}$. Therefore in  $\sqcup_{i=1}^{l} A_i \in \pn{n}$ we set $\pi_{\sqcup}(x_i, A_j) = q_{x_i}$ and $\lambda_{\sqcup}(x_i, A_j) = (x_i)\overline{A_{i}}$. Hence we have now described how the copy of $A_{j}$ acts on all inputs in $\sqcup_{i\ne j, 1\le i \le l}X_i$.

Repeating the above for each $A_{j}$, $ 1 \le j \le l$, we now have that $\sqcup_{i=1}^{l} A_i $ is connected and all  states are defined on $X_n:= \{0,1 \ldots n-1\}$.

The proof that $\sqcup_{i=1}^{l} A_i \in \spn{n}$ (i.e that the resulting transducer is synchronizing) requires only the obvious amendments to the 2 element case proven in the previous Section. Therefore the following theorem is valid:

\begin{Theorem}
Let $n \in \mathbb{N}$ and let $d_i$ $1 \le i \le l$ be an increasing sequence of non-zero natural numbers such that $\sum_{i=1}^{l}d_i = n$. Let $X_1 := \{0,1,\ldots, d_1  -1\}$ and for  $2 \le i \le l$ let $X_i := \{\sum_{j=1}^{i-1}d_{j} - 1, \sum_{j=1}^{i-1}d_{j}, \ldots, \sum_{j=1}^{i}d_{j} - 1\}$. Furthermore let $A_{i}$ $1 \le i \le l$ be synchronous synchronizing transducers on the alphabet $X_i$. Then $\sqcup_{i=1}^{l}A_i$ is an element of $\spn{n}$. If $k_i$ is the synchronizing level each $A_i$, $1 \le i \le l$ then the synchronizing level of $\sqcup_{i=1}^{l}A_i$ is at most $\max_{1 \le i \le l}\{k_i\} +1 $.
\end{Theorem}

\begin{Remark}
If we begin with elements $A_i \in \pn{d_{i}}$ acting on the alphabet $X_i$, such that one of the $A_i$ does not possess a homeomorphism state, then the resulting transducer $\sqcup_{k=1}^{l}A_{k}$ does not represent a homeomorphism of $X_n^{\mathbb{Z}}$.
\end{Remark} 
\begin{proof}
To see this let $i \in \{1\ldots l\}$ be such that $A_i$ does not posses a homeomorphism state. Then since $A_i$ is a synchronous transducer there is a state $q_i$ of $A_i$ and $x_i, y_i \in X_i$ such that $\lambda_{A_i}(x_i,q_i) = \lambda_{A_i}(x'_i,q_i)$. Let $p_i:= \pi_{A_i}(x_i,q_{i})$ and $p'_i=\pi_{A_{i}}(x'_i,q_{i})$. Let $q_{z_i}$ be the state of $A_i$ such that $\pi_{A_i}(z_i, q_{z_{i}}) = q_{z_{i}}$, and let $\Gamma_i$ in $X_i^{\ast}$ be a path from $q_{z_{i}}$ to $q_i$. Furthermore let $x_j \in X_j$ for $j \in \{1,\ldots,l\}\backslash\{i\}$, and let $q_{x_{j}}$ be the state of $A_{j}$ such that $\pi_{A_{j}}(x_j,q_{x_{j}})=q_{x_{j}}$.

Now observe that by definition of $\sqcup_{k=1}^{l}A_{k}$, the words $x_jz_i\Gamma_ix_ix_j$ and $x_jz_i\Gamma_ix'_ix_j$ are such that $\pi_{\sqcup}(x_jz_i\Gamma_ix_ix_j,q_{x_{j}}) = q_{x_{j}}$ and $\pi_{\sqcup}(x_jz_i\Gamma_ix'_ix_j,q_{x_{j}}) = q_{x_{j}}$. Moreover $\lambda_{\sqcup}(x_jz_i\Gamma_ix_ix_j,q_{x_{j}})=\lambda_{\sqcup}(x_jz_i\Gamma_ix'_ix_j,q_{x_{j}})$.

Therefore $\sqcup_{k=1}^{l}A_{k}$ maps the bi-infinite strings $\ldots(x_jz_i\Gamma_ix_ix_j)\ldots$ and $\ldots(x_jz_i\Gamma_ix_ix_j)\ldots$ to the same element of $X_n^{\mathbb{Z}}$, and so it is not injective. 
\end{proof}
\begin{Remark}
If we restrict instead to $\hn{d_i}$ instead of $\pn{d_i}$ then  the resulting transducer $\sqcup_{k=1}^{l}A_{k}$ will be in $\hn{n}$ as we will see below.
\end{Remark}

It was shown in the 2 element case, that for $A$ and $B$ acting on alphabets $X_1$ and $X_2$ such that $X_1 \sqcup X_2 := \{0,1,\ldots,n-1\}$ then $A \sqcup B$ has finite order if and only if $A$ and $B$ have finite order. In this more general setting we prove the following stronger result.

\begin{Theorem}\label{embeddingdirectsumsofshift}
Let $n \in \mathbb{N}$ and let $d_i$ $1 \le i \le l$ be an increasing sequence of non-zero natural numbers such that $\sum_{i=1}^{l}d_i = n$. Let $X_1 := \{0,1,\ldots, d_1  -1\}$ and for  $2 \le i \le l$ let $X_i := \{\sum_{j=1}^{i-1}d_{j}, \sum_{j=1}^{i-1}d_{j}, \ldots, \sum_{j=1}^{i}d_{j} - 1\}$. By an abuse of notation let $\spn{d_i}$ denote the monoid of synchronous, synchronizing transducers  on the alphabet $X_{i}$. Then the map $\phi: \bigoplus_{i=1}^{l} \spn{d_i} \to \spn{n}$, $(A_1, \ldots, A_l) \mapsto \sqcup_{i=1}^{l}A_i$ is a monomorphism.
\end{Theorem}
\begin{proof}
That this map is injective follows from that the fact that the action of each $A_i$ on $X_i^{\mathbb{Z}}$ is replicated exactly when we restrict $\sqcup_{i=1}^{l}A_i$ to $X_i^{\mathbb{Z}}$. Therefore we need only prove that $\phi$ is a homomorphism.

Let $(A_1, \ldots, A_l)$ and $(B_1 \ldots, B_l)$ be elements of $\phi: \bigoplus_{i=1}^{l} \spn{d_i}$, and let $(C_1, \ldots, C_l)$ be their product, hence $C_i = \core(A_i \ast B_i)$. We shall show that the $D:=\core(\sqcup_{i=1}^{l}A_i * \sqcup_{i=1}^{l}B_i) = \sqcup_{i=1}^{l}C_i$.

First notice that for $q_i \in Q_{A_i}$ and  $p_j \in Q_{B_j}$ the  pair $(q_i, p_j)$ is not a state of $D$. This is because for any word $\Gamma \in X_n^{\ast}$ such that the state of $\sqcup_{r=1}^{l}A_r$ forced by $\Gamma$ is $q_i$, then $\Gamma$ must have a non-empty suffix in $X_i^{\ast}$, and hence so also must its output through any state of $\sqcup_{r=1}^{l}A_r$ by construction. Therefore the output of $\Gamma$ through any state of $\sqcup_{r=1}^{l}B_r$ will synchronise to a state in $Q_{B_i}$. Therefore the states of $D$ are precisely a subset of $\sqcup_{i=1}^{l}Q_{A_{i}} \times Q_{B_{i}}$.

Now since the states of $D$ intersecting $Q_{A_{i}} \times Q_{B_{i}}$ arising from the transducer product $A_i \ast B_i$ form precisely the sub-transducer $C_i$, therefore to conclude the proof it suffices (by the injectivity of $\phi$) to show two things. Firstly, that for $j \ne i$ all states of $A_j \times B_j$ act on $X_i$ precisely as $\overline{C_{i}}_{1}= \overline{A_{i}}_{1} \times \overline{B_{i}}_{1}$ (the final equality follows from Claim \ref{cyclswellbhvd}). Secondly, that all states $(q_j, p_j)$ of $A_j \times B_j$ read an $x_i \in X_i$ into the unique state of $C_i$ with a loop labelled by $x_i$.

The first part follows from the following observation. By construction for any $j \ne i$ the copy of $A_j$ in $\sqcup_{i=1}^{l}A_{i}$ acts on $X_i$ precisely as $\overline{A_i}_{1}$ does, similarly in $\sqcup_{i=1}^{l}B_{i}$. Now by Claim \ref{cyclswellbhvd}, $\overline{C_{i}}_{1}:= \overline{A_{i}}_{1} \times \overline{B_{i}}_{1}$. Therefore the first part is proved.

For the second part consider the following. Notice that for any $j \ne i $ and for any state $q_j$, a state of the copy of $A_j$ in  $\sqcup_{i=1}^{l}A_{i}$, and for any $x_i \in X_i$ we have that $ \pi_{\sqcup A}(x_i, q_j) = q_{x_i}$, where $\pi_{\sqcup A}$ is the transition function of $\sqcup_{i=1}^{l}A_{i}$, and $q_{x_i}$ is the unique state of $A_{i}$ such that $\pi_{A_i}(x_i, q_{x_i}) = q_{x_i}$. An analogous statement holds for $\sqcup_{i=1}^{l}B_{i}$. Therefore given $(q_j, p_j) \in A_j \times B_j$, and $x_i \in X_i$, we have $\pi_{\sqcup D}(x_i, q_j, p_j) = (\pi_{\sqcup A }(x_i, q_j),\pi_{\sqcup B}((x_i)\overline{A_i}_{1},p_j))$, however this is simply the state $(q_{x_i},p_{(x_i)\overline{A_i}_{1}})$ of $A_{i} \times B_{i}$. However, since by definition of $\overline{A_i}_{1}$, $(x_i)\overline{A_i}_{1} = \lambda_{A_i}(x_i, q_{x_i})$, then $(q_{x_i},p_{(x_i)\overline{A_i}_{1}})$ is precisely the unique state of $C_i$ with a loop labelled by $x_i$.

\end{proof}

\begin{Remark}
It is straight-forward to see from the above that $\phi$ maps $\bigoplus_{i=1}^{l}\hn{d_i}$ to a subgroup of $\hn{n}$ since $\bigoplus_{i=1}^{l}\hn{d_i}$ is a subgroup of $\spn{n}$. 
\end{Remark}

\begin{corollary}
Let $n \in \mathbb{N}$ and let $d_i$ $1 \le i \le l$ be an increasing sequence of non-zero natural numbers such that $\sum_{i=1}^{l}d_i = n$. Let $X_1 := \{0,1,\ldots, d_1  -1\}$ and for  $2 \le i \le l$ let $X_i := \{\sum_{j=1}^{i-1}d_{j} - 1, \sum_{j=1}^{i-1}d_{j}, \ldots, \sum_{j=1}^{i}d_{j} - 1\}$. By an abuse of notation let $\spn{d_i}$ denote the monoid of synchronous synchronizing transducers  on the alphabet $X_{i}$. Given $(A_1, \ldots, A_l) \in \bigoplus_{i=1}^{l} \spn{d_i}$, $\sqcup_{i=1}^{l}A_i$ has finite order if and only if each of the $A_i$'s have finite order. Moreover the order of $\sqcup_{i=1}^{l}A_i$ is precisely the lowest common multiple of the orders of the $A_i$. 
\end{corollary}
\begin{proof}
This is a consequence of Theorem \ref{embeddingdirectsumsofshift} and well known results about direct sums of groups. 
\end{proof}

It is a result by Boyle, Franks and Kitchens  \cite{BoyleFranksKitchens90} that for $n \ge 3$ $\hn{n}$  contains free groups. Therefore we have the following corollary:

\begin{corollary}
Let $n \ge 3$ and let $m$ and $l$ be natural numbers such that $n = 3m +l$ where $0 \le l \le 3$. Then $\hn{n}$ contains a subgroup isomorphic to $\Pi_{i=1}^{m} F_2$ where $F_2$ is the free group on two generators.
\end{corollary}

Notice that since $F_2 \times F_2$  has undecidable subgroup membership problem, it follows that for  $n \ge 6$ $\hn{n}$  has undecidable subgroup membership problem.

\begin{Remark}\label{futherconstructions}
We can modify the construction above. Let $n$, $d_i$ and $A_i$, $1\le i \le l$, be as before. For each $A_i$ fix a permutation $\sigma _{i}$ of $X_i$ and an element $S_i \in Q_{A_i}^{d_i}$. Then we may form the transducer $\sqcup_{i=1, \sigma_i, S_i}^{l} A_i = \gen{X_n, \sqcup_{i=1}^{l}A_{i},\lambda_{\sqcup}, \pi_{\sqcup}}$. For a given $j \in \{1,\ldots l\}$ the copy of $A_j$ in $\sqcup_{i=1, \sigma_i, S_i}^{l} A_i$ is precisely $A_j$ when restricted to $X_j$. However for $i \ne j$, and any state $q_j$ of  $A_j$, then given any $x_i \in X_i$ we have $\lambda_{\sqcup}(x_i, q_j) = (x_i)\sigma_{i}$; $\pi_{\sqcup}(x_i, q_j)$ is the entry of $S_i$ corresponding to the position of $x_i$ when the elements of $X_i$ are ordered according to the natural ordering induced from $\mathbb{N}$. Then  once more the resulting element of $\spn{n}$ is synchronizing and has finite order if and only if all the $A_i$'s have finite order.
\end{Remark}

\subsection{On the difference between the synchronizing and bi-synchronizing level}

Using the techniques developed above we shall now construct a class of examples of finite order elements which are all synchronizing at level 1, but whose inverses are synchronizing at the maximum possible level for the given number of states. A side-effect of the construction is that the alphabet size increases with the gap in the size of the synchronizing and bi-synchronizing level.

Our base transducer $B$ is the transducer in Figure \ref{BandApremerge} to the left. Let $A$ be the transducer on the right.
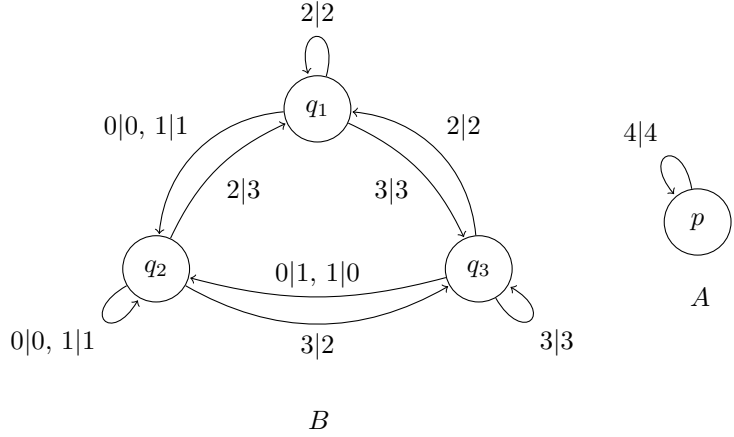
\begin{figure}[H]
 \begin{center}
\begin{tikzpicture}[shorten >=0.5pt,node distance=3cm,on grid,auto] 

   \node[state] (q_0)   {$q_1$}; 
   \node[state] (q_1) [below left=of q_0] {$q_2$}; 
   \node[state] (q_2) [below right=of q_0] {$q_3$}; 
   \node[state] (q_3) [xshift = 5cm, yshift = -1.5cm]   {$p$}; 
    \path[->] 
    (q_0) edge [in=105,out=75,loop] node [swap]{$2|2$} ()
          edge  [out=335,in=115]node [swap]{$3|3$} (q_2)
          edge  [out=185,in=85]node [swap]{$0|0$, $1|1$} (q_1)
    (q_1) edge[out=65,in=205]  node  [swap]{$2|3$} (q_0)
          edge [in=240,out=210, loop] node [swap] {$0|0$, $1|1$} ()
          edge [out=330,in=210] node [swap]{$3|2$} (q_2)
    (q_2) edge[out=95,in=355]  node [swap]{$2|2$} (q_0)
          edge [out=195,in=345]  node [swap]{$0|1$, $1|0$} (q_1) 
          edge [in=330,out=300, loop] node [swap] {$3|3$} ()
    (q_3) edge [in=130,out=100,loop] node [swap]{$4|4$} ();
     \node [xshift=2.5 cm, yshift= -2cm,text width=1cm] at (q_1)
    	          {
    	          $B$ };
    \node [xshift=0.4 cm, yshift= -1cm,text width=1cm] at (q_3)
        	          {
        	          $A$ }; 	          
    	          
\end{tikzpicture}
 \end{center}
 \caption{The base transducer $B$ and an element of $\hn{1}$}
 \label{BandApremerge}
 \end{figure}
 
 Notice that $B$ is synchronizing at level $1$ but bi-synchronizing at level $2$. It is a consequence of the collapsing procedure (see \cite{BlkYMaisANav}) that a transducer with $j$ states is synchronizing at level at most $j-1$ since we must identity two states at each step of the algorithm. Therefore $B^{-1}$ attains the maximum synchronizing level for a 3 state transducer. One can check that  $B$ has order $4$.
 
 Now we attach $A$ to $B$ using the construction described in Remark \ref{futherconstructions}. Let $\sigma_2$ be any permutation of $\{0,1,2,3\}$ that maps $0$ to $3$. There is only one permutation of $\{4\}$; form $\sqcup_{i=1, \sigma_i, S_i}^{2} C_i = \gen{X_n, \sqcup_{i=1}^{2}C_{i},\lambda_{\sqcup}, \pi_{\sqcup}}$ where $C_1 = B$, $C_2 = A$, and $\sigma_1$ is the identity map. The resulting transducer is as shown in Figure \ref{B'resultofmerge} to the left:
\begin{figure}[h!]
\begin{center}
\begin{tikzpicture}[shorten >=0.5pt,node distance=3cm,on grid,auto] 

   \node[state] (q_0)   {$q_1$}; 
   \node[state] (q_1) [xshift = -4.5cm, yshift = -3.5cm] {$q_2$}; 
   \node[state] (q_2) [xshift = 4.5cm, yshift = -3.5cm] {$q_3$}; 
   \node[state] (q_3) [xshift = 0cm, yshift = -3cm]   {$p$}; 
   \node[state] (q_4) [xshift=7cm, yshift = -3cm] {$p'$};
    \path[->] 
    (q_0) edge [in=105,out=75,loop] node [swap]{$2|2$} ()
          edge  [out=342,in=108]node [swap]{$3|3$} (q_2)
          edge  [out=185,in=85]node [swap]{$0|0$, $1|1$} (q_1)
          edge  [out=260, in=100] node [swap] {$4|4$} (q_3)
    (q_1) edge[out=72,in=198]  node  [swap]{$2|3$} (q_0)
          edge [in=240,out=210, loop] node [swap] {$0|0$, $1|1$} ()
          edge [out=330,in=210] node [swap]{$3|2$} (q_2)
          edge [out=20, in=180] node  {$4|4$} (q_3)
    (q_2) edge[out=95,in=355]  node [swap]{$2|2$} (q_0)
          edge [out=202,in=338]  node [swap]{$0|1$, $1|0$} (q_1) 
          edge [in=330,out=300, loop] node [swap] {$3|3$} ()
          edge [out=160, in=0] node[swap] {$4|4$} (q_3)
    (q_3) edge [in=160,out=130,loop] node [swap]{$4|4$} ()
          edge [out=80, in=280] node [swap] {$2|1$} (q_0)
          edge [out= 190, in= 0] node {$0|3$, $1|0$} (q_1)
          edge [in= 180, out= 350] node[swap] {$3|2$}(q_2)
    (q_4) edge [loop above] node {$5|5$}();
     \node [xshift=0.4cm, yshift= -3cm,text width=1cm] at (q_3)
             	          {
             	          $B'$ }; 
    \node [xshift=0.4 cm, yshift= -1cm,text width=1cm] at (q_4)
            	          {
            	          $A'$ };

\end{tikzpicture}
 \end{center}
 \caption{The resulting transducer $B'$ which is a merge of $B$ and $A$ and an element $A'$ of $\hn{1}$ }
\label{B'resultofmerge}
 \end{figure}
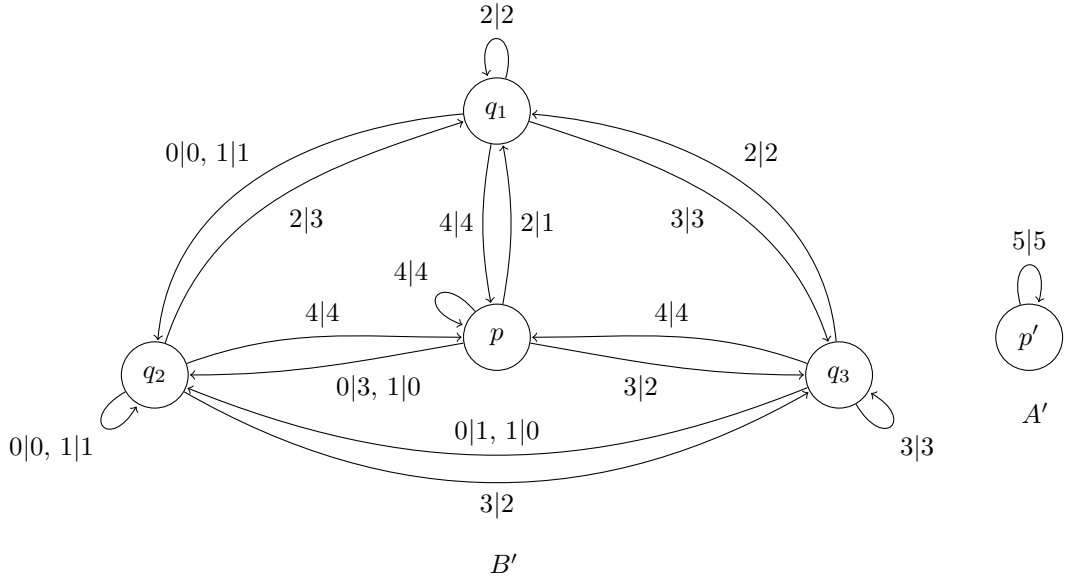
 
Since all the states $B$ map $\{0,1\} \to \{0,1\}$, $p$ is not $\omega$-equivalent to any state of $B':= \sqcup_{i=1, \sigma_i, S_i}^{2} C_i$. Moreover notice that since there is a path  $p \stackrel{0|3}{\to} q_2 \stackrel{2|3}{\to} q_1 \stackrel{3|3}{\to} q_4$, in order to identify $p^{-1}$ with $q_1^{-1}$, $q_2^{-1}$ and $q_3^{-1}$ in $B'$ we must first have identified $q_1^{-1}, q_2^{-1}$ and $q_3^{-1}$. Therefore since $B^{-1}$ is synchronizing at level $2$, it takes 3 steps to collapse $B'^{-1}$ to a single state. It follows  that $B'$ is bi-synchronizing at level $3$ and synchronizing at level $1$.

Now since all the states of $B'$ fix $4$, we can repeat the process. Let $A'$ be the transducer to the right of Figure \ref{B'resultofmerge}, and let $\sigma_2'$ be any permutation of $\{0,1,2,3,4\}$ that maps $4$ to $3$. Then by the repeating the arguments above the transducer $B'':= \sqcup_{i=1, \sigma_i', S_i}^{2} C_i = \gen{X_n, \sqcup_{i=1}^{2}C_{i},\lambda_{\sqcup}, \pi_{\sqcup}}$ where $C_1 = B'$ $C_2 = A'$, and $\sigma_1'$ is the identity map, is bi-synchronizing at level $4$ and synchronizing at level 1. We may continue on in this way.

Notice that since the initial transducers $B$ and $A$ have finite order, then by Remark \ref{futherconstructions} all the transducers $B'$, $B''$ and so on have finite order.

\section{The automaton group generated by a synchronizing transducer has exponential growth}\label{proofofmainresult}

The results of this section connect the graph theoretic properties of the graph of bad pairs to the existence of free subsemigroups in the automaton semigroup generated by an element of $\hn{n}$. We begin with the following proposition:

\begin{proposition}\label{looporfreesmgp}
Let $A \in \widetilde{\T{H}}_{n}$ be an element of infinite order. Then either there is a $j \in \mathbb{N}$ such that the minimal graph of bad pairs $\overline{G}_{j}(A)$ has a loop or the automaton semigroup generated by $A$ contains a free semigroup of rank at least 2.
\end{proposition}
\begin{proof}
We may assume, by changing the alphabet size that $A$ is bi-synchronizing at level 1.

Since $A$ has infinite order then for each $j \in \mathbb{N}$, $(\dual{A})^{j}$ splits. Fix $j \in \N$ and let $\mathrm{Top}_{j}$ be the set of pairs of states $\{p_1,p_2\}$ such that there exits $(p_1, s_1, \ldots, s_r)$ and $(p_2,s_1',\ldots,s_r')$ which split  $(\dual{A})^j$ where $r$ is the minimal splitting length of $(\dual{A})^j$. By definition $\mathrm{Top}_{j}$ is the set of tops of minimal length splits of $\dual{A}_{j}$.  Analogously, for fixed $j \in \mathbb{N}$ let $\mathrm{Bottom}_{j}$ be the set of bottoms of minimal length splits. That is $\mathrm{Bottom} _{j}$ consists of sets $\{t_1, t_2\}$ such that there exists a split of minimal length $(\Gamma,(s_1, \ldots, s_r),(s_1',\ldots,s_r'))$ of $\dual{A}_{j}$ with bottom $\{t_1, t_2\}$.

Since $A$ has finitely many states there exists an infinite subset $\mathcal{J}'\subset \mathbb{N}$ and a fixed set of pairs $\{t_1,t_2\}$ such that $\{t_1,t_2\} \in \mathrm{Bottom}_{j}$ for all $j \in \T{J}'$. Now consider the set of tops of all splits of $\dual{A}_{j}$, $j \in \T{J}'$ with bottom $\{t_1, t_2\}$. Since $|\T{J}'|  = \infty$ and $\{t_1, t_2\} \in \mathrm{Bottom}_{j}$ for all $j \in \mathcal{J}'$, there exists an infinite subset $\T{J} \subset \T{J}'$ and a fixed set of pairs $\{p_1,p_2\}$ such that $\{p_1,p_2\} \in \mathrm{Top}_{j}$ for all $j \in \T{J}$ and there exists splits of $\dual{A}_{j}$ with top  $\{p_1, p_2\}$ and bottom  $\{t_1, t_2\}$ for all $j \in \T{J}$.

If $\{p_1,p_2\} = \{t_1,t_2\}$  we are done, since  $\overline{G}_{j}(A)$ has a loop for any $j \in \T{J}$. Therefore assume that this is not the case. Under this assumption, we have two cases to consider.

{\bf{Case 1}}: Suppose that there are $i, i' \in \mathbb{N}$, $i,i' \ge 1$ and $S_1, S_1' \in Q^{i}$ and $S_2, S_2' \in Q^{i'}$ such that $(p_1, S_1)$ is $\omega$-equivalent to $(t_1, S_1')$ and $(p_2,S_2)$ is $\omega$-equivalent to $(t_2,S_2')$. We may assume that $i = i'$ by padding out one of the pairs $(p_i,S_i)$ and $(t_i, S_i')$, $i=1,2$.

Let $j \in \T{J}$ be such that $j > i+1$. Consider $(\dual{A})^{j}$, it has minimal splitting length, $r$, greater than or equal to $j$. Now by choice of $\{p_1, p_2\}$, there exists $(p_1, s_1, \ldots, s_{r-1})$, $(p_2, s_1', \ldots, s_{r-1}')$ elements of $Q^r$ and $\Gamma$ a state of $(\dual{A})^{j}$ such that $((p_1, s_1, \ldots, s_{r-1}), (p_2, s_1', \ldots, s_{r-1}'), \Gamma)$  is a split of $(\dual{A})^{j}$ with bottom $\{t_1, t_2\}$. Hence, by minimality of $r$, it now follows that $((p_1, S_1, s_{i+1}, \ldots, s_{r-1}), (p_2, S_2, s_{i+2}',\ldots, s_{r-1}', \Gamma)$ is also a split of $(\dual{A})^{j}$ with bottom $(t_1,t_2)$. However this now implies, again by minimality of $r$ and since $(t_1,S_1')$ and $(t_2,S_2')$ are $\omega$-equivalent to $(p_1,S_1)$ and $(p_2,S_2)$ respectively, that $((t_1, S_1', s_{i+1}, \ldots, s_{r-1}),( t_2, S_2, s_{i+2}',\ldots, s_{r-1}'), \Gamma)$ is also a split of $(\dual{A})^{j}$ with bottom $(t_1,t_2)$. Therefore $(\dual{A})^{j}$ has a loop.

{\bf{Case 2}}: We assume that Case 1 does not hold, that is  for all   $i, i' \in \N$ there does not exist a choice of $S_1, S_1' \in Q^{i}$  and $S_2, S_2' \in Q^{i'}$ such that $(p_1,S1)$ is $\omega$-equivalent to $(t_1, S_1')$ and  $(p_2, S_2)$ is $\omega$-equivalent to  $(t_2, S_2')$ . We may also assume that none of the graph of bad pairs $\overline{G}_{j}(A)$ has a loop for any $j$ greater than the minimal synchronizing level of $A$, since otherwise we are done. 

The latter assumption implies that $\mathfrak{S}_{j, |Q_A|^2+1}$ consists of transformations with image size 1 by Remark \ref{ifinfiniteorderwithnocircuitsthenmusthavezero}. However since $(\dual{A})^{j}$ splits for every $j \in \mathbb{N}$, then for $j$ larger than the minimal synchronizing level, there are elements  $\Gamma \in X_n^j$, such that $\sigma_{\Gamma}$ has image size at least 2. Fix an arbitrary such $\Gamma$. Since $\sigma_{\Gamma}^{|Q_A|^2+1}$ has image size 1, then there is a state $p \in Q_{A}$ such that $(p)\sigma_{\Gamma} = p$. Therefore there is a pair of states $p_1, p_2 \in Q_{A}$ such that there is a split of $(\dual{A})^{j}$ with top $\{p_1, p_2\}$, and bottom $\{t_1, p_2\}$.   

The above argument now implies that we may chose $\{p_1, p_2\}$ and $\{t_1, t_2\}$ above so that $p_2 = t_2$, and there exists $\Gamma \in X_n^{j}$, $j \in \T{J}$, such that $(p_1)\sigma_{\Gamma}  = t_1$ and $(p_2)\sigma_{\Gamma} = p_2$. 

 Now since case one does not hold, and $p_2 = t_2$, therefore it follows that for any $m \in \mathbb{N}\backslash\{0\}$ and any $S_1, S_2 \in Q^{m}$ that $p_1S_1$ is not $\omega$-equivalent to $t_1S_2$. We now argue that the sub-semigroup $\gen{p_1, t_1}$ of  $\calscr{S}(A)$ (the automaton semigroup generated by $A$) is free.

Now as $A$ has infinite order, it follows that $Core(A^i) \nwequal Core(A^j)$ for any $i \ne j \in \mathbb{N}$. Therefore given two words $v$ and $w $ in $\gen{p_1, t_1}$ such that $v$ and $w$ are $\omega$-equivalent it follows that $|v| = |w|$.

Therefore consider the case of words $v, w \in  \gen{p_1,t_1}$ such that $|v| = |w|$. Suppose $v = v_1 \ldots v_l$ and $w = w_1 \ldots w_l$, where $|v| = |w| = l$. Let $1 \le i \le l$ be the minimal index so that $v_i \ne w_i$. We may assume that $v_i = p_1$ and $w_i = t$. Therefore $v = v_1\ldots v_i p_1 v_{i+2} \ldots v_{l}$ and $w = v_1\ldots v_i t_1 w_{i+2} \ldots w_{l}$. Hence $v$ is $\omega$-equivalent to $w$ if and only if $p_1 v_{i+2} \ldots v_{r}$ is $\omega$-equivalent to $t_1 w_{i+2} \ldots w_{r}$. However  by assumption this is not the case.  Therefore given any two distinct words in $\{p_1,t_1\}^{\ast}$, they represent distinct automorphisms of the $n$-ary rooted tree hence we conclude that $\gen{p_1, t_1}$ is a free semigroup.
\end{proof}

\begin{corollary}
Let $A \in \widetilde{\T{H}}_{n}$ be an element of infinite order. Then either there is a $j \in \mathbb{N}$ such that the  graph of bad pairs $G_{j}(A)$ has a loop otherwise the automaton semigroup generated by $A$ contains a free semigroup of rank at least 2.
\end{corollary}

In the proposition below we introduce a condition on the graph of bad pairs $G_{j}(A)$ which guarantee the existence of free subsemigroups of certain rank in the automaton semigroup generated by an element of $\shn{n}$. This condition at first glance appears to be very strong, however we shall introduce a large class of examples which satisfy the hypothesis of the Proposition. In particular whenever the graph of bad pair has a loop  the hypothesis is immediately satisfied.

\begin{proposition} \label{infinite=expgrowth1}
Let $A \in \widetilde{\T{H}}_{n}$ and suppose that $A$ is synchronizing at level $k$ and is minimal. Let $G_{j}(A)$ be the graph of bad pairs for some $j \ge k \in \N$. Suppose there is a subset $\T{S}$ of the set of states of $A$, such that the following things hold:
\begin{enumerate}[label = (\roman*)]
\item $|\T{S}| \ge 2$,
\item the set $\T{S}(2)$ of two element subsets of $\T{S}$ is a subset of the vertices of $G_{j}(A)$,
\item for each element of $\T{S}(2)$ there is a vertex  accessible from it which belongs to a circuit.
\end{enumerate}
Then the automaton semigroup generated by $A$ contains a free semigroup of rank at least $|\T{S}|$. In particular the automaton semigroup generated by $A$ has exponential growth.
\end{proposition}
\begin{proof}
First observe that since $G_{j}(A)$ is assumed to have a circuit, by Lemma \ref{lemmagphofbdprs} $A$ has infinite order. Now let $U$ and $V$ be distinct non-empty words in $\T{S}^{\ast}$, if $|U| \ne |V|$ then since $A$ has infinite order  $A_{U}$ cannot be $\omega$-equivalent to $A_{V}$ by Lemma \ref{corehasinfiniteorderimpliesdifferentpowersnotequivalent}. Therefore we may assume that $|U| = |V|$.

Let $U= u_1 \ldots u_r$ and $V= v_1 \ldots v_r$ and let $1 \le i \le r$ be the minimal index so that $u_i \ne v_i$. If $i=r$ then we are done, since $U= Sq$ and $V= Sp$ (or vice versa) for some $S \in  \T{S}^{r-1}$ and $q\ne p \in \T{S}$. 

Therefore assume that $i \le r$ and that $U = SqT_1$ and $V = SpT_2$ for $S \in \T{S}^{i-1}$, $T_1,T_2 \in \T{S}^{r-i}$ and $q,p \in \T{S}$. If $U$ and $V$ are not $\omega$-equivalent, we are done. Therefore assume that $U$ and $V$ are $\omega$-equivalent. 

Since ${p,q} \in \T{S}(2)$, there is a path in $G_{j}(A)$ from $\{p,q\}$ to a vertex  which belongs to a circuit. Therefore we may assume that there is path in the graph $G_j(A)$ as follows :

\[
\{p,q\}:= \{p_0, q_0\} \to \{p_1, q_1\} \to \ldots \to \{q_{l}, p_{l}\} \to \{q_{l+1}, p_{l+1}\}
\]

where for each $\{p_a,q_a\}$, $0\le a \le l$ there is a split of length $m_a$ with top $\{p_a, q_a\}$ such that the bottom depends only on the top, and the bottom is $\{p_{a+1}, q_{a+1}\}$ and $\{p_{l+1}, q_{l+1}\}$ is a vertex on a circuit in $G_{j}(A)$. Notice that $m_a \ge 1$ for all $1 \le a \le l$. Therefore by travelling along this circuit in $G_{j}(A)$ as long as required, we may also assume that $m_0 + m_1 + \ldots + m_l +1 \ge r-i +1$.

By appending a common suffix to $U$ and $V$, thus preserving $\omega$-equivalence, if necessary we may further assume that $r-i+1 =|qT_1| = |pT_2|$ is  equal to $m_1 + m_1 \ldots + m_l +1$. Redefining $T_1$ and $T_2$ we assume that $U = S qT_1t_1$  and $V= SpT_2t_2$ where $|qT_1| = |pT_2| = m_0+ m_1 \ldots + m_l +1$ and $t_1 $ and $t_2$ are possibly distinct elements of $Q_A$.  Since $|qT_1| = |pT_2| = m_0 + m_1 \ldots + m_l +1$, write $qT_1 = R_1R_2 \ldots R_l$ and $pT_2 = P_1 P_2 \ldots P_l$ where $R_{a}, P_{a} \in Q_A^{m_a}$ for $1 \le a \le l$, moreover $R_1$ begins with $q$ and $P_1$ begins with $p$.

 Since $\{q,p\}$ is a vertex of $G_{j}(A)$, there is a word $\Gamma$ of length $j$ belonging to a split of length $m_0$, whose bottom depends only on the top $\{q,p\}$, and with bottom $\{q_1,p_1\}$.  Let $\Lambda$ be the word such that the output when processed through $A_{S}$ is $\Gamma$. Let $S_{\Gamma}$ be the state of $A^{m_0}$ such that $\pi_{Am_{0}}(\Gamma, P_1) = \pi_{Am_{0}}(\Gamma, Q_1)$. Such an $S_{\Gamma}$ exists  by definition of what it means for the bottom of a split to depend only on its top (see Definition \ref{bottomdependsonlyontop} ). Then we have, on reading $\Lambda$ through $A_{U}$ and $A_{V}$ respectively that we  transition to the states $S'S_{\Gamma}Q_1'Q_3' \ldots Q_l't_1'$ and $S'S_{\Gamma}P_1'P_3' \ldots P_l' t_2'$. Moreover $Q_1'$ begins with $q_1$ and $P_1'$ begins with $p_1$. Once more $Q_a' \in Q_A^{m_{a}}$ for $1 \le a \le l$.

Since, by assumption, each $\{p_a,q_a\}$ for $1 \le a \le l$ has an outgoing edge corresponding to a split of length $m_a$ whose bottom, $\{p_{a+1}, q_{a+1}\}$, depends only on its top, we can now repeat the argument of the above paragraph until the last letters of the final pair of state are a vertex of $G_r(A)$. Therefore we are in the situation that $i=r$ at which point we conclude that the final pair of states are not $\omega$-equivalent.

 Now since $U$ and $V$ are $\omega$-equivalent, then the final pair of states should also be $\omega$-equivalent, since we read the same word from $A_{U}$ and $A_{V}$ into this pair. This yields the desired contradiction. Therefore we conclude that $A_{U}$ and $A_{V}$ are not $\omega$-equivalent. 

The above now means that the semigroup $\gen{A_{p}| p \in \T{S}}$ satisfies no relations and so is a free semigroup. In fact this argument actually demonstrates that for any word $W \in Q^{\ast}$ ($Q$ being the set of states of $A$), the semigroup $\gen{A_{Wp}| p \in \T{S}}$ is a free semigroup. 
\end{proof}

\begin{corollary}
Let $A \in \widetilde{\T{H}}_{n}$ and suppose that $A$ is synchronizing at level $k$ and is minimal. As usual let $G_{j}(A)$ be the graph of bad pairs for some $j \ge k \in \N$. Suppose there is a subset $\T{S}$ of set of states of $A$, such that the following things hold:
\begin{enumerate}[label = (\roman*)]
\item $|\T{S}| \ge 2$,
\item The set, $\T{S}(2)$, of two element subsets of $\T{S}$ is a subset of the vertices of $G_{j}(A)$,
\item For each element of $\T{S}(2)$ there is a vertex  accessible from it which belongs to a circuit.
\end{enumerate}
Then the automaton semigroup generated by $A$ has exponential growth.
\end{corollary}

There are a few ways of extending the argument. One can also show that for a subset $\T{S} \subset Q$ satisfying the conditions of the proposition, and for any vertex on a path from a vertex of $G_{k}$ to a vertex accessible from $\T{S}(2) $, then the pair of states making up this vertex generate a free semigroup. Notice that if the graph of bad pairs has a circuit then the conditions of the proposition are satisfied.

\begin{Remark}
In proving Propositions \ref{looporfreesmgp} and \ref{infinite=expgrowth1} we have made use of the cancellative property of automata groups generated by elements of $\hn{n}$, in particular the above arguments can be extended to elements of $\spn{n}$ where we still retain this cancellative property. 
\end{Remark}

\begin{corollary}
Let $A \in \widetilde{\T{H}}_{n}$ then the automaton semigroup generated by $A$ contains a free semigroup of rank at least $2$.
\end{corollary}
\begin{proof}
This follows from Propositions \ref{looporfreesmgp} and  \ref{infinite=expgrowth1}.
\end{proof}

\begin{Theorem}\label{synchronizingmeansexpgrowth}
Let $A \in \widetilde{\T{H}}_{n}$ then the automaton semigroup generated by $A$ has exponential growth.
\end{Theorem}
\begin{proof}
This follows from standard results in the literature on the growth rates of groups and semigroups and the fact that the automaton semigroup generated by $A$ contains a free semigroup.
\end{proof}

\subsection{Further conditions for having infinite order: avoiding loops}

In this subsection we outline a method for detecting when an element of $\shn{n}$ has infinite order which does not depend on detecting loops. This turns out to be particularly effective when $n=3$. Our approach shall be to deduce implications on the local action of states of the transducer from a power of the dual transducer being a zero.

First we need the following notion.

Let $A \in \shn{n}$. For each letter $i \in X_n$ let $[i]:= \{ \pi_{A}(i, p) \mid p \in Q_A
\}$ and let $[i]^{-1}:=  \{ \pi_{A^{-1}}(i, p^{-1}) \mid p^{-1} \in Q_A^{-1}
\}$. Note that it is not necessarily the case that if $p \in [i]$ then $p^{-1} \in [i]^{-1}$. Let $\mathfrak{P}(A)_1 := \{[i] \mid i \in X_n\}$. Now refine  $\mathfrak{P}_{1}$ as follows: whenever $i, j \in X_n$ are such that  if $[i] \cap [j] \ne \emptyset$, then let $[i,j]:=  [i] \cup [j]$, let $\mathfrak{P}(A)_{2}$ be the result of this process. An element of $\mathfrak{P}(A)_{2}$ is either of the form $[i,j]$ for $i,j \in X_n$ or just $[i]$ for some $i \in X_n$. Repeat the process: whenever two elements of $\mathfrak{P}_{2}$ have non-empty intersection, we take their union, and let  $[i_1, i_2, \ldots, i_m]$ denote the resulting set, where the $i_l$'s are distinct for $1 \le l \le m$, $[i_l] \subset [i_1, i_2, \ldots, i_m]$ and $m$ is at most $4$. Recursively form sets  $\mathfrak{P}_{j}$ for $j \in \mathbb{N}$. Since $|X_n| = n$ there is a $j \in \N$ such that $\mathfrak{P}(A)_{j} = \mathfrak{P}(A)_{j+1}$. Let $\mathfrak{P}(A) = \mathfrak{P}(A)_{j}$ for this $j$. Notice that $\mathfrak{P}(A)$ is a partition of the states of $A$, and we call $\mathfrak{P}(A)$ the \emph{letter induced partition of $A$}.

\begin{lemma}\label{Lemma: at least 2 letters have the same set of states}
Let $A \in \hn{n}$ and let $\mathfrak{P}(A)$ be the letter induced partition of $A$, then there exists $P \in  \mathfrak{P}(A)$ and distinct letters $i$ and $j$ in $X_n$ such that $[i] \cup [j] \subset P$.
\end{lemma}
\begin{proof}
Let $A \in \hn{n}$. Since $A^{-1}$ is synchronizing, it follows (see the Appendix of \cite{BlkYMaisANav}) that there are distinct states $p^{-1}$ and $q^{-1}$ of $A^{-1}$ such that for all $l \in X_n$ we have $\pi_{A^{-1}}(l, p) = \pi_{A^{-1}}(l,q)$. Now since $A$ is minimal and synchronous, $A^{-1}$ is also minimal, therefore there  is an $i' \in X_n$ such that $i =\lambda_{A^{-1}}(i', p^{-1}) \ne \lambda_{A^{-1}}(i', q^{-1}) = j$. Hence in $A$ we have, $\pi(i, p) = \pi(j, q)$. It follows by definition that there is some $P \in \mathfrak{P}(A)$ such that $[i] \cup [j] \subset P$.
\end{proof}

\begin{lemma} \label{Lemma: letter induced partition bounds image size}
Let $A \in \hn{n}$ and let $\mathfrak{P}(A)$ be the letter induced partition of $A$. Let $k \in \N$ be greater than or equal to the synchronizing level of $A$. Suppose that $\dual{A}_{k}$ splits and that if $\dual{A}_{k+1}$ splits then it  has minimal splitting length strictly greater than the minimal splitting length of $(\dual{A})^{k}$. Then we have the following:

\begin{enumerate}[label = (\roman*)]
\item $\mathfrak{P}(A) \ne \{[0,1 \ldots, n-1] \}$ \label{lemma: letter partition not everything}
\item For any $\Gamma \in X_n^{k}$ the transformation  $\sigma_{\Gamma}$ has image size strictly less than $n$ \label{lemma: n gives bound on image of transformation}. In particular given $\Gamma \in X_n^{k}$ then for a given $P \in \mathfrak{P}(A)$ then there exists a $q \in Q$ such that for all  $t \in P$ we have $q=(t)\sigma_{\Gamma}$.
\end{enumerate}

\end{lemma}
\begin{proof}
Let $A \in \hn{n}$ be as in the statement of the lemma and let $l$ be the minimal splitting length of $(\dual{A})^{k}$. Let $m \in \N$ be minimal such that $\mathfrak{P}(A)_{m} = \mathfrak{P}(A)$.

For part \ref{lemma: letter partition not everything} let $\Gamma$ be a word in $X_n^{k}$ such that $\sigma_{\Gamma}$ has image size at least 2. Let  $q_1q_2 \ldots q_l \in Q_A^{l}$ be any $l$ tuple of states, then since $\dual{A}_{k}$ has minimal splitting length $l$ then the $l$'th letter of the output when $q_1q_2 \ldots q_l$ is processed from the state $\Gamma$ of $(\dual{A})^{k}$ depends only on $q_1$. In particular for $i \in X_n$, and for any word $p_1\ldots p_l$ in the state of $A$ the $l$'th letter of the output when $p_1 \ldots p_l$ is processed from $i\Gamma$ depends only on the state $\pi(i, p_1)$. In particular this letter is equal to $(\pi(i, p_1))\sigma_{\Gamma}$. However since the splitting length of  $(\dual{A})^{k+1}$, if it splits, is strictly greater than $l$ then it must be the case that  $(\pi(i, p_1))\sigma_{\Gamma} =(\pi(i, p))\sigma_{\Gamma}$ for any state $p$ of $Q_A$. 

Now if $j \in X_n$ is such that $[i] \cap [j] \ne \emptyset$ i.e 
$[i,j]$ is an element of  $\mathfrak{P}(A)_{2}$ then there are 
states $q_1$ and $q_2$ of $A$ such that  $\pi(i, q_1) = \pi(j, 
q_2)$. It therefore follows that $(\pi(i, q_1))\sigma_{\Gamma} = 
(\pi(j, q_2))\sigma_{\Gamma}$. By the previous paragraph we 
therefore have that for any $q$ in $[i]$ and $p \in [j]$, 
$(q)\sigma_{\Gamma} = (p)\sigma_{\Gamma}$. 

Now assume that for all $1 \le r < m$ for any set $P_1 \in \mathfrak{P}(A)_{r}$ and any pair of states $q_1$ and $q_2$ in $P_1$ we have that  $(\pi(i, q_1))\sigma_{\Gamma} = (\pi(j, q_2))\sigma_{\Gamma}$. Now let $P_1$ and $P_2$ in $\mathfrak{P}(A)_{r}$ such that $P_1 \cap P_2 \ne \emptyset$. This means that there is a pair $i, j \in X_n$ such that $[i] \subset P_1$ an $[j] \subset P_2$ such that $[i] \cap [j] \ne \emptyset$. Therefore by repeating the argument in the previous paragraph we have that $(\pi(i, q_1))\sigma_{\Gamma} = (\pi(j, q_2))\sigma_{\Gamma}$ for any pair of state $q_1 \in [i]$ and $q_2 \in [j]$. By the inductive assumption we therefore have that $(\pi(i, q_1))\sigma_{\Gamma} = (\pi(j, q_2))\sigma_{\Gamma}$ for any pair of states $q_1 \in P_1$ and $q_2 \in P_2$.

Now since $\sigma_{\Gamma}$ has image size at least 2, there are 
states $t_1$ and $t_2$ of $A$ such that $(t_1)\sigma_{\Gamma} \ne 
(t_2)\sigma_{\Gamma}$. Now as $A$ is core and synchronizing, 
there are elements $P_1$ and $P_2$  of $\mathfrak{P}(A)$ such 
that $t_1 \in P_1$ and $t_2 \in P_2$. Now by observations in the 
previous paragraph it follows that $P_1 \ne P_2$. This 
demonstrates \eqref{lemma: letter partition not everything}. 

The second part of the lemma now follows since as demonstrated above for $P \in \mathfrak{P}(A)$, there is a fixed $q \in Q_A$ such that $q =(t)\sigma_{\Gamma}$ for all $t \in P$. Therefore, by Lemma~\ref{Lemma: at least 2 letters have the same set of states}, we have  $ 2 \le |\im(\sigma_{\Gamma})| = |\mathfrak{P}(A)| < n$.  
\end{proof}

\begin{Remark}
Notice that if $A \in \hn{n}$ has infinite order, then there are 
infinitely many numbers $k \in \N$, where $k$ is greater than or 
equal to the minimal synchronizing level of $A$, such  
$(\dual{A})^{k}$ has splitting length strictly less than 
$(\dual{A})^{k+1}$. For each such $k$, any $\Gamma \in X_n^{k}$ 
and a given $P \in \mathfrak{P}(A)$ all elements of $P$ have the 
same image under $\sigma_{\Gamma}$. In particular since 
$(\dual{A})^{k}$ splits there elements $P_1$ and $P_2$ such that for $t_1 \in P_1$ and $t_2 \in P_2$ there is some $\Delta 
\in X_n^{k}$ such that $(t_1)\sigma_{\Delta} \ne 
(t_2)\sigma_{\Delta}$. Furthermore we may insist that there is 
an infinite subset $\T{J} \subset \N$ such that for all $j \in 
\T{J}$ there is a $\Delta \in X_n^{j}$ such that 
$(t_1)\sigma_{\Delta} \ne (t_2)\sigma_{\Delta}$ for $t_1 \in 
P_1$ and $t_2 \in P_2$. This follows  since there are infinitely 
many numbers $k \in \N$, where $k$ is greater than or equal to 
the minimal synchronizing level of $A$, such  $(\dual{A})^{k}$ has 
splitting length strictly less than $(\dual{A})^{k+1}$. Now by 
repeating, with slight modifications, the proof of Proposition~\ref{looporfreesmgp}, if no power of $\dual{A}$ has a loop then for any pair $t_1 \in P_1$ and $t_2 \in P_2 $ the semigroup generated by $t_1$ and $t_2$ is free.
\end{Remark}

\begin{Remark}
Notice that if $A \in \hn{3}$ has infinite order, then $\mathfrak{P}(A)$ has only two elements, $P_1$ and $P_2$. Moreover for all numbers $k \in \N$, where $k$ is greater than or equal to the minimal synchronizing level of $A$, and the splitting length of $\dual{A}_{k+1}$ is strictly greater than the splitting length of $\dual{A}_{k}$ we have that there is some $\Gamma \in X_n^{k}$ such that $(t_1) \sigma_{\Gamma} \ne (t_2)\sigma_{\Gamma}$ for any pair $t_1 \in P1$ and $t_2 \in P_2$.
\end{Remark}

Now consider the case that $n=3$. By Lemmas ~\ref{Lemma: at least 2 letters have the same set of states} and ~\ref{Lemma: letter induced partition bounds image size}, if an element $A \in \hn{n}$ has infinite order then $\mathfrak{P}(A)$ contains only the elements $[i_1, i_2]$ and $[i_3]$ where $\{i_1, i_2, i_3\} := X_3$. 

\begin{lemma}\label{Lemmma: finite order imposes conditions on local map of powers}
Let $A \in \hn{3}$. Let $i_1$, $i_2,$ and $i_3$ be distinct elements of $X_3$. Suppose that $[i_1] \cap [i_2] \ne \emptyset$. If for some $l \in \N$ there is an $i_a \in X_3$ and a pair of states $S_1$, $S_2$ of  $A^l= \gen{X_n, Q^{l}, \pi_{Al}, \lambda_{Al}}$  such that:
\begin{enumerate}[label = (\roman*)]
\item $\pi_{Al}(i_a, S_1) = \pi_{Al}(i_a, S_2)$,
\item $\lambda_{Al}(i_a, S_1) \in \{i_1, i_2\}$ and, \item $\lambda_{Al}(i_a, S_2) = i_3$ 
\end{enumerate}
then either $(\dual{A})^{k} = (\dual{A})^{k+1}$ where $k$ is the minimal synchronizing level of $A$, or $A$ has infinite order.
\end{lemma}
\begin{proof}
It suffices to show by Proposition~\ref{lemma 1.3} that if $A \in \hn{3}$ with minimal synchronizing level $k$, has finite order and satisfies the conditions of the lemma then $\dual{A}_{k+1} = \dual{A}_{k}$.

Therefore let $A \in \hn{3}$ be an element of finite order which satisfies the conditions of the lemma. Let $k \in \N$ be the minimal synchronizing level of $A$. Let $\mathfrak{P}_{A}$ be the letter induced partition of $A$. If $\dual{A}_{k+1} = \dual{A}_{k}$ we are done. Thus suppose that there is some $m >k$, $m\in \N$ such that $\dual{A}_{m+2} = \dual{A}_{m+1}$ but $\dual{A}_{m}$ splits. 

Since $\dual{A}_{m}$ splits it follows from  Lemma~\ref{Lemma: letter induced partition bounds image size} and the condition that $[i_1] \cap [i_2] \ne \emptyset$ that the letter induced partition of $A$ consists of the elements $[i_1, i_2]$ and $[i_3]$. Moreover there are states $q_1$ and $q_2$, and some $\Gamma \in X_n^{m}$ such that $q_1 = (t_1)\sigma_{\Gamma}$ and $q_2 = (t_2)\sigma_{\Gamma}$  for any pair $t_1 \in [i_1, i_2]$ and $t_2 \in [i_3]$.

Now let $S3 = \pi_{Al}(i_a, S_1) = \pi_{Al}(i_a, S_2)$, and let $\Delta \in X_n^{m}$ be such that $\lambda_{Al}(\Delta, S_3) = \Gamma$. Consider the word $i_a \Delta$ a state of $\dual{A}_{m+1}$. Now after processing the words $S_1$ and $S_2$ from the state $i_a \Delta$ of $\dual{A}_{m+1}$  the active states are $i_b \Gamma$ and $i_3 \Gamma$ $i_b \in \{i_1, i_2\}$. Now since $[i_b] \subset [i_1, i_2]$, it follows from the previous paragraph that for any input of length $l$ processed from the state $i_b \Gamma$ the $l$'th letter of the output must be $q_1$. Likewise for any input of length $l$ processed from the state $i_3 \Gamma$ the $l$'th letter of the output must be $q_2$. Therefore we see that $\dual{A}_{m+1}$ splits also which is a contradiction.
\end{proof}

As a corollary we have:

\begin{corollary}
Let $A \in \hn{3}$ and suppose that $\mathfrak{P}(A) \ne \{[0,1,2]\}$.  Let $i_1$, $i_2$ be distinct elements of $X_3$ such that $[i_1, i_2]$ is in $\mathfrak{P}(A)$, then if $[i_1]^{-1} \cap [i_2] ^{-1} = \emptyset$ either $\dual{A}_{k+1} = \dual{A}_{k}$ where $k$ is the minimal synchronizing level of $A$ or $A$ has infinite order.
\end{corollary}
\begin{proof}
Let $A \in \hn{3}$ satisfy the conditions of the lemma and let $k$ be the minimal synchronizing level of $A$. Furthermore assume that $\dual{A}_{k}$ splits and so $\dual{A}_{k+1} \ne \dual{A}_{k}$. 

Since $A$ is synchronizing it follows that there are states $q_1$ and $q_2$ of $A$ such that for all $i \in X_n$ $\pi_{A}(i, q_1) = \pi_{A}(i, q_2)$ (see the appendix of \cite{BlkYMaisANav}). Since $A$ is minimal there is a $j \in X_n$ such that $\lambda(j, q_1) \ne \lambda(j, q_2)$. 

Now the condition that $[i_1]^{-1} \cap [i_2] ^{-1} = \emptyset$ implies by Lemma~\ref{Lemma: letter induced partition bounds image size} that, since $A$ $\dual{A}_{k}$ splits, either  $[i_1]^{-1} \cap [i_3]^{-1} \ne \emptyset$ or $[i_2]^{-1} \cap [i_3]^{-1} \ne \emptyset$. We assume by relabelling if necessary that $[i_1]^{-1} \cap [i_3]^{-1} \ne \emptyset$. This means that, since $\dual{A}$ splits,  $\mathfrak{P}(A^{-1})$ consists of the elements $[i_1, i_3]^{-1}:= [i_1]^{-1} \cup [i_3]^{-1}$ and $[i_2]^{-1}$. Hence we have that $\lambda(j, q_1) = i_1$ and $\lambda(j, q_2) = i_3$ or $\lambda(j, q_1) = i_3 $ and $\lambda(j, q_2) = i_1$. In either case we have that $A$ satisfies the conditions of Lemma~\ref{Lemmma: finite order imposes conditions on local map of powers} and we are done.

\end{proof}

\begin{Remark}
The above corollary implies that if $A \in \hn{3}$ is such that  $\mathfrak{P}(A) = \{[i_1, i_2], [i_3]\}$ for $\{i_1, i_2, i_3\} = X_n$, then either $\mathfrak{P}(A^{-1}) = \{[i_1, i_2]^{-1}, [i_3]^{-1}\}$ where $[i_1, i_2]^{-1} := [i_1]^{-1} \cup [i_2]^{-1}$ or $\dual{A}_{k}$ does not split.
\end{Remark}

We conclude the section with the following lemma:

\begin{lemma}
Let $A \in \hn{3}$. Assume that $\mathfrak{P}(A) = \{[i_1, i_2], [i_3]\}$ and $\mathfrak{P}(A^{-1}) = \{[i_1, i_2]^{-1}, [i_3]^{-1}\}$. Let $q_1$ and $q_2$ be distinct states of $A$  such that for all $i \in X_n$, $\pi_A(i, q_1) = \pi_{A}(i, q_2)$ and $\{q_1, q_2\}$ is a subset of some $P \in \mathfrak{P}(A)$.  If there are (not necessarily distinct) states $p_1$, $p_2$ of $A$ and (not necessarily distinct letters $j_1$ and $j_2$ in $X_n$ such that $\pi(j_1, p_1) = q_1$, $\pi(j_2, p_2) = q_2$, $\lambda(j_1, p_1) \in \{i_1, i_2\}$, and $\lambda(j_2, p_2) =  i_3$, then there is a conjugate $B$ of $A$ such that $|Q_B| < |Q_A|$.
\end{lemma}
\begin{proof}
Let $A \in \hn{3}$ satisfy the conditions of the lemma. Observe that the condition $\mathfrak{P}(A) = \{[i_1, i_2], [i_3]\}$ and $\mathfrak{P}(A^{-1}) = \{[i_1, i_2]^{-1}, [i_3]^{-1}\}$ implies that whenever a state $q$ of $A$ is such that there is some state $p$ of $A$ and an $i \in X_n$ with, $\pi_{A}(i, p) = q$ and $\lambda_{A}(i, p) = i_3$, then for any other state $p'$ and any letter $i'$ such that $\pi_{A}(i', p') = q$ we must have that $\lambda_{A}(i', p') = i_3$. Thus if $i \in X_n$ is such that $\lambda_{A}(i, q_1) = i_3$ then $\lambda_{A}(i, q_2) = i_3$. Therefore if $\overline{q_{2}^{-1}q_1}$ is the permutation of  $X_n$ induced by the state $q_{2}^{-1}q_1$ of $A^{-1}A$,  $\overline{q_{2}^{-1}q_1}$ fixes $i_3$. This is because if $i = \lambda_{A^{-1}}(i_3, q_2^{-1})$ then $\lambda(i, q_1) = i_3$. Likewise let $\overline{q_{1}^{-1}q_2}$ be the permutation of  $X_n$ induced by the state $q_{1}^{-1}q_2$ of $A^{-1}A$,  $\overline{q_{2}^{-1}q_1}$ and this also fixes $i_3$ by a similar argument. Moreover for any state $t$ of $Q_A$ such that there is some $j \in X_n$ and $\pi_{A}(j, t) = q_1$, we must also have that $\lambda_{A}(j, t) \in \{i_1, i_2\}$. Now since $q_1$ and $q_2$ are states of $A$ then $j_1$ and $j_2$ are either equal, or $\{j_1, j_2\} = \{i_1, i_2\}$, by an abuse of notation write $[j_1, j_2]$ for element of $P$ containing $q_1$ and $q_2$.

Let  $C = \gen{X_3, Q_C, \pi_C, \lambda_C}$ where $Q_C:= \{c_1, c_2\}$ be defined as follows. $\pi_{C}(i, \centerdot): Q_C \to \{c_1\}$ if $i \in \{j_1, j_2\}$ otherwise  $\pi_C(i, \centerdot): Q_C \to \{c_2\}$. The map $\lambda_{C}(\centerdot, c_1):  X_3 \to X_3$ is the identity permutation. Set the map $\lambda_{C}(\centerdot, c_2):  X_3 \to X_3$ to be the permutation $\overline{q_{2}^{-1}q_1}$ if $q_1, q_2 \in [i_1, i_2]$ otherwise  set $\lambda_{C}(\centerdot, c_2):  X_3 \to X_3$ to be the permutation  $\overline{q_{1}^{-1}q_2}$. Notice that  both $c_1$ and $c_2$  map $i_3$ to $i_3$. Moreover since $q_1$ and $q_2$ are distinct states of $A$ and $A$ is a minimal transducer we also have that  the state $c_2$ induces the transposition swapping $i_2$ and $i_2$. Therefore $C$ is a minimal transducer. Furthermore since whenever we read $1_1$ and $i_2$ the active state is $c_1$ and the output is an element of the set $\{i_1, i_2\}$ we also have that $C$ is bi-synchronizing at level 1 and has order 2.

Now consider $\core(C A C)$. Since $A$ is synchronizing it follows that $\core{(C A C)}$ is synchronizing. Let $k$ be greater than maximum of the  minimal synchronizing length of $\core(CAC)$ and the minimal synchronizing length of $A$. Using the conditions that $\pi(j_1, p_1) = q_1$ and $\lambda(j_1, p_1) \in \{i_1, i_2\}$, there is a string $\Gamma \in X_n^{k}$ with  last letter equal to $j_1$ such that the state of $A$ forced by $\Gamma$ is $q_1$. This means by an observation in the first paragraph that the output of $\Gamma$ when processed from any state has last letter in the set $\{i_1, i_2\}$. Likewise there is a word $\Delta \in X_n^{k}$ with last letter $j_2$ such that the state of $A$ forced by $\Delta$ is $q_2$ and the output of $\Delta$ when processed from any state has last letter equal to $i_3$. Now since $q_1$ and $q_2$ belong to the same element $[j_1, j_2]$ of  $\mathfrak{P}(A)$, it follows that any word of length $k$ in $W_{q_1}$ or $W_{q_2}$ must have last letter in the set $\{j_1, j_2\}$.

Now all states of $C$ map $\{j_1, j_2\}$ to the set $\{j_1, j_2\}$ (since they all fix $i_3$), therefore for any word $\Lambda \in X_n^{k}$ and any state $c$ of $C$ such that $\lambda_{C}(\Lambda, c) \in W_{q_1} $ we must that the state of $C$ forced by $\Lambda$ is $c_1$ and the last letter of $\Lambda$ is in the set $\{j_1, j_2\}$. Therefore reading such a word $\Lambda$ from any state of $CAC$ beginning with $c$ the active state will be  $(c_1, q_1, c_{[i_1, i_2]})$, where $c_{[i_1, i_2]} = c_1$ if $\{j_1, j_2\} = \{i_1, i_2\}$  otherwise $c_{[i_1, i_2]} = c_2$.  This is because by an observation in the first paragraph all single letter inputs to the state $q_1$ have output in the set $\{i_1, i_2\}$. Therefore $(c_1, q_1, c_{[i_1, i_2]})$ is a state of $\core(CA C)$.  Likewise for any word $\Lambda' \in X_n^{k}$ and any state $c'$ of $C$ such that $\lambda_{C}(\Lambda', c') \in W_{q_2} $ we must have that the state of $C$ forced by $\Lambda'$ is $c_1$ and the last letter of $\Lambda'$ is in the set $\{j_1, j_2\}$. Therefore reading such a word $\Lambda'$ from any state of $CAC$ beginning with $c'$ the active state will be  $(c_1, q_2, c_{[i_3]})$, where $c_{[i_3]} = c_1$ if $\{j_1, j_2\} = \{i_3\}$  otherwise $c_{[i_3]} = c_2$.  This is because by an observation in the first paragraph all single letter inputs to the state $q_2$ have output equal to $i_3$. Therefore $(c_1, q_2, c_{[i_3]})$ is also a state in $\core(CAC)$.

Now the above arguments are actually independent of $q_1$ and $q_2$ and demonstrate that if  $d q d'$ is a state of $\core(CAC)$ then $d$ depends only the set $S$ of $\mathfrak{P}(A)$ such that $q \in S$ and $d'$ depends only on the set $S'$ of $\mathfrak{P}(A^{-1})$ such that $q^{-1} \in S'$.  Therefore $\core(CA C)$ has as many states as $A$. 

We now demonstrate that $(c_1, q_2, c_{[i_3]})$ and $(c_1, q_1, c_{[i_1, i_2]})$ are $\omega$-equivalent. Since $C$ is synchronizing at level  1, since both states of $CAC$ begin with $c_1$, since $q_1$ and $q_2$ satisfy  $\pi_A(i, q_1) = \pi_{A}(i, q_2)$ for all $i \in X_n$, and since all states of $C$ read $i_1$ and $i_2$ to the same location, it follows that for any word $i \in X_n$ we have $\pi_{CAC}(i,(c_1, q_2, c_{[i_3]})) = \pi_{CAC}(i,(c_1, q_2, c_{[i_1, i_2]}))$. This is because for any $i \in X_n$, $\{\lambda_{A}(i, q_1), \lambda_{A}(i, q_2) \} = \{i_1, i_2\}$ or $\{\lambda_{A}(i, q_1), \lambda_{A}(i, q_2) \} = \{i_3\}$. Thus, it suffices to show that $(c_1, q_2, c_{[i_3]})$ and $(c_1, q_1, c_{[i_1, i_2]})$ induce the same permutation on $X_n$. However this follows by construction, since if $\{j_1, j_2\} = \{i_1, i_2\}$ we have that $c_{[i_3]} = c_2$, the permutation of $X_n$ induced by $c_2$ is $\overline{q_2^{-1}q_1}$ and $c_\{[i_1, i_2]\} = c_1$ (recall $c_1$ induces the identity permutation on $X_n$). Therefore the permutation of $X_n$ induced by the states $(c_1, q_2, c_{[i_3]})$ and $(c_1, q_1, c_{[i_1, i_2]})$ coincide and is equal to $\overline{q_1}$. On the other hand if  $\{j_1, j_2\} = \{i_3\}$ we have that $c_{[i_3]} = c_1$, $c_{[i_1, i_2]} = c_2$ and the permutation of $X_n$ induced by $c_2$ is equal to $\overline{q_1^{-1}q_2}$. Therefore the permutation of $X_n$ induced by the states $(c_1, q_2, c_{[i_3]})$ and $(c_1, q_1, c_{[i_1, i_2]})$ coincide and is equal to $q_2$. Therefore setting $B$ to be the minimal transducer representing $\core(CAC)$ we see that $B \in \hn{3}$ is a conjugate of $A$ with  $|Q_A| - |Q_B| \ge 1$.
\end{proof}

\subsection{The growth rate of Cayley machines}
In this section we show that for a finite group $G$, the automaton semigroup  generated by the Cayley machine, $\T{C}(G)$ has growth rate, $|G|^n$. To this end, we begin by describing the construction of the Cayley machine.

Let $M$ be a finite monoid (e.g. a finite group), then one can form the automaton $\T{C}(M) := \gen{M, M,\pi, \lambda}$ called its \emph{Cayley machine}, with input and output alphabet, $M$ and state set $M$. The transition and rewrite function satisfy the following rules for $l,m \in M$:

\begin{enumerate}[label = (\arabic*.)]
\item $\pi(l,m) := ml$
\item $\lambda(l,m) := ml$
\end{enumerate}

In each case $ml$ is the evaluation of the product of $m$ and $l$ in the monoid $M$. If $M$ is a finite group $G$  then by Cayley's Theorem no two states of $\T{C}(G)$ are $\omega$-equivalent, and the functions $\pi(\centerdot,m): M \to M$ and $\lambda(\centerdot,m): M \to M$ are bijections. Hence $\T{C}(G)$ is reduced and invertible. It is not hard to see that $(\T{C}(G))^{-1}$ is synchronizing at level 1 (or is a reset automaton).

\begin{Remark}\label{smpgenbycayleymacineisfree}
With a little work it can be shown that $(\T{C}(G))^{-1}$ satisfies the conditions of Proposition \ref{infinite=expgrowth1} where, in this case, $\T{S} = G$. This shows that the automaton semigroup generated by $\T{C}(G)$ is free. Silva and Steinberg give a proof of this in \cite{SilvaSteinberg}.
\end{Remark}

We have the following lemma for synchronizing transducers:

\begin{lemma}\label{core=power}
Let $A = \gen{X_n,Q,\pi,\lambda} \in \widetilde{\T{P}}_{n}$ be a transducer, which is synchronizing at level $k$. Furthermore assume that for every $\Gamma \in X_n^k$ and for all states $q \in Q$, there is a state $p \in Q$ such that $\lambda(\Gamma,p) \in W_{q}$. Then under this condition, $A$ has the property that for all $m \in \mathbb{N}$, $Core(A^m) = A^m$.
\end{lemma}

\begin{proof}
We may assume, by increasing the alphabet size, that $A$ is synchronizing at level $1$.

We proceed by induction on $m$. For $m=1$ it holds that $A = Core(A)$ by assumption that $A \in \widetilde{\T{P}}_{n}$.

Assume $Core(A^j) = A^j$ for all $j \le m-1$.

Consider $A^{m-1} = Core(A^{m-1})$. Fix an arbitrary state $b_1\ldots b_{m-1} \in A^{m-1}$. There is a state $a_1\ldots a_{m-2}a_{m-1}$ and letters $x$ and $y$ in $X_n$ such that 
\[
a_1\ldots a_{m-2}a_{m-1} \stackrel{x|y}{\longrightarrow} b_1\ldots b_{m-1} 
\] 

Let $y'$ be the output when $x$ is read from $a_1\ldots a_{m-2}$. Notice that since $A$ is synchronizing at level 1 the state of $A$ forced by $y'$ must be $b_{m-1}$. By assumption, for every state $q \in Q$ there is a state $p$ such that $\lambda(y',p) \in W_{q}$. Therefore given an arbitrary $q \in Q$, by setting $a_{m-1} := p$ we may assume that $y \in W_{q}$  moreover, the inductive hypothesis guarantees that $a_1\ldots a_{m-1}$ is a state of $Core(A^{m-1})$. 

Observe that $A^{m-1}$ is synchronizing at level $m-1$ and so there is a word $\Delta$ of length $m-1$ labelling a loop based at $a_1\ldots a_{m-2}a_{m-1}$. Let $\Lambda$ be the output of this loop. Then reading $\Delta x$ in $(A)^{m-1}$ from the state $a_1\ldots a_{m-2}a_{m-1}$ the output is $\Lambda y$. Now, the state of $A$ forced by $\Lambda y$ is $q$, therefore reading $\Delta x$ through any state $a_1\ldots a_{m-2}a_{m-1}s$ for any $s \in Q$, the active state becomes $b_1\ldots b_{m-1}q$.

The above paragraph now implies that $b_1,\ldots b_{m-1}q$ is a state of $Core(A^m)$, since $A^m$ is synchronizing at level $m$, hence the state of $A^m$ forced by $\Delta x$ is $b_1,\ldots b_{m-1}q$. Therefore for any $q \in Q$, $b_1\ldots b_{m-1}q$ is a state of $Core(A^m)$. Moreover $b_1\ldots b_{m-1}$ was arbitrary, so we conclude that $Core(A^m) = A^m$ as required.
\end{proof}

In our next result we apply Lemma \ref{core=power} to the transducer $(\T{C}(G))^{-1}$ for a finite group $G$ by showing that $(\T{C}(G))^{-1}$ satisfies the condition of the lemma.

\begin{Theorem}
Let $G$ be a finite group, then $|(\T{C}(G))^{n}| = |G|^n$, hence the automaton $\T{C}{G}$ has growth rate $|G|^{n}$. Moreover every state of $\T{C}(G)^{n}$ is accessible from every other state.
\end{Theorem}
\begin{proof}
Since, either by Remark \ref{smpgenbycayleymacineisfree} or a result in \cite{SilvaSteinberg}, the automaton semigroup generated by $\T{C}(G)$ is free it suffices to show that $(\T{C}(G)^{-1})$ satisfies the conditions of Lemma \ref{core=power}. 

Since the states of $(\T{C}(G))^{-1}$ are in bijective correspondence with the states of $\T{C}(G)$ we shall let $g'$ be the state of $(\T{C}(G))^{-1}$ corresponding to the state $g$ of $\T{C}(G)$.

Let $g,h \in G$. We shall show that there is a state $m'$ of $(\T{C}(G))^{-1}$ such that $\lambda_{G}'(g,m') = h$ (here $\lambda_{G}'$ represents the rewrite function of $(\T{C}(G))^{-1}$).

By definition of $\T{C}(G)$ it suffices to take $m' = (gh^{-1})'$.
\end{proof}

\end{document}